[1994/12/01]% LaTeX date must December 1994 or later
\documentclass[draft]{article}
\title{Estimates on fractional higher derivatives of\\weak solutions for 
the Navier-Stokes equations}
\author{Kyudong Choi\thanks{University of Texas at Austin},
Alexis F. Vasseur\thanks{University of Oxford}}
\date{\today }

\usepackage{amsmath,amsthm,amsfonts,amsopn,verbatim}

\chardef\bslash=`\\ % p. 424, TeXbook
\newtheorem{thm}{Theorem}[section]
\newtheorem{cor}[thm]{Corollary}
\newtheorem{lem}[thm]{Lemma}
\newtheorem{prop}[thm]{Proposition}

\newtheorem*{thm*}{Theorem}
\newtheorem*{lem*}{Lemma}
\newtheorem*{prop*}{Proposition}
\theoremstyle{definition}
\newtheorem{defn}{Definition}[section]
\newtheorem*{ack}{Acknowledgment:}

\theoremstyle{remark}
\newtheorem{rem}{Remark}[section]

\usepackage{amscd,amsthm}
\usepackage{}

\DeclareMathOperator{\ebdiv}{div}
\newcommand{\eval}[2][\right]{\relax
  \ifx#1\right\relax \left.\fi#2#1\rvert}

\begin{document}
\maketitle
\markboth{Fractional higher derivatives of weak solutions for 
Navier-Stokes}{Fractional higher derivatives of weak solutions for 
Navier-Stokes}
\renewcommand{\sectionmark}[1]{}
\begin{abstract}
We study weak solutions of the 3D Navier-Stokes equations in whole space
with $L^2$ initial data. 
It will be proved that 
$\nabla^\alpha u $ is locally integrable
in space-time
for any real $\alpha$ such that 
$1< \alpha <3$, which says that almost third derivative is locally integrable.
Up to now, only second derivative $\nabla^2 u$ has been known to 
be locally integrable by standard parabolic regularization.
We also present sharp estimates of 
those quantities in weak-$L_{loc}^{4/(\alpha+1)}$.
These estimates depend only
on the $L^2$ norm of initial data and  integrating domains. 
Moreover, they are valid even for $\alpha\geq 3$ as
 long as $u$ is smooth.
The proof uses a good approximation of Navier-Stokes and a blow-up technique, which let us to focusing on a local study.
For the local study, we use  De Giorgi method with a new pressure decomposition.
To handle non-locality of the fractional Laplacian, we will adopt
some properties of the Hardy space and Maximal functions.
\end{abstract}

\section{Introduction and main result}

\qquad In this paper, any derivative signs ($\nabla,\Delta,(-\Delta)^{\alpha/2},D,\partial$ and etc) denote derivatives in only space variable 
$x\in \mathbb{R}^3$  unless time 
variable $t\in\mathbb{R}$
%_{\geq0}$ 
is clearly specified.
We study the 3-D Navier-Stokes equations % in $\mathbb{R}^3$. 
\begin{equation}\label{navier}
\begin{split}
\partial_tu+(u\cdot\nabla)u+ \nabla P -\Delta u&=0 \quad\mbox{and}\\
\ebdiv u&=0,\quad \quad t\in ( 0,\infty ),\quad x\in \mathbb{R}^3
\end{split}
\end{equation} with $L^2$ initial data
\begin{equation}\label{initial_condition}
u_0\in L^2(\mathbb{R}^3) ,\quad \ebdiv u_0= 0.% \quad x\in \mathbb{R}^3. 
\end{equation}

%Before presenting the idea of proof, here are historic remarks. 
Regularity of weak solutions for the 3D Navier-Stokes equations has long history.
%Long times ago, 
Leray \cite{leray} 1930s and Hopf \cite{hopf} 1950s
proved existence of a global-time 
weak solution for any given $L^2$ initial data. 
Such Leray-Hopf weak solutions $u$ lie 
in $L^\infty(0,\infty;L^2(\mathbb{R}^3))$ and 
$\nabla u$ do in $L^2(0,\infty;L^2(\mathbb{R}^3))$ and
satisfy the energy inequality:
\begin{equation*}%\label{energy_eq}
\|u(t)\|_{L^2(\mathbb{R}^3)}^2+
2\|\nabla u\|^2_{L^2(0,t;L^2(\mathbb{R}^3))}
\leq \|u_0\|_{L^2(\mathbb{R}^3)}^2 \quad \mbox{ for a.e. } t<\infty.
\end{equation*}
%Simple interpolation gives $u\in L^{10/3}((0,\infty)\times\mathbb{R}^3))$
Until now, regularity and uniqueness of such weak solutions are 
generally open.\\% while those of local-time smooth solution has been shown 
%with smooth initial values.\\%still an open problem.\\

Instead, many criteria which 
%makes weak solutions become regular 
ensure regularity of weak solutions
have been developed. Among
them the most famous one is Lady{\v{z}}enskaja-Prodi-Serrin Criteria 
(\cite{lady},\cite{prodi} and \cite{Serrin}),
which says:
if $u\in L^p((0,T);L^q(\mathbb{R}^3)) $ for 
some $p$ and $q$ satisfying $\frac{2}{p}
+\frac{3}{q}=1$ and $p<\infty$, then it is regular. Recently, the limit case $p=\infty$ 
was established in the paper of 
Escauriaza, Ser{\"e}gin and {\v{S}}ver{\'a}k
\cite{seregin}. We may impose similar conditions to derivatives of velocity, vorticity
or pressure. (see
Beale,  Kato and Majda \cite{bkm}, Beir{\~a}o da Veiga \cite{beirao} and 
Berselli and  Galdi \cite{berselli}) Also, many other conditions
 exist (e.g. see
Cheskidov and Shvydkoy \cite{Cheski},
Chan \cite{chan}
 and 
 \cite{vas_bjorn}).  \\

On the other hand, many efforts have been given to measuring the size 
of possible singular set.
This approach has been initiated by 
Scheffer \cite{scheffer}. Then, Caffarelli, Kohn and Nirenberg \cite{ckn}
improved the result and showed that possible singular sets have zero Hausdorff measure of one dimension for
 certain class of weak solutions (suitable weak
 solutions)
satisfying the following additional inequality 
\begin{equation}\label{suitable}
\partial_t \frac{|u|^2}{2} + \ebdiv (u \frac{|u|^2}{2}) + \ebdiv (u P) +
 |\nabla u |^2 - \Delta\frac{|u|^2}{2} \leq 0
\end{equation} in the sense of distribution. 
There are many other proofs of this fact (e.g. see Lin 
\cite{lin}, \cite{vas:partial} 
and Wolf \cite{wolf}). Similar criteria for interior points with other quantities 
can be found in many places (e.g. see Struwe \cite{struwe},
Gustafson,  Kang and  Tsai \cite{gusta}, Ser{\"e}gin \cite{seregin2}
and Chae, Kang and Lee \cite{chae}). 
Also,   Robinson and  Sadowski \cite{robinson} and  Kukavica \cite{kukavica}
studied 
%that
  box-counting dimensions
  of 
singular sets.\\

In this paper, our main concern is about space-time $L^p_{(t,x)}=L^p_tL^p_x$ 
estimates of higher derivatives for weak 
solutions assuming only $L^2$
initial data. 
$\nabla u\in L^{2}((0,\infty)\times\mathbb{R}^3)$ is obvious from the energy
inequality, and
simple interpolation gives $u\in L^{10/3}$. %((0,\infty)\times\mathbb{R}^3)$.
For second derivatives of weak solutions, 
from standard parabolic regularization theory (see 
Lady{\v{z}}enskaja,  Solonnikov and 
 Ural$'$ceva 
 \cite{LSU}),
we know $\nabla^2 u \in L^{5/4}$  by   considering
$(u\cdot\nabla)u$ as a source term. 
With  different ideas, Constantin \cite{peter}   showed 
$L^{\frac{4}{3}-\epsilon}$ for any small $\epsilon>0$ in periodic setting,
%. Following same spirit of Constantin, 
and later Lions \cite{lions} improved it up to 
weak-$L^{\frac{4}{3}}$  (or  $ L^{\frac{4}{3},\infty}$)
%in whole space-time 
by assuming $\nabla u_0$ lying in the space % $\mathbb{M}$
of all bounded measures in $\mathbb{R}^3$. 
They used natural structure of the equation with some interpolation technique.
On the other hand, Foia{\c{s}}, Guillop{\'e} and Temam \cite{foias_guillope_tem:higher} and Duff \cite{duff}
obtained other kinds of estimates for higher derivatives 
of weak solutions while  Giga and  Sawada \cite{giga} 
and Dong and Du \cite{dong} covered mild solutions. For asymptotic behavior,
we refer Schonbek and  Wiegner \cite{Scho_and_Wiegner}.\\

Recently in \cite{vas:higher}, it has been shown that, for any small
$\epsilon > 0$, any integer $d\geq1$ and any smooth solution $u$ on $(0,T)$,  
we have bounds of $\nabla^d u$ in 
$L_{loc}^{\frac{4}{d+1}-\epsilon}$,
%((0,T)\times \mathbb{R}^3)$, 
which
 depend only on $L^2$ norm of initial data
once we fix $\epsilon$, $d$ and the domain of integration. 
It can be considered as a 
natural extension of the  result of Constantin \cite{peter} for higher derivatives.
But the idea is completely different in the sense that 
\cite{vas:higher} used the Galilean invariance of transport part of 
the equation and the partial regularity criterion in the version 
of \cite{vas:partial}, which re-proved
the famous result of Caffarelli,  Kohn and  Nirenberg \cite{ckn} by using a parabolic
 version of the De Giorgi method \cite{De_Giorgi}. 
It is noteworthy that this method
 gave full regularity to
the critical Surface Quasi-Geostrophic equation in
 \cite{caf_vas}.
The limit 
non-linear scaling 
$p=\frac{4}{d+1}$ 
%is  important because it has 
appears from
the following invariance of the Navier-Stokes scaling
$u_\lambda(t,x)=
\lambda u(\lambda^2 t,\lambda x)$:
\begin{equation}\label{best_scaling} 
\|\nabla^d u_\lambda\|^p_{L^p}=\lambda^{-1}\|\nabla^d u\|^p_{L^p}.
\end{equation}
%smooth solutions
%of \eqref{navier} on $(0,T)\times \mathbb{R}^3$ with \eqref{initial_condition}
%atisfy $\nabla^d u\in L_{loc}^{\frac{4}{d+1}-\epsilon}((0,T)\times \mathbb{R}^3)$
%%locally 
%and its norm is bounded by $L^2$ norm of initial data.
%  . From now on,

In this paper, our main result is better than
the above result of \cite{vas:higher} in the sense of the following three directions. 
First, we achieve the limit case weak-$L^{\frac{4}{d+1}}$
(or  $ L^{\frac{4}{d+1},\infty}$) as Lions  \cite{lions} did for second derivatives.
Second, we make similar bounds for fractional derivatives as well as classical derivatives.
Last, we consider not only smooth solutions but also global-time weak solutions.
These three improvements will give us that
$\nabla^{3-\epsilon}u$, which is almost third derivatives  of weak solutions, is locally integrable
 on $(0,\infty)\times\mathbb{R}^3$. \\
 %even beyond a possible blow-up time as well.\\

%Here is our main theorem.
Our precise result is the following:
\begin{thm}\label{main_thm}
There exist universal constants  
$C_{d,\alpha}$ which depend only on integer $d\geq1$ and real $\alpha\in[0,2)$
with the following two properties $(I)$ and $(II)$:\\

%For any integer $d\geq1$ and for any real $\alpha$ with $0\leq\alpha<2$,  
%we have $C_{d,\alpha}$ with following properties:\\

(I) 
Suppose that we have
a smooth solution  $ u $  of \eqref{navier} 
 on  $ (0,T)\times \mathbb{R}^3 $ for some $0<T\leq\infty$
% (with possible blow-up at  $ 0  $ and $  T $ )
 with some initial data \eqref{initial_condition}. Then it satisfies
%(with possible blow-up at  $ 0  $ and $  T $),
 \begin{equation}\label{main_thm_eq}
\quad\|(-\Delta)^{\frac{\alpha}{2}}\nabla^d u\|
_{L^{p,\infty}(t_0,T;L^{p,\infty}(K))}
 \leq C_{d,\alpha}\Big(\|u_0\|_{L^2(\mathbb{R}^3)}^{2} 
+ \frac{|K|}{t_0}\Big)^{\frac{1}{p}}
\end{equation}

 for  any $t_0\in(0,T)$, any integer $d\geq 1$, any $\alpha\in[0,2)$
 % with $0\leq\alpha<2$ 
and any bounded open subset
 $K$ of $ \mathbb{R}^3$, where $p = \frac{4}{d+\alpha+1}$
 and $|\cdot|=$ the Lebesgue measure  in $\mathbb{R}^3$. \\

(II) For any initial data \eqref{initial_condition}, we can construct a
suitable weak solution $u$ of \eqref{navier} on  $ (0,\infty)\times \mathbb{R}^3 $ such that
% with the followings:\\ 
%there exists
 $(-\Delta)^{\frac{\alpha}{2}}\nabla^d u$ 
is locally integrable in  $(0,\infty)\times \mathbb{R}^3$
 %lies in $ L^1_{loc}((0,\infty)\times \mathbb{R}^3)$
 for $d=1,2$ and for $\alpha\in[0,2)$ with $(d+\alpha)<3$.
Moreover, 
the  estimate \eqref{main_thm_eq} holds
with $T=\infty$
under the same setting of the above part $(I)$
as long as $(d+\alpha)<3$. % and  for any $t_0\in(0,\infty)$.
\end{thm}
Let us begin with some simple remarks.
\begin{rem}\label{frac_rem} 
For any suitable weak solution $u$, 
we can define $(-\Delta)^{\alpha/2}\nabla^d u$ 
in the sense of distributions $\mathcal{D}^\prime$ for any integer $d\geq0$ and for any real
$\alpha\in[0,2)$:% for any $f\in L^2(\mathbb{R}^3)$:
\begin{equation}\label{fractional_distribution}
 <(-\Delta)^{\alpha/2}\nabla^d u;\psi >_{\mathcal{D}^\prime,\mathcal{D}}
 =  (-1)^{d}\int_{(0,\infty)\times\mathbb{R}^3}u\cdot(-\Delta)^{\alpha/2}\nabla^d \psi\mbox{ } dxdt
 \end{equation} for any $\psi\in \mathcal{D}=C_c^\infty((0,\infty)\times\mathbb{R}^3)$
where 
$(-\Delta)^{\alpha/2}$ in the right hand side is 
the traditional 
fractional Laplacian  in $\mathbb{R}^3$ defined by the Fourier transform.
Note that %It is easily verified that
$(-\Delta)^{\alpha/2}\nabla^d \psi$ lies in $L_t^\infty L^2_x$. 
Thus, this definition from \eqref{fractional_distribution} makes sense due to
$u\in L_t^\infty L^2_x$. Note also $(-\Delta)^{0}=Id$. For  more general extensions
of this fractional Laplacian operator, we 
 recommend
 Silvestre \cite{silve:fractional}.
\end{rem}

\begin{rem}\label{rmk_dissipation of energy}
Since we impose only \eqref{initial_condition} to $u_0$, the estimate \eqref{main_thm_eq}
is a
% sort of 
 (quantitative)
%  version of
   regularization result to higher derivatives.
Also, in the proof, we will see that $\|u_0\|_{L^2(\mathbb{R}^3)}^{2}$ in \eqref{main_thm_eq} can be
relaxed to $\|\nabla u\|_{L^2((0,T)\times\mathbb{R}^3)}^{2}$. Thus
it says that any (higher) derivatives can be controlled by having only 
$L^2$ estimate of dissipation of energy.
\end{rem}

\begin{rem}\label{weak_L_p}
 The result of the part $(I)$ for $\alpha=0$ extends the result of the previous paper \cite{vas:higher}
because  for any $0<q<p<\infty$ and any bounded subset $\Omega\subset\mathbb{R}^n$, we have 
\begin{equation*}
\|f\|_{L^{q}(\Omega)}\leq C\cdot
\|f\|_{L^{p,\infty}(\Omega)}
\end{equation*}   where C depends only
on   $p$, $q$, dimension $n$ and Lebesgue measure of $\Omega$ (e.g. see Grafakos \cite{grafakos}).
\end{rem}
\begin{rem}
The assumption ``smoothness'' in the part $(I)$ is about pure differentiability.
For example, the result of the part $(I)$ for $d\geq1$
and $\alpha=0$
holds once we know that  $u$ is $d$-times differentiable. In addition,
constants in \eqref{main_thm_eq} are independent of any possible blow-time $T$.
\end{rem}
\begin{rem}
$p=4/(d+\alpha+1)$ is a very interesting relation as mentioned before. Due to this $p$, the estimate \eqref{main_thm_eq}
is 
%a type of 
a non-linear estimate while many other $a$ $priori$ estimates are linear. % in general.
Also, from  the part $(II)$  when $(d+\alpha)$ is very close to $3$, 
we can see that almost third derivatives of weak solutions are locally integrable. 
%$(II)$ for the case $d=2$ and $\alpha$ close to $1$
%implies that almost third derivatives of weak solutions are locally integrable. 
Moreover, imagine that the part $(II)$ for $d=\alpha=0$ be true even though
we can NOT prove it here.
This would imply that this weak solution $u$ could lie in $L^{4,\infty}$ which is beyond the 
best known
estimate $u\in L^{10/3}$ from $L^2$ initial data. \\
\end{rem}

Before presenting the main ideas, we want to mention that 
 Caffarelli,  Kohn and  Nirenberg \cite{ckn} 
 contains two different kinds of 
local regularity criteria. The first one is quantitative, and it says that
if $\| u\|_{L^3(Q(1))}$ and $\| P\|_{L^{3/2}(Q(1))}$ is small, 
then $u$ is bounded by some universal constant in $Q({1/2})$. 
The second one  says that
 $u$ is locally bounded near the origin
 if $\limsup_{r\rightarrow 0}$ $r^{-1}\|\nabla u\|^2_{L^2(Q(r))}$ is 
small. So it  is qualitative
in the sense that 
the conclusion says not that $u$ is  bounded by a universal constant but 
that $\sup |u|$ for some local neighborhood is not infinite.\\
%we do not have a certain fixed bound.\\

On the other hand,  
there is a different quantitative local regularity criterion in
\cite{vas:partial}, which showed
that for any $p>1$, there exists $\epsilon_p$  such that
\begin{equation}\label{vas:partial_result}
\mbox{if}\quad
\|u\|_{L^\infty_tL^2_x(Q_1)}+\|\nabla u\|_{L^2_tL^2_x(Q_1)}+
\|P\|_{L^p_tL^1_x(Q_1)}\leq \epsilon_p, \mbox{ then } |u|\leq 1 
\mbox { in } Q_{1/2}
\end{equation} Recently, this criterion was used in
\cite{vas:higher} in order to obtain higher derivative estimates.
The main proposition in \cite{vas:higher} says that
%$\|\nabla u\|^2_{L^2(Q(1))}$, $\|\nabla^2 P\|_{L^1(Q(1))}$
if both $ \||\nabla u|^2+|\nabla^2 P|\|_{L^1(Q(1))}$
 and some other 
 quantity about pressure
 are small, then $u$ is bounded by $1$ 
at the origin
%in $Q({1/2})$
once $u$ has a mean zero property in space. 
%This result can be seen 
%as a quantitative version of  local regularity criterion.
We can observe that $\|\nabla u\|^2_{L^2(Q(1))}$ and $\|\nabla^2 P\|_{L^1(Q(1))}$
have the same best scaling like \eqref{best_scaling} among 
all the other quantities which we can obtain from  $L^2$ initial data.
However, the other quantity about pressure 
 has a slightly worse scaling. That is the reason that  %it 
%spoiled this estimate a
% little bit. So, up to now,
 the limit case   $L^{\frac{4}{d+1},\infty}$ has been missing in \cite{vas:higher}.\\
% it shows  $L^{\frac{4}{d+1}-\epsilon}$ 
%estimate for any small $\epsilon>0$.

Here are the  main ideas of proof. 
%to prove our main theorem of this paper
% which let us to achieving 
%so that we achieved 
%the improvements in three ways, respectively.
First, 
in order to obtain the missing limit case
%improvement from $L^{\frac{4}{d+1}-\epsilon}$ to 
$ L^{\frac{4}{d+1},\infty}$, we will see that it requires an equivalent estimate 
of \eqref{vas:partial_result} for $p=1$.
%is due to obtaining a new
%local regularity criterion. 
%, which was used in \cite{vas:higher}.
%Our new result requires an equivalent estimate for $p=1$.
Here we extend this result  up to $ p=1$ for some
approximation of the Navier-Stokes (see the proposition 
\ref{partial_problem_II_r}).
% for solutions of some approximation of Navier-Stokes
%and \ref{partial_problem_smooth} for suitable weak solutions of Navier-Stokes).
%For $p>1$ there exists a universal constant $C_p$ such that any suitable weak solution $u$ 
%of \eqref{navier} and \eqref{suitable}  in
%$[-1,1]\times B(1)$ verifying
%\begin{equation}\label{partial_restriction}
%\sup_{t\in[-1,1]}(\int_{B(1)}|u|^2dx) + \int_{-1}^1\int_{B(1)}|\nabla u|^2dx dt
% + 
% \int_{-1}^1\Big(\int_{B(1)}|P|dx\Big)^p dt \leq C_p
%\end{equation}
%is bounded by 1 on $ [-\frac{1}{2},1]\times B(\frac{1}{2})$. \\
To obtain this first goal, we will introduce a new pressure decomposition
(see the lemma \ref{lem_pressure_decomposition}), 
which will be used in  the De Giorgi-type argument. This makes us to 
%so that we can 
remove the bad scaling term about pressure 
%in proposition 10 
in \cite{vas:higher}.
As a result, 
by using the Galilean invariance property and   some blow-up technique with the standard Navier-Stokes scaling,
%$u(t,x)\rightarrow 
%\lambda u(\lambda^2 t,\lambda x)$,
we can proceed our local study in order to obtain a  better 
version of a quantitative partial regularity criterion
for some
approximation of the Navier-Stokes 
(see the proposition 
\ref{local_study_thm}).
% for solutions of the approximation and \ref{prop_local_study_smooth_version}
%for suitable weak solutions.
%This has the ideal scaling among any apriori estimates from $L^2$ initial data.
As a result, we can prove   $ L^{\frac{4}{d+1},\infty}$ estimate
%and they gives us the result 
for classical derivatives ($\alpha=0$ case).\\

Second, 
the result for fractional derivatives ($0<\alpha<2$ case) is  not obvious at all
because there is no proper interpolation theorem 
%between weak spaces 
for $L_{loc}^{p,\infty}$ spaces.
For example, due to the non-locality of the fractional Laplacian operator, the fact $\nabla^2 u\in L_{loc}^{\frac{4}{3},\infty}$ 
with $\nabla^3 u\in L_{loc}^{1,\infty}$ 
does not imply the case of fractional derivatives even if we assume $u$ is smooth. 
%In this paper we 
%approach via the integral representation of fractional Laplacian. 
Moreover, even though we assume that 
$\nabla^2 u\in L^{\frac{4}{3}}(\mathbb{R}^3)$ 
and $\nabla^3 u\in L^{1}(\mathbb{R}^3)$ which we can NOT prove here,
the standard interpolation theorem still requires $L^p(\mathbb{R}^3)$ for some $p>1$ (we refer 
%page 168 in 
Bergh  and L{\"o}fstr{\"o}m  \cite{bergh}).
To overcome
the difficulty,
%of fractional Laplacian operator, 
we will use the Maximal functions
of $u$ which capture its behavior of long-range part. 
Unfortunately, second derivatives of pressure, which lie in the Hardy space $\mathcal{H}
\subset L^1(\mathbb{R}^3)$
 from 
 %compensated compactness of 
 Coifman,  Lions,  Meyer and  Semmes  \cite{clms},
do not have an integrable Maximal function since
the Maximal operator is not bounded on $L^1$. In order to handle non-local parts of pressure,
%we adopt
%compensated compactness result
%and 
we will use some
property    of Hardy space,
which says that some integrable functions play 
a similar role of the Maximal function(see \eqref{hardy_property}).\\
%which says that every Hardy space function has an integrable 
%Maximal-LIKE function (see \eqref{hardy_property}).\\% at \eqref{pressure_hardy_used}.\\
%In fact, a norm on the unit ball of any Hardy space function contains some
%information about the function itself in whole space.\\
%case is trivial
%usual interpolation works only between $d=1$ and $d=2$ because
%our estimate says $\nabla^3 u\in L^{1,\infty}$ and interpolation does not work %$L^{1,\infty}$.(e.g. see lemma 2 in \cite{vas:higher})

Finally, the result $(II)$ for weak solutions comes from 
specific approximation of Navier-Stokes equations that
 Leray \cite{leray} used in order to construct a global time weak solution 
:
$\partial_tu_n+((u_n*\phi_{(1/n)})\cdot\nabla)u_n+ \nabla P_n -\Delta u_n=0$ and $
\ebdiv u_n=0$ where $\phi$ is a fixed mollifier in $\mathbb{R}^3$,
 and $\phi_{(1/n)}$ is defined by $\phi_{(1/n)}(\cdot)=n^3\phi(n\mbox{ }\cdot)$.
Main advantage for us of adopting this approximation is that 
it has strong existence theory of global-time smooth solutions
$u_n$ for each $n$, and it is well-known that there exists
a suitable weak solution $u$ as a weak limit.
In fact,  for any integer $d\geq1$ 
and for any $\alpha\in[0,2)$,
we will obtain  bounds for $u_n$ in the form of \eqref{main_thm_eq}
with $T=\infty$,
which is uniform in $n$.
%for any smooth solutions $u_n$ of the above approximated equation.
%, which is well-known to converge to a suitable weak solution $u$. 
Since $p=4/(d+\alpha+1)$ is greater than $1$ for the case $(d+\alpha)<3$, 
we can know that $(-\Delta)^{\frac{\alpha}{2}}
\nabla^d u $ exists as a   locally integrable function from weak-compactness of $L^p$ for $p>1$. \\

\noindent However, to prove \eqref{main_thm_eq} uniformly for the approximation 
is nontrivial because our proof is based on local study while
%Main difficulties comes from that 
the approximation is not
 scaling-invariant  with the standard Navier-Stokes scaling: %for local study
%unlike the original Navier-Stokes equations. 
After the scaling, the advection velocity $u*\phi_{(1/n)}$
depends the original velocity $u$ more non-locally than before.
Moreover, when we consider the case of fractional derivatives of weak solutions, it requires
even Maximal of Maximal functions to handle  non-local parts of 
the advection velocity 
which depends   the original velocity 
non-locally.\\

The paper is organized as follows. In the next section, preliminaries with
the main propositions
\ref{partial_problem_II_r} and 
\ref{local_study_thm}
will be introduced. 
Then we prove those propositions \ref{partial_problem_II_r} and \ref{local_study_thm}
in  sections \ref{proof_partial_prob_II_r} and 
%in the section 
\ref{proof_local study},
respectively.
% By using proposition \ref{partial_problem_II_r},
%proposition \ref{local_study_thm} will be proved 
Finally we will explain how  
the proposition \ref{local_study_thm}
implies the part $(II)$ of the theorem 
\ref{main_thm}
% of theorem 
%\ref{main_thm} in section \ref{proof_main_thm_II}. Then 
% we talk about how to prove the part $(I)$ 
 for $\alpha=0$ and for $0<\alpha<2$  in subsections 
\ref{prof_main_thm_II_alpha_0}
%\ref{proof_main_thm_I}
and 
%in the subsections 
\ref{prof_main_thm_II_alpha_not_0} respectively while
the part $(I)$   will be covered in the subsection \ref{proof_main_thm_I}.
%Then 
% we talk about how to prove the part $(I)$  of the theorem \ref{main_thm}
%in the subsection
After that, 
the appendix contains some missing proofs of 
technical lemmas.\\

\section{Preliminaries, definitions and main propositions}\label{prelim}
We begin this section by fixing some notations and reminding
some well-known results on analysis. After that we will present
definitions of two approximations and two main propositions. In this paper, any derivatives, convolutions and Maximal functions
are with respect to space variable $x\in\mathbb{R}^3$ unless time variable 
is specified.\\

\noindent  \textbf{Notations for general purpose}\\

%Here are many notations and definitions.\\

We define %and fix $B(r),B(x;r),Q(r),\phi$ and $\phi_r$ by:
$B(r) =\mbox{ the ball in } \mathbb{R}^3$ 
  centered at the origin with radius  $r$,
  $Q(r) =(-r^2,0)\times B(r)$, the cylinder in  $  \mathbb{R}\times\mathbb{R}^3$ and 
  $B(x;r) =\mbox{ the ball in } \mathbb{R}^3$ 
 centered at x with radius $r$.\\

To the end of this paper, we fix  $\phi \in C^{\infty}(\mathbb{R}^3)$ satisfying:\\
\begin{equation*}\begin{split}
\int_{\mathbb{R}^3} &\phi(x)dx = 1,\quad
 supp(\phi) \subset B(1),\quad
 0 \leq \phi \leq 1\\
\end{split}\end{equation*}
\begin{equation*}\begin{split}
 \phi(x)&=1 \mbox{ for } |x|\leq \frac{1}{2} \quad\mbox{ and }\quad
 \phi \mbox{ is radial. }
\end{split}\end{equation*}
For real number $r > 0 $ , we define functions $\phi_r \in C^{\infty}(\mathbb{R}^3)$
 by $\phi_r(x) =\frac{1}{r^3}\phi(\frac{x}{r})$. Moreover, for $r=0$,
we define $\phi_r=\phi_0=\delta_0$ as the Dirac-delta function, which implies that
the convolution between $\phi_0$ and  any function becomes
the function itself. % the identity map. \\
%Note that $ supp(\phi_r) \subset B(r)$.\\
From the Young's inequality for convolutions, we can observe
\begin{equation}\label{young}
 \|f*\phi_r\|_{L^p(B(a))}\leq\|f\|_{L^p(B(a+r))}
\end{equation} due to  $supp(\phi_r)\subset B(r)$ for any $p\in[1,\infty]$,
for any $f\in L^p_{loc}$ and for any $a,r>0$.\\

\noindent  \textbf{ $L^p$,  weak-$L^p$  and Sobolev spaces $W^{n,p}$}\\

Let $K$ be a open subset $K$ of $\mathbb{R}^n$.
For $0<p<\infty $, we define $L^{p}(K)$ by the standard way
with (quasi) norm $\|f\|_{L^{p}(K)}
= (\int_K|f|^pdx)^{(1/p)}$. 
From the Banach-–Alaoglu theorem,
any sequence which is bounded in $L^{p}(K)$ for $p\in(1,\infty)$
has a weak limit from some subsequence due to the weak-compactness.\\

 Also, for $0<p<\infty $, the weak-$L^p(K)$ space
(or $L^{p,\infty}(K)$)  is defined by  \\
\begin{equation*}\begin{split}
L^{p,\infty}(K) = \{ f \mbox{ measurable in } K\subset\mathbb{R}^d
\quad: \sup_{\alpha>0}\Big(\alpha^p\cdot|\{|f|>\alpha\}\cap K|\Big)<\infty\}
\end{split}\end{equation*} with (quasi) norm $\|f\|_{L^{p,\infty}(K)}
= \sup_{\alpha>0}\Big(\alpha\cdot|\{|f|>\alpha\}\cap K|^{1/p}\Big)$.
From the Chebyshev's inequality, we have $\|f\|_{L^{p,\infty}(K)}\leq\|f\|_{L^{p}(K)}$ for any $0<p<\infty$.
Also, for $0<q<p<\infty$, $L^{p,\infty}(K)\subset L^{q}(K)$
once $K$ is bounded (refer the remark \ref{weak_L_p} in the beginning).\\

For any integer $n\geq0$ and for any $p\in[1,\infty]$, we denote $W^{n,p}(\mathbb{R}^3)$ and $W^{n,p}(B(r))$
as the standard Sobolev spaces for the whole space $\mathbb{R}^3$ and for any ball $B(r)$ in $\mathbb{R}^3$,
respectively.\\

\noindent  \textbf{The Maximal function $\mathcal{M}$ and the Riesz transform $\mathcal{R}_j$}\\

The Maximal function $\mathcal{M}$ in $\mathbb{R}^d$ is
 defined by the following standard way:
\begin{equation*}\begin{split}
\mathcal{M}(f)(x)=&\sup_{r>0}\frac{1}{|B(r)|}\int_{B(r)}|f(x+y)|dy.
%\mathcal{M}^{(t)}(g)(x)=&\sup_{r>0}\frac{1}{2r}\int_{x-r}^{x+r}|f(y)|dy\\
\end{split}\end{equation*}
Also, we can express this Maximal operator as a
supremum of convolutions: $\mathcal{M}(f)=C\sup_{\delta>0}\Big(\chi_\delta *|f|\Big)$ 
where $\chi=\mathbf{1}_{\{|x|<1\}}$ is the characteristic function of the unit ball,
and $\chi_\delta(\cdot)=(1/\delta^3)\chi(\cdot/\delta)$.
One of properties of the Maximal function is that 
$\mathcal{M}$ is bounded
from  $L^{p}(\mathbb{R}^d)$  to $L^{p}(\mathbb{R}^d)$ 
for $p\in(1,\infty]$ and 
from $L^{1}(\mathbb{R}^d)$ to  $L^{1,\infty}(\mathbb{R}^d)$.
In this paper,
we denote $\mathcal{M}$ and $\mathcal{M}^{(t)}$ as  the Maximal functions
in $\mathbb{R}^3$ and in 
$\mathbb{R}^1$, respectively. \\

For $1\leq j\leq 3$, the Riesz Transform $\mathcal{R}_j$ in $\mathbb{R}^3$ is defined by:
\begin{equation*}\begin{split}
\widehat{\mathcal{R}_j(f)}(x)=\mathit{i}\frac{x_j}{|x|}\hat{f}(x)
\end{split}\end{equation*} 
for any $f\in\mathcal{S}$ (the Schwartz space). % where $\quad \widehat{ }$\quad is Fourier transform. 
Moreover we can extend such definition for functions $L^{p}(\mathbb{R}^3)$
for $1<p<\infty$ and it is well-known that
$\mathcal{R}_j$ is bounded in $L^{p}$ for the same range of $p$.\\

\noindent  \textbf{The Hardy space $\mathcal{H}$}\\

The Hardy space $\mathcal{H}$ in $\mathbb{R}^3$ is defined by
\begin{equation*}
 \mathcal{H}(\mathbb{R}^3)=\{f\in L^1(\mathbb{R}^3)\quad: \quad\sup_{\delta>0}|
\mathcal{P}_\delta * f|\in L^1(\mathbb{R}^3)  \}
\end{equation*} where $\mathcal{P}=
%\frac{\Gamma((n+1)/2)}{\pi^{(n+1)/2}}
C
(1+|x|^2)^{-2}$ is the
Poisson kernel and $\mathcal{P}_\delta$ is defined
by $\mathcal{P}_\delta(\cdot)=\delta^{-3}\mathcal{P}(\cdot/\delta)$.
A norm of 
 $\mathcal{H}$ is defined by $L^1$ norm of $\sup_{\delta>0}|
\mathcal{P}_\delta * f|$. Thus $\mathcal{H}$
is a subspace of $L^1(\mathbb{R}^3)$ and 
$\|f\|_{L^{1}(\mathbb{R}^3)}
\leq\|f\|_{\mathcal{H}(\mathbb{R}^3)}$ for any $f\in\mathcal{H}$.
 Moreover,  
the Riesz Transform is bounded from $\mathcal{H}$ to $\mathcal{H}$.\\

\noindent One of important applications of the Hardy space is the compensated compactness
(see Coifman,  Lions, Meyer and Semmes \cite{clms}). Especially, it says that if $E,B\in L^2(\mathbb{R}^3)$
and $curlE=\ebdiv B=0$ in distribution, then 
$E\cdot B\in\mathcal{H}(\mathbb{R}^3)$ and we have 
\begin{equation*}
\|E\cdot B\|_{\mathcal{H}(\mathbb{R}^3)}\leq
C\cdot\|E\|_{L^2(\mathbb{R}^3)}\cdot\|B\|_{L^2(\mathbb{R}^3)}
\end{equation*} for some universal constant $C$.
In order to obtain some regularity of second derivative of pressure, 
we can combine compensated compactness 
with boundedness of the Riesz transform in $\mathcal{H}(\mathbb{R}^3)$.
For example, if $u$ is  a  weak solution  
of  the Navier-Stokes \eqref{navier}, %on  $ (0,T)\times \mathbb{R}^3 $
 then a corresponding pressure $P$ satisfies \begin{equation}\label{pressure_hardy}  
\|\nabla^2 P\|_{L^1(0,\infty;\mathcal{H}(\mathbb{R}^3))}\leq
C\cdot\|\nabla u\|_{L^2(0,\infty;L^2(\mathbb{R}^3))}^2
\end{equation} (see Lions \cite{lions} or the lemma 7 in \cite{vas:higher}).\\
\begin{comment}
Or if $u$ is a solution of \eqref{navier_Problem I-n} in (Problem I-n)
 then \begin{equation}\begin{split}\label{pressure_hardy_Problem I-n} 
\|\nabla^2 P\|_{L^1(0,\infty;\mathcal{H}(\mathbb{R}^3))}&\leq
C\|\Delta  P\|_{L^1(0,\infty;\mathcal{H}(\mathbb{R}^3))}\\
&\leq C\cdot\|\nabla (u * \phi_{1/n})\|_{L^2(0,\infty;L^2(\mathbb{R}^3))}
\|\nabla u\|_{L^2(0,\infty;L^2(\mathbb{R}^3))}\\
%&=C\cdot\|(\nabla u) * \phi_{1/n}\|_{L^2(0,\infty;L^2(\mathbb{R}^3))}
%\|\nabla u\|_{L^2(0,\infty;L^2(\mathbb{R}^3))}\\
&\leq C\cdot\|\nabla u\|_{L^2(0,\infty;L^2(\mathbb{R}^3))}^2.\\
\end{split}\end{equation} from $-\Delta P=\ebdiv\ebdiv
\Big( (u * \phi_{1/n})\otimes  u$\Big).\\ 
\end{comment}

\noindent Now it is well known that if we replace  the
Poisson kernel $\mathcal{P}$ with any  function
$\mathcal{G}\in C^\infty(\mathbb{R}^3)$ with compact support, then we have a constant 
$C$ depending only on $\mathcal{G}$ such that
\begin{equation}\begin{split}\label{hardy_property}
\| \sup_{\delta>0}|\mathcal{G}_{\delta}* f|
\|_{L^{1}(\mathbb{R}^3)}
\leq C\| \sup_{\delta>0}|\mathcal{P}_{\delta}* f|
\|_{L^{1}(\mathbb{R}^3)}
= C\|f\|_{\mathcal{H}(\mathbb{R}^3)}
\end{split}\end{equation}
 where $\mathcal{G}_\delta(\cdot)=\mathcal{G}(\cdot/\delta)/\delta^3$.
(see Fefferman and Stein \cite{fefferman} or 
see Stein \cite{ste:harmonic}, Grafakos \cite{grafakos} for modern texts). 
Due to the supremum and the convolution in \eqref{hardy_property},
 we can say that 
 even though   the Maximal function 
 $\sup_{\delta>0}\Big(\chi_\delta *|f|\Big)$ of any non-trivial Hardy space function $f$
 is not integrable, there exist at least 
 integrable functions $\Big(\sup_{\delta>0}\Big|\mathcal{G}_{\delta}* f\Big|\Big)$, which can  capture non-local data as 
Maximal functions  do. 
However, note  the position of the absolute value sign in \eqref{hardy_property},
which is outside of the convolution while
it is inside of the convolution
for the Maximal function.
It implies that 
\eqref{hardy_property} is slightly weaker than the Maximal function
in the sense of controlling non-local data.
 This weakness is the reason that we  introduce
 certain definitions of  $\zeta$ and $h^{\alpha}$
in the following.\\ 

\noindent  \textbf{Some notations which will be useful for fractional derivatives 
 $ {(-\Delta)^{{\alpha}/{2}}}$}\\

The following two definitions of  $\zeta$ and $h^{\alpha}$  will be used
only in the proof for fractional derivatives.
We define $\zeta$ by $\zeta(x)=\phi(\frac{x}{2})-\phi(x)$. 
Then we have
\begin{equation}\begin{split}\label{property_zeta}
& \zeta \in C^{\infty}(\mathbb{R}^3),\quad supp(\zeta)\subset B(2),
 \quad\zeta(x)=0 \mbox{ for } |x|\leq \frac{1}{2} \\
&\mbox{  and }\sum_{j=k}^{\infty}\zeta(\frac{x}{2^j})=1 \mbox{  for }|x|\geq 2^{k}
\mbox{   for any integer } k.
\end{split}\end{equation} %($k$ don't need to be positive)\\

\noindent In addition, we define function $h^{\alpha}$
for $\alpha>0$ 
by
$h^{\alpha}(x)=\zeta(x)/|x|^{3+\alpha}$.
Also we define $(h^{\alpha})_{\delta}$ and
 $(\nabla^{d}{h^{\alpha})_\delta}$  by 
 $(h^{\alpha})_{\delta}(x)=\delta^{-3}h^{\alpha}(x/\delta)$ and 
 $(\nabla^{d}{h^{\alpha})_\delta}(x)
=\delta^{-3}(\nabla^{d}{h^{\alpha})}(x/\delta)$
for $\delta>0$ and for positive integer $d$, respectively.
Then they satisfy
%\begin{equation*}\begin{split}
%& h^{\alpha} \in C^{\infty}(\mathbb{R}^3),
%\quad supp(h^{\alpha})\subset B(2)- B(1/2),
%\end{split}\end{equation*} and 
\begin{equation}\begin{split}\label{property_h}
& (h^{\alpha})_{\delta} \in C^{\infty}(\mathbb{R}^3),
\quad supp((h^{\alpha})_{\delta})\subset B(2\delta)-
 B(\delta/2), \\
%\quad (h^{\beta})_{\delta}=0 \mbox{ for } |x|\leq \frac{1}{2}\delta \\
&\mbox{  and }\frac{1}{|x|^{3+\alpha}}\cdot\zeta
(\frac{x}{2^j})=\frac{1}{(2^j)^{\alpha }} \cdot (h^{\alpha})_{2^j}(x)
 \mbox{  for any integer } j.%\mbox{  for }|x|\geq 2^{-k}
\end{split}\end{equation} 
\ \\

\noindent  \textbf{The definition of the fractional Laplacian  
 $ {(-\Delta)^{{\alpha}/{2}}}$}\\

%For $\alpha=0$, we define
% the fractional Laplacian $ (-\Delta)^{\frac{\alpha}{2}}=I$ as the
% identity map while,
 For $-3<\alpha\leq2$
 and for $f\in \mathcal{S}(\mathbb{R}^3)$ 
(the Schwartz space), $ (-\Delta)^{\frac{\alpha}{2}}f$  % in $\mathbb{R}^3$ 
is defined by the Fourier transform: 
\begin{equation}\label{fractional_fourier}
%  ((-\Delta)^{\frac{\alpha}{2}})^{f}f
\widehat{(-\Delta)^{\frac{\alpha}{2}}f}(\xi)=|\xi|^\alpha \hat{f}(\xi)
%=C_{\alpha}\cdot P.V.\int_{\mathbb{R}^3}
%\frac{f(x)-f(y)}{|x-y|^{3+\alpha}}dy
\end{equation} 
Note that 
 $(-\Delta)^{0}=Id$. % In this paper, we always assume $0\leq\alpha<2$.\\
Especially, for  $\alpha\in(0,2)$, the fractional Laplacian can also be defined  
by the singular integral for any $f\in\mathcal{S}$: 
\begin{equation}\label{fractional_integral}
  (-\Delta)^{\frac{\alpha}{2}}f(x)=C_{\alpha}\cdot P.V.\int_{\mathbb{R}^3}
\frac{f(x)-f(y)}{|x-y|^{3+\alpha}}dy.\\
\end{equation} 

We introduce two notions of approximations to Navier-Stokes.
The first one or (Problem I-n) is the approximation  Leray \cite{leray} used
while the second one or (Problem II-r) will be used in local study after we apply some certain scaling
to (Problem I-n).\\

\noindent  \textbf{Definition of  {(Problem I-n)}:
the first approximation
to Navier-Stokes}\\

%The following approximation is the one Leray \cite{leray} used.

%Here is  the definition of ({Problem I-n}) which approximates
%the Navier-Stokes. 
\begin{defn}\label{Problem I-n} Let $n\geq1$ be either an integer or the infinity $\infty$, and let $0<T\leq \infty$.
Suppose that $u_0$ satisfy \eqref{initial_condition}.
We say that $(u,P)\in [C^{\infty}\big((0,T)\times\mathbb{R}^3\big)]^2$
%on $C^\infty((0,\infty)\times\mathbb{R}^3)$ 
is a solution of {(Problem I-n)} on $(0,T)$ for the data $u_0$
if it  satisfies 
%there exist a function P on $(0,\infty)\times\mathbb{R}^3$ and $(u,P)$
\begin{equation}\label{navier_Problem I-n}
 \begin{split}
\partial_tu+((u*\phi_{{\frac{1}{n}}})\cdot\nabla)u+ \nabla P -\Delta u&=0\\
\ebdiv u&=0 \quad t\in ( 0,T ),\quad x\in \mathbb{R}^3 
\end{split}
\end{equation}  
%for some $P$ on $(0,\infty)\times\mathbb{R}^3$ 
and
\begin{equation}\label{initial_condition_Problem I-n}
u(t)\rightarrow u_0*\phi_{\frac{1}{n}} \mbox{ in } L^2 \mbox{-sense as }
t \rightarrow 0.
%u(0)=(u_0*\phi_{\frac{1}{n}}) \quad \mbox{,} \quad x\in \mathbb{R}^3 
\end{equation}% as distributions.\\

\end{defn}
\begin{rem}
When $n=\infty$, \eqref{navier_Problem I-n} is  the Navier-Stokes  on $(0,T)\times\mathbb{R}^3$ with initial value $u_0$.
\end{rem}
\begin{rem}\label{remark_leray}

If $n$ is not the infinity but an positive integer, then for any given $u_0$ of \eqref{initial_condition}, we have existence 
and uniqueness theory of (Problem I-n) on $(0,\infty)$ with the energy equality 

%Note that if you integrate the above equation in space-time, then we have 
\begin{equation}\label{energy_eq_Problem I-n}
\|u(t)\|_{L^2(\mathbb{R}^3)}^2+
2\|\nabla u\|^2_{L^2(0,t;L^2(\mathbb{R}^3))}
= \|u_0*\phi_{\frac{1}{n}}\|_{L^2(\mathbb{R}^3)}^2.\\
\end{equation} for any $t<\infty$
and it is well-known that we can extract a sub-sequence which 
converges to a suitable weak 
solution $u$ of \eqref{navier} and \eqref{suitable} with the initial data 
$u_0$ of \eqref{initial_condition}
by limiting procedure on  a sequence of solutions of (Problem I-n)
(see Leray \cite{leray}, or see Lions \cite{lions}, Lemari{\'e}-Rieusset \cite{lemarie} for modern texts).
\end{rem}
\begin{rem}
As mentioned in the introduction section, we can observe that this notion (Problem I-n) 
is not invariant under the standard Navier-Stokes scaling  $u(t,x)\rightarrow 
\epsilon u(\epsilon^2 t,\epsilon x)$ 
due to the advection velocity $(u*\phi_{1/n})$
unless $n$ is the infinity. \\
%That is the reason we introduce the next definiton
%of  {(Problem II-r)}.  \\
\end{rem}

\noindent  \textbf{Definition of  {(Problem II-r)}:
the second approximation
to Navier-Stokes}\\

%We introduce one more definition. 
%It appears when
%we scale a solution of (Problem I-n) by standard Navier-Stokes scaling $u(t,x)\rightarrow 
%\epsilon u(\epsilon^2 t,\epsilon x)$.
\begin{defn}\label{problem II-r}  Let $0\leq r<\infty$ be real.
We say that $(u,P)\in [C^{\infty}\big((-4,0)\times\mathbb{R}^3\big)]^2$ is a solution of {(Problem II-r)} 
%$(u,w,P)$ on $(-4,\infty)\times\mathbb{R}^3$ 
if it satisfies %$u(-4)\in L^2(\mathbb{R}^3)$ and
\begin{equation}\label{navier_Problem II-r}
 \begin{split}
\partial_tu+(w
\cdot\nabla)u+ \nabla P -\Delta u&=0\\
\ebdiv u&=0, \quad t\in (-4,0 ), \quad x\in \mathbb{R}^3 
\end{split}
\end{equation} % as distributions
% for some $P$ on $(0,\infty)\times\mathbb{R}^3$ and 
where $w$ is the difference of two functions:
\begin{equation}\begin{split}\label{w_Problem II-n}
w(t,x)
%&= (u*\phi_{r})(t,x)-\int_{\mathbb{R}^3}\phi(y)(u*\phi_{r})(t,y)dy\\
&= w^\prime(t,x) -  w^{\prime\prime}(t),
 \quad t\in (-4,0 ), x\in \mathbb{R}^3 \\
\end{split}\end{equation} which are defined by u in the following way:
\begin{equation*}\begin{split}
w^\prime(t,x)&=(u*\phi_{r})(t,x) \quad\mbox{ and } \quad
w^{\prime\prime}(t)=\int_{\mathbb{R}^3}\phi(y)(u*\phi_{r})(t,y)dy.
\end{split}\end{equation*}

\end{defn}
\begin{rem}
This notion of {(Problem II-r)} gives us the mean zero property for 
the advection velocity $w$: $\int_{\mathbb{R}^3}\phi(x)w(t,x)dx = 0 $ on $(-4,0)$. Also this $w$ is divergent free from the definition.
%Note that, by the definition, $w$ satisfies the mean zero property %$\int_{\mathbb{R}^3}\phi(x)w(t,x)dx = 0 $ on $(-4,0)$
%and $\ebdiv w =0$.  
Moreover, by multiplying $u$ to \eqref{navier_Problem II-r}, 
we have
\begin{equation}\label{suitable_Problem II-r}
\partial_t \frac{|u|^2}{2} + \ebdiv (w \frac{|u|^2}{2}) + \ebdiv (u P) +
 |\nabla u |^2 - \Delta\frac{|u|^2}{2}= 0\\
\end{equation} in classical sense 
because our definition needs $u$ to be $C^{\infty}$.
\end{rem}
\begin{rem}
We will introduce some specially designed $\epsilon$-scaling 
which is a bridge between (Problem I-n) and (Problem II-r)
(it can be found in \eqref{special designed scaling}).
This scaling
is based on the Galilean invariance in order to obtain
the mean zero property for the velocity $u$:
$\int_{\mathbb{R}^3}\phi(x)u(t,x)dx = 0 $ on $(-4,0)$.
Moreover, after this  $\epsilon$-scaling is applied to  solutions of (Problem I-n),
the resulting functions will
satisfy not conditions of (Problem II-$\frac{1}{n}$) but 
those of (Problem II-$\frac{1}{n\epsilon}$) (it can be found \eqref{special designed scaling result}).
These things will be
stated precisely in the section \ref{proof_main_thm_II}.
\end{rem}

\begin{rem}
When $r=0$, the equation \eqref{navier_Problem II-r} is
% the smooth solution of 
the Navier-Stokes  on $(-4,0)\times\mathbb{R}^3$ 
once
% such $u$ exists and if 
we  assume the mean zero property for $u$.\\
%: $\int_{\mathbb{R}^3}\phi(x)u(t,x)dx = 0 $ on $(-4,0)$.
\end{rem}

Now we present  two main local-study propositions which require
the notion of (Problem II-r). These are kinds of partial regularity theorems 
for solutions of (Problem II-r). 
The main difficulty to prove these two propositions is that 
 $\bar{\eta}$ and $\bar\delta>0$ should be  independent  of any $r$ in $[0,\infty)$.
We will prove this independence  very carefully, which is the heart of the  section
\ref{proof_partial_prob_II_r} and 
\ref{proof_local study}.\\ 
%in the following sections \ref{proof_partial_prob_II_r} and 
%\ref{proof_local study}.\\

\noindent  \textbf{The first local study proposition 
for  (Problem II-r)}\\

The following one is a quantitative version of partial regularity theorems
which extends that of \cite{vas:partial} up to $p=1$. The proof 
will be based on the De Giorgi iteration with a  new pressure decomposition 
lemma \ref{lem_pressure_decomposition} which will appear later.
\begin{prop}\label{partial_problem_II_r}
There exists a $\bar\delta>0$ with the following property:\\

If u is a solution of (Problem II-r) for some $0\leq r<\infty$ verifying both
\begin{equation*}\begin{split}
 &\| u\|_{L^{\infty}(-2,0;L^{2}(B(\frac{5}{4})))}+ 
\|P\|_{L^1(-2,0;L^{1}(B(1)))}+\| \nabla u\|_{L^{2}(-2,0;L^{2}(B(\frac{5}{4})))}
\leq \bar{\delta}\\
\end{split}\end{equation*}
\begin{equation*}\begin{split}
\mbox{ and }\quad\quad&\| \mathcal{M}(|\nabla u|)
\|_{L^{2}(-2,0;L^{2}(B(2)))}\leq \bar{\delta},
%&\| w\|_{L^2(-2,0;L^{\infty}(B(1)))}+\| \nabla w\|_{L^2(-2,0;L^{\infty(B(1)))}
%\leq \bar{\gamma}\cdot\bar{\delta}\\
\end{split}\end{equation*}
 then we have
\begin{equation*}
 |u(t,x)|\leq 1 \mbox{ on } [-\frac{3}{2},0]\times B(\frac{1}{2}).
\end{equation*}\\
\end{prop}
The above proposition, whose proof will appear in the  section \ref{proof_partial_prob_II_r},
contains two bad scaling terms
$\| u\|_{L_t^{\infty}L_x^{2}}$ and $\|P\|_{L_t^1L_x^{1}}$,
while  the following proposition \ref{local_study_thm} does not have those two. %bad scaling terms. 
Instead, the proposition \ref{local_study_thm} will assume
 the mean-zero property on $u$ with the additional terms.
 We will see later that these additional ones  have
 the best scaling  like $|\nabla u|^2$ (also, see \eqref{best_scaling}).\\ 

\noindent  \textbf{The second local study proposition 
for  (Problem II-r)}\\

\begin{prop}\label{local_study_thm}
There exists  a  $\bar{\eta}>0$ and there exist
  constants $C_{d,\alpha}$ depending only on $d$ and $\alpha$
 %for any integer $d\geq 0$ and for any real $\alpha$ with $0\leq\alpha<2$
 with the following property:\\

If $u$ is a solution of (Problem II-r) for some $0\leq r<\infty$ verifying both            
\begin{equation}\label{local_study_condition1}
                 \int_{\mathbb{R}^3}\phi(x)u(t,x)dx = 0
\quad \mbox{ for } t\in(-4,0) \mbox{ and}\\
\end{equation}
\begin{equation}\begin{split}\label{local_study_condition2}
&\int_{-4}^{0}\int_{B(2)}\Big(|\nabla u|^2(t,x)+
|\nabla^2 P|(t,x)+|\mathcal{M}(|\nabla u|)|^2(t,x)\Big)dxdt\leq \bar{\eta},\\
               \end{split}\end{equation}

 then $|\nabla^d u|\leq C_{d,0}$ on 
$Q(\frac{1}{3})=(-(\frac{1}{3})^2,0)\times B(\frac{1}{3})$ for every integer $d\geq 0$.\\

Moreover if we assume further 
%$u\in L^{\infty}(-4,\infty;L^2(\mathbb{R}^3))$ and 
\begin{equation}\begin{split}\label{local_study_condition3}
&\int_{-4}^{0}\int_{B(2)}
\Big(
|\mathcal{M}(\mathcal{M}(|\nabla u|))|^2+
|\mathcal{M}(|\mathcal{M}(|\nabla u|)|^q)|^{2/q}\\+
&|\mathcal{M}(|\nabla u|^q)|^{2/q}+
\sum_{m=d}^{d+4} \sup_{\delta>0}(|(\nabla^{m-1}{h^{\alpha})_\delta}
*\nabla^2 P|)
\Big)dxdt\leq \bar{\eta}
               \end{split}\end{equation} 
for some integer $d\geq 1$ and  for some real $\alpha\in(0,2)$
where $q=12/(\alpha+6)$,
 then $|(-\Delta)^{\frac{\alpha}{2}}\nabla^d u|\leq C_{d,\alpha}$
on $Q(\frac{1}{6})$ for such $(d,\alpha)$.

\end{prop}
\begin{rem}
For the  definitions of $h^{\alpha}$ and  $(\nabla^{m-1}{h^{\alpha})_\delta}$, see around \eqref{property_h}.
\end{rem}
The proof will be given in the  section \ref{proof_local study} 
which will use the conclusion of the previous proposition \ref{partial_problem_II_r}.
Moreover we will use an induction argument and the integral representation 
of the fractional Laplacian in order to get estimates
for the fractional case.
The Maximal function term of \eqref{local_study_condition2}
is introduced to estimate non-local part of $u$ while
the Maximal of Maximal function terms of \eqref{local_study_condition3}
is to estimate non-local part of $w$ which is already non-local.
On the other hand, because $\nabla^2 P$ has only $L^1$ integrability,
we can not have  $L^1$ Maximal function of $\nabla^2 P$. 
Instead, we use 
 some integrable functions, which is the last term of \eqref{local_study_condition3}.
 This term plays the role which captures
non-local information of pressure (see \eqref{hardy_property}).
These will be stated 
clearly in  sections \ref{proof_local study} and \ref{proof_main_thm_II}.\\

\section{Proof of  the first local study proposition \ref{partial_problem_II_r}  }\label{proof_partial_prob_II_r}

\qquad This section is devoted to prove the proposition \ref{partial_problem_II_r} which is 
a 
%sort of 
partial regularity theorem for (Problem II-r).
% which will be used in the proof of the proposition \ref{local_study_thm}.
Remember that we are looking for $\bar\delta$ which must
 be independent of $r$. \\
 
 In the first subsection \ref{definition_for_thoerem_partial_problem_II_r},
 we  present some lemmas about the advection velocity $w$ and 
 a new pressure 
 decomposition. After that, two big lemmas \ref{lem_partial_1}
and \ref{lem_partial_2}
in the subsections 
\ref{proof_lem_partial_1} and \ref{proof_lem_partial_2}, which give us a control for big $r$ and small $r$ respectively, follow. Then 
the actual proof of the proposition
\ref{partial_problem_II_r} will appear in the last subsection \ref{combine_de_giorgi} 
where we can combine those two big lemmas.\\
 %Thus most difficulties comes from such independence and pressure term.% First, 
%we will give several lemmas and their proofs. After then, 
%the proof of proposition will be given.\\

\subsection{A control on the
 advection velocity $w$ and a new pressure decomposition}\label{definition_for_thoerem_partial_problem_II_r}

\quad The following lemma % and corollary 
says that
convolution of any functions with $\phi_r$
can be controlled by just one  point value of
the Maximal function with some factor of $1/r$.
Of course,  it is useful when $r$ is away from $0$.
\begin{lem}\label{convolution_lem}
 Let  f be an integrable function in $\mathbb{R}^3$. 
Then for any integer $d\geq0$,
there exists $C=C(d)$ such that 
\begin{equation*}
 \|\nabla^d (f*\phi_r)\|_{L^{\infty}(B(2))}\leq \frac{C}{r^d}
\cdot(1+\frac{4}{r})^3\cdot\inf_{x\in B(2)}\mathcal{M}f(x)
\end{equation*} for any $0<r<\infty$.
 
\end{lem}
\begin{proof}
 Let $z,x\in B(2)$. Then
\begin{equation*}\begin{split}
 & |\nabla^d (f*\phi_r)(z)| =|(f*\nabla^d \phi_r)(z)|
 =|\int_{B(z,r)}f(y)\nabla^d\phi_r(z-y)dy|\\
&\leq\|\nabla^d \phi_r\|_{L^{\infty}}\int_{B(z,r)}|f(y)|dy
=\frac{\|\nabla^d \phi\|_{L^{\infty}}}{r^{d+3}}\int_{B(z,r)}|f(y)|dy  \\
&\leq\frac{\|\nabla^d \phi\|_{L^{\infty}}}{r^{d+3}}\frac{(r+4)^3}{(r+4)^3}
\int_{B(x,r+4)}|f(y)|dy  
\leq\frac{C}{r^d}
\cdot(1+\frac{4}{r})^3\cdot\mathcal{M}f(x).
\end{split}\end{equation*}%where $C = \frac{3}{4}\pi\|\nabla^d \phi\|_{L^{\infty}}$\\

We used $B(z,r)\subset B(x,r+4)$. Then we take $\sup$ in $z$ and $\inf$ in $x$.
Recall that $\phi(\cdot)$ is the fixed function in this paper.\\
\end{proof}

The following corollary is just an application of the previous lemma
to solutions of (Problem II-r).
\begin{cor}\label{convolution_cor}
 Let $u$ be a solution of (Problem II-r) for $0<r<\infty$. Then % with correspoding $P$ and $w$
 for any integer $d\geq0$,
there exists $C=C(d)$ such that
\begin{equation*}
 \| w\|_{L^2(-4,0;L^{\infty}(B(2)))}\leq {C}
\cdot(1+\frac{4}{r})^3\cdot\|\mathcal{M}(|\nabla u|)\|_{L^2{(Q(2))}}
\end{equation*} and 
\begin{equation*}
 \|\nabla^d w\|_{L^2(-4,0;L^{\infty}(B(2)))}\leq \frac{C}{r^{d-1}}
\cdot(1+\frac{4}{r})^3\cdot\|\mathcal{M}(|\nabla u|)\|_{L^2{(Q(2))}}
\end{equation*} if $d\geq1$.
\end{cor}
\begin{proof}
Recall $\int_{\mathbb{R}^3}w(t,y)\phi(y)dy=0$ and $supp(\phi)\subset B(1)$. Thus for $z\in B(2)$
\begin{equation*}\begin{split}
|w(t,z)|=&\Big|\int_{\mathbb{R}^3}w(t,z)\phi(y)dy-\int_{\mathbb{R}^3}w(t,y)
\phi(y)dy\Big|\\
\leq&\|\nabla w(t,\cdot)\|_{L^{\infty}(B(2))}\int_{\mathbb{R}^3}|z-y|\phi(y)dy\\
\leq&C\|(\nabla u) *\phi_r(t,\cdot)\|_{L^{\infty}(B(2))}\cdot\int_{\mathbb{R}^3}\phi(y)dy\\
\leq&{C}\cdot(1+\frac{4}{r})^3\cdot \inf_{x\in B(2)}\mathcal{M}(|\nabla u|)(t,x).
\end{split}\end{equation*}
For last inequality, we used the lemma \ref{convolution_lem} to $\nabla u$.
For $d\geq1$, use $\nabla^d w =\nabla^{d-1}\Big[(\nabla u)*\phi_r\Big]$.
%with lemma \ref{convolution_lem}.

\end{proof}

To use De Giorgi type argument, we require more notations
which will be used only in this section.
\begin{equation}\begin{split}\label{def_s_k}
\mbox{For  real }&k\geq 0, \mbox{ define }\\
 B_k &=\mbox{ the ball in } \mathbb{R}^3 \mbox{ centered at the origin with radius }
\frac{1}{2}(1+\frac{1}{2^{3k}}),\\
T_k &= -\frac{1}{2}(3+\frac{1}{2^k}),\\
Q_k &=[T_k,0]\times B_k \quad\mbox{ and}\\
s_k &= \mbox{ distance between } B^c_{k-1} \mbox{ and } B_{k-\frac{5}{6}}\\
 &=2^{-3k}\Big((\sqrt{2}-1)2\sqrt{2}\Big).
\end{split}\end{equation} Also we define $s_\infty=0$.
Note that $0<s_1 <\frac{1}{4}$ and the sequence $\{s_k\}_{k=1}^{\infty}$ is strictly decreasing to zero
as k goes to $\infty$.\\

\noindent For each integer $k\geq 0 $, we define and fix a function $\psi_k \in C^{\infty}(\mathbb{R}^3)$ satisfying:\\
\begin{equation}\begin{split}
&\psi_k = 1 \quad\mbox{ in } B_{k-\frac{2}{3}} , \quad\psi_k = 0 \quad\mbox{ in } B_{k-\frac{5}{6}}^C\\
&0\leq \psi_k(x) \leq 1 , \quad
|\nabla\psi_k(x)|\leq C2^{3k}  \mbox{ and }
 |{\nabla}^2\psi_k(x)|\leq C2^{6k} \mbox{ for }x\in\mathbb{R}^3.
\end{split}\end{equation} This  $\psi_k$ plays role of a cut-off function for $B_k$.\\

To prove the proposition \ref{partial_problem_II_r}, We need the following
 important lemma about pressure decomposition. Here
we decompose our pressure term into three parts: 
a non-local part which depends on $k$, 
a local part which depends on $k$ and
a non-local part which does not depend on $k$
and will be absorbed into the velocity component later.
%Here ``local" and ``non-local" are about 
%how much a piece depends on  the velocity term. 
 %You may compare with lemma 3 of \cite{vas:partial}

\begin{lem}\label{lem_pressure_decomposition}
There exists a  constant ${\Lambda_1}>0$ with 
the following property:\\
% Let $A_{ij}\in L^p(B_{\frac{1}{6}}) $  $ 1\leq i,j\leq 3 $ for some $p>1$ and $ P\in L^1(B_{0})$
%with $-\Delta P = \sum_{ij}\partial_i \partial_j A_{ij}$ in $B_{0}$.
Suppose $A_{ij}\in L^1(B_0) $  $ 1\leq i,j\leq 3 $ %for some $1<p<\infty$ 
and $ P\in L^1(B_0)$ %for some $1<q<\infty$
with $-\Delta P = \sum_{ij}\partial_i \partial_j A_{ij}$ in $B_0$.
% Let $A_{ij}\in L^p(\mathbb{R}^3) $  $ 1\leq i,j\leq 3 $ for some $1<p<\infty$ 
%and $ P\in L^q(\mathbb{R}^3)$ for some $1<q<\infty$
%with $-\Delta P = \sum_{ij}\partial_i \partial_j A_{ij}$ in $\mathbb{R}^3$.
Then, there exist a function $P_3 $ with $P_3|_{B_{\frac{2}{3}}} \in L^{\infty} $ 
such that, for any $k\geq1$, we can decompose $P$ by\\
\begin{equation}\label{pressure_decomposition_expressition}
P = P_{1,k}  + P_{2,k} + P_{3}\quad\mbox{     in } B_{\frac{1}{3}},
\end{equation}
%with following properties:\\
and they satisfy
\begin{equation}\label{lem_pressure_decomposition_p1k}  \|\nabla P_{1,k}\|_{L^{\infty}(B_{k-\frac{1}{3}}))} 
   + \|P_{1,k}\|_{L^{\infty}(B_{k-\frac{1}{3}}))}  
\leq {\Lambda_1}
  2^{12k} \sum_{ij}\|A_{ij}\|_{L^1(B_{\frac{1}{6}})},
\end{equation}
\begin{equation}\label{lem_pressure_decomposition_p2k} -\Delta P_{2,k} = 
\sum_{ij}\partial_i \partial_j (\psi_k A_{ij}) \quad\quad \mbox{in } \mathbb{R}^3 \quad\mbox{ and}
\end{equation}
\begin{equation}\label{lem_pressure_decomposition_p3} \|\nabla P_{3}\|_{L^{\infty}(B_{\frac{2}{3}})} \leq {\Lambda_1}
 (\|P\|_{L^1(B_{\frac{1}{6}})} + \sum_{ij}\|A_{ij}\|_{L^1(B_{\frac{1}{6}})}).
\end{equation}
Note that ${\Lambda_1}$ is a totally independent constant. \\
%Note that local pressure term $P_{2,k} $ is same with lemma 3 of \cite{vas:partial}.\\
\end{lem}
\begin{proof}
%Main difference between the pressure decomposition
%It is very standard decomposition except $P_{3}$.\\
The product rule and the hypothesis imply
\begin{equation*}\begin{split}   
-\Delta(\psi_1 P) &= -\psi_1 \Delta P - 2\ebdiv((\nabla \psi_1)P) + P\Delta\psi_1\\
 &= \psi_1\sum_{ij}\partial_i \partial_j A_{ij}  - 2\ebdiv((\nabla \psi_1)P) + P\Delta\psi_1\\
&= - \Delta P_{1,k} - \Delta P_{2,k} - \Delta P_3
\end{split}
\end{equation*}
where  $P_{1,k}$, $ P_{2,k} $ and $ P_3 $ are defined by
\begin{equation*}\begin{split} 
 - \Delta P_{1,k} &= \sum_{ij}\partial_i \partial_j ((\psi_1 -\psi_k) A_{ij}) \\
- \Delta P_{2,k} &=  \sum_{ij}\partial_i \partial_j (\psi_k A_{ij}) \\
 - \Delta P_3\mbox{ } &= - \sum_{ij}\partial_j[(\partial_i \psi_1)(A_{ij})]
  - \sum_{ij}\partial_i[(\partial_j \psi_1)(A_{ij})] \\&+ \sum_{ij}(\partial_i \partial_j \psi_1)(A_{ij})
  - 2\ebdiv((\nabla \psi_1)P) + P\Delta\psi_1. %\stackrel{\triangle}{=} I\mbox{ .}\\
\end{split}\end{equation*}
 $P_{1,k}$ and $P_3$ are defined by the representation formula 
${(-\Delta)}^{-1}(f) = \frac{1}{4\pi}(\frac{1}{|x|} * f)$\\
while $P_{2,k}$  by the Riesz transforms.\\
\\
Since $\psi_1 = 1 $ on $ B_{\frac{1}{3}}$, we have $\Delta P = \Delta(\psi_1 P)$ on $ B_{\frac{1}{3}}$. 
Thus \eqref{pressure_decomposition_expressition} holds.\\
\\
By definition of $P_{2,k}$, \eqref{lem_pressure_decomposition_p2k} holds.\\
\\
For \eqref{lem_pressure_decomposition_p1k} and \eqref{lem_pressure_decomposition_p3},
 it follows the proof of the lemma 3 of \cite{vas:partial} directly. 
For completeness, we present a proof here. Note that $ (\psi_1 -\psi_k) $ 
is supported in $ (B_{\frac{1}{6}} - B_{k-\frac{2}{3}} )$ and  
$ \nabla\psi_1 $ is supported in $ (B_{\frac{1}{6}} - B_{\frac{1}{3}} )$. Thus
for $x\in B_{k-\frac{1}{3}}$, 
\begin{equation*}\begin{split}
|P_{1,k}(x)| &= \Bigg|\frac{1}{4\pi}\int_{(B_{\frac{1}{6}} - B_{k-\frac{2}{3}} )}
\frac{1}{|x-y|}\sum_{ij}(\partial_i \partial_j ((\psi_1 -\psi_k) A_{ij}))(y)dy\Bigg|\\
&\leq \sup_{y\in B^C_{k-\frac{2}{3}}}(|\nabla^2_y\frac{1}{|x-y|}|) \cdot \sum_{ij}
\int_{B_{\frac{1}{6}}}|A_{ij}(x)|dy\\
&\leq C\cdot\sup_{y\in B^C_{k-\frac{2}{3}}}(\frac{1}{|x-y|^3}) \cdot 
\sum_{ij}\|A_{ij}\|_{L^1(B_{\frac{1}{6}})}
\leq C_1\cdot2^{9k} \cdot 
\sum_{ij}\|A_{ij}\|_{L^1(B_{\frac{1}{6}})}.
\end{split}\end{equation*}
We used integration by parts and  facts $|x-y| \geq 2^{-3k}$ and $ |(\psi_1 -\psi_k)|\leq 1 $ . \\
%so that we have \eqref{lem_pressure_decomposition_p1k}.\\
\\ In the same way, for $x\in B_{k-\frac{1}{3}}$,
\begin{equation*}\begin{split}
|\nabla P_{1,k}(x)| 
&\leq C_2\cdot2^{12k} \cdot 
\sum_{ij}\|A_{ij}\|_{L^1(B_{\frac{1}{6}})}.\\                 
\end{split}\end{equation*}
%Now we have \eqref{lem_pressure_decomposition_p1k}.\\

For $x\in B_{\frac{2}{3}}$, 
\begin{equation*}\begin{split}
|\nabla P_{3}(x)| 
= & \Bigg|\frac{1}{4\pi}\int_{(B_{\frac{1}{6}} - B_{\frac{1}{3}} )}
(\nabla_y\frac{1}{|x-y|})\Big[-\sum_{ij}\partial_j[(\partial_i \psi_1)(A_{ij})]
- \sum_{ij}\partial_i[(\partial_j \psi_1)(A_{ij})] \\ 
&\quad + \sum_{ij}(\partial_i \partial_j \psi_1)(A_{ij}))
- 2\ebdiv((\nabla \psi_1)P) + P\Delta\psi_1 \Big]dy\Bigg|\\
\leq& C_3\Big(\sum_{ij}\|A_{ij}\|_{L^1(B_{\frac{1}{6}})}
+ \|P\|_{L^1(B_{\frac{1}{6}})}\Big).
\end{split}\end{equation*}

These prove \eqref{lem_pressure_decomposition_p1k} and \eqref{lem_pressure_decomposition_p3} and we keep the 
constant  ${\Lambda_1} = max(C_1,C_2,C_3)$ for future use.

\end{proof}

%From the above lemma, it will be clear that $\nabla P_3$ has $L_t^1L_x^{\infty}$ bound
%in the following lemma

Before presenting De Giori arguments for 
large $r$ and small $r$, we need more notations. % for De Giori argument.
In the following two main lemmas \ref{lem_partial_1}
and \ref{lem_partial_2}, $P_3$ 
will be constructed from solutions $(u,P)$ for (Problem II-r) by
using the previous lemma \ref{lem_pressure_decomposition} and 
it will be clearly shown %in the following lemma \ref{lem_partial_1} 
that $\nabla P_3$ has $L_t^1L_x^{\infty}$ bound. Thus we can define%, for  integer $k\geq 0$,
%, define
\begin{equation}\begin{split}\label{def_e_k}
E_k(t) = &\frac{1}{2}(1-2^{-k}) + \int_{-1}^{t}\|\nabla P_3(s,\cdot)
\|_{L^{\infty}(B_{\frac{2}{3}})}ds, \\
& \mbox{ for } t \in[-2,0] \mbox{ and for } k\geq 0.
\end{split}\end{equation} %where, in the following lemma \ref{lem_partial_1}, $P_3$ 
%will be contrcuted from solutions $(u,P)$ for (Problem II-r) and 
%it will be clearly shown in the following lemma \ref{lem_partial_1} 
%that $\nabla P_3$ has $L_t^1L_x^{\infty}$ bound.
Note that $E_k$ depends on $t$.
% and
% it's well defined for solutions $(u,P)$ for (Problem II-r)
%because we will give $L_t^1L_x^{\infty}$ bound for $P_3$
%in the following lemma \ref{lem_partial_1}.\\
We also define followings like in \cite{vas:partial}
\begin{equation*}\begin{split}
v_k &= (|u|-E_k)_{+},\\
%v_k(t,x) &= (|u(t,x)|-E_k(t))_{+}\\
d_k & = \sqrt{\frac{E_k \mathbf{1}_{\{|u|\geq E_k\}}}{|u|}|\nabla|u||^2 + 
\frac{v_k}{|u|}|\nabla u|^2} \quad\mbox{ and}\\
U_k &= \sup_{t\in[T_k,0]}\Big(\int_{B_k}|v_k|^2 dx \Big) + \int\int_{Q_k}|d_k|^2 dx dt\\
&= \|v_k\|_{L^{\infty}(T_k,0;L^2(B_k))}^2 + \|d_k\|^2_{L^2(Q_k)}.
\end{split}\end{equation*}
In this way, $P_3$ will be absorbed into $v_k$, which is the key
idea of proof of
this proposition \ref{partial_problem_II_r}. \\

\subsection{De Giorgi argument to get a control for large $r$}\label{proof_lem_partial_1}

The following  big lemma says that
we can obtain a certain uniform non-linear estimate 
in the form of $W_k\leq C^k\cdot W_{k-1}^\beta $ when  $r$ is large.
Then an elementary lemma can give us the conclusion (we will see
the lemma \ref{lem_recursive} later).
On the other hand,
 for small $r$, we have the factor of $(1/r^3)$ which blows up
as $r$ goes to zero. This weak point implies that 
we still need some extra work after this lemma.
(it will be the next big lemma  \ref{lem_partial_2}).
\begin{lem}\label{lem_partial_1}
There exist universal constants ${\delta}_1>0$ and $\bar{C}_1>1$ such that 
if u is a solution of (Problem II-r) for some $0<r<\infty$ verifying both
\begin{equation*}\begin{split}
 &\| u\|_{L^{\infty}(-2,0;L^{2}(B(\frac{5}{4})))}+ 
\|P\|_{L^1(-2,0;L^{1}(B(1)))}+\| \nabla u\|_{L^{2}(-2,0;L^{2}(B(\frac{5}{4})))}
\leq {\delta}_1\\ \mbox{ and } 
&\| \mathcal{M}(|\nabla u|)\|_{L^{2}(-2,0;L^{2}(B(2)))}\leq {\delta}_1,\\
%&\| w\|_{L^2(-2,0;L^{\infty}(B(1)))}+\| \nabla w\|_{L^2(-2,0;L^{\infty(B(1)))}
%\leq \bar{\gamma}\cdot{\delta_{1}}\\
\end{split}\end{equation*} %for some $\bar{\gamma}\geq 1$,\\

 then we have 
\begin{equation*}
 U_k \leq \begin{cases}& (\bar{C}_1)^k U_{k-1}^{\frac{7}{6}} ,\quad \mbox{ for any } k\geq 1 \quad\mbox{ if } r\geq s_{1}\\
 &\frac{1}{r^3}\cdot (\bar{C}_1)^k U_{k-1}^{\frac{7}{6}} ,
\quad \mbox{ for any } k\geq 1  \quad\mbox{ if } r< s_{1}. \end{cases}
\end{equation*}
\end{lem}
\begin{rem}
 $s_1$ is a pre-fixed constant defined in 
 \eqref{def_s_k} such that $0<s_1<1/4$, and
 $({\delta}_1,\bar{C}_1)$ is independent of any 
$ 0<r<\infty$. It will be clear that 
the exponent $7/6$ is not optimal and we can make it close to $(4/3)$
arbitrarily. However, any exponent bigger than $1$ is enough for our study.
\end{rem}
\begin{proof}

We assume ${\delta_{1}} <1$. First we claim that there exists
a universal constant ${\Lambda_2}\geq 1$ 
%which is independent of any $ 0<r<\infty$
such that
\begin{equation}\label{w_iu_j}
\||w|\cdot|u|\|_{L^{2}(-2,0;L^{3/2}(B_{\frac{1}{6}}))}\leq {\Lambda_2}\cdot{\delta_{1}} \quad\mbox{for any } 0<r<\infty. 
\end{equation} 

%$\mathbb{CASE}$: $\mathbb{LARGE}$ $\mathbb{R}$ for $r\geq s_{1}$.\\

In order to prove the above claim \eqref{w_iu_j}, we need to separate it into
a large $r$ case and a small $r$ case:\\

\textbf{(I)-large $r$ case.} From the corollary \ref{convolution_cor} if $r\geq s_{1}$, then
\begin{equation}\begin{split}\label{w_large_r}
 \| w\|_{L^2(-4,0;L^{\infty}(B(2)))}%&\leq {C}
%\cdot(1+\frac{4}{r})^3\cdot\|\mathcal{M}(|\nabla u|)\|_{L^2{(Q(2))}}\\
&\leq {C}
\cdot(1+\frac{4}{s_{1}})^3\cdot\|\mathcal{M}(|\nabla u|)\|_{L^2{(Q(2))}}\\
&\leq {C}\|\mathcal{M}(|\nabla u|)\|_{L^2{(Q(2))}}\leq {C}{\delta_{1}}.
\end{split}\end{equation} Likewise, 
\begin{equation}\begin{split}\label{nabla_w_large_r}
 \|\nabla w\|_{L^2(-4,0;L^{\infty}(B(2)))}%&\leq {C}
%\cdot(1+\frac{4}{r})^3\cdot\|\mathcal{M}(|\nabla u|)\|_{L^2{(Q(2))}}\\
%&\leq {C}
%\cdot(1+\frac{4}{s_1})^3\cdot\|\mathcal{M}(|\nabla u|)\|_{L^2{(Q(2))}}\\
%&\leq {C}\|\mathcal{M}(|\nabla u|)\|_{L^2{(Q(2))}}\\
&\leq {C}{\delta_{1}}.
%^{\frac{1}{4}}\\
\end{split}\end{equation}
With Holder's inequality and $B_{\frac{1}{6}}\subset B_{0}=B(1)\subset B(\frac{5}{4})\subset B(2)$,
\begin{equation*}\begin{split}
\||w|\cdot|u|\|_{L^{2}(-2,0;L^{3/2}(B_{\frac{1}{6}}))}&\leq C
\| u\|_{L^{\infty}(-2,0;L^{2}(B(\frac{5}{4}))}\cdot\| w\|_{L^2(-4,0;L^{\infty}(B(2)))}\\
&\leq C  \cdot{\delta_{1}}^2\leq C_1 \cdot{\delta_{1}}.    
\end{split}\end{equation*} so we obtained \eqref{w_iu_j} for  $r\geq s_{1}$.\\

%$\mathbb{CASE}$: $\mathbb{SMALL}$ $\mathbb{R}$ for $r<s_{1}$.\\

\textbf{(II)-small $r$ case.} On the other hand, if $r< s_{1}$, then
\begin{equation}\begin{split}\label{w_small_r_with_r^3}
 \| w\|_{L^2(-4,0;L^{\infty}(B(2)))}&\leq {C}
\cdot(1+\frac{4}{r})^3\cdot\|\mathcal{M}(|\nabla u|)\|_{L^2{(Q(2))}}\\
%&\leq {C}
%\cdot(\frac{r+4}{r})^3\cdot\|\mathcal{M}(|\nabla u|)\|_{L^2{(Q(2))}}\\
&\leq {C}\frac{1}{r^3}\|\mathcal{M}(|\nabla u|)\|_{L^2{(Q(2))}}\leq {C}\frac{1}{r^3}{\delta_{1}}
\end{split}\end{equation} and
\begin{equation}\begin{split}\label{nabla_w_small_r_with_r^3}
 \|\nabla w\|_{L^2(-4,0;L^{\infty}(B(2)))}%&\leq {C}
%\cdot(1+\frac{4}{r})^3\cdot\|\mathcal{M}(|\nabla u|)\|_{L^2{(Q(2))}}\\
%&\leq {C}
%\cdot(\frac{r+4}{r})^3\cdot\|\mathcal{M}(|\nabla u|)\|_{L^2{(Q(2))}}\\
%&\leq {C}\frac{1}{r^3}\|\mathcal{M}(|\nabla u|)\|_{L^2{(Q(2))}}\\
&\leq {C}\frac{1}{r^3}{\delta_{1}}.
%^{\frac{1}{4}}\\
\end{split}\end{equation}

However, it is not enough to prove \eqref{w_iu_j} because  $\frac{1}{r^3}$ factor
 blows up as $r$ goes to zero. So, instead, we use the idea that $w$ and $u$ are similar if $r$ is small:

 \begin{equation*}\begin{split}
\| u\|_{L^{4}(-2,0;L^{3}(B_0))}
&\leq  
C\Big(\| u\|_{L^{\infty}(-2,0;L^{2}(B_0))}+
\| \nabla u\|_{L^{2}(-2,0;L^{2}(B_0))}\Big)
 \leq C{\delta_{1}} %\leq C{\delta_{1}}\\
 \end{split}\end{equation*}
and 
\begin{equation*}\begin{split}
% &\| u\|_{L^{\infty}(-2,0;L^{2}(B(1))}\leq {\delta_{1}}, \\
\| w^{\prime}\|_{L^{4}(-2,0;L^{3}(B_{\frac{1}{6}}))}
 &= \| u*\phi_r\|_{L^{4}(-2,0;L^{3}(B_{\frac{1}{6}}))}
 \leq\| u\|_{L^{4}(-2,0;L^{3}(B_0))}  \leq C{\delta_{1}} \\
 \end{split}\end{equation*}
 because $u*\phi_r$ in $B_{\frac{1}{6}}$ depends only on $u$  in $B_0$.
(recall that $r\leq s_1$ and $s_1$ is the distance $B^c_{0}$ and $B_{\frac{1}{6}}$
and refer \eqref{young}).   
 For $w^{\prime\prime}$,
 \begin{equation}\begin{split}\label{w_prime_prime_small_r}
%&\| w^{\prime\prime}\|_{L^2(-4,0;L^{\infty}(B(2)))}      \leq {\delta_{1}}\\                          
\| w^{\prime\prime}\|_{L^\infty(-2,0;L^{\infty}(B(2)))}
 &=\| w^{\prime\prime}\|_{L_t^\infty((-2,0))}\\
&=\|\int_{\mathbb{R}^3}\phi(y)(u*\phi_{r})(y)dy\|_{L_t^\infty((-2,0))}\\
&\leq C\|\|u*\phi_{r}\|_{L_x^1(B(1))}\|_{L_t^\infty((-2,0))}\\
&\leq C\|\|u\|_{L_x^1(B(\frac{5}{4}))}\|_{L_t^\infty((-2,0))}\\
&\leq {C}\| u\|_{L^\infty(-2,0;L^{2}(B(\frac{5}{4})))}\\
%&\leq {C}
%\Big(\| \nabla u\|_{L^2(-2,0;L^{2}(B(\frac{5}{4})))}
%+\| u\|_{L^{\infty}(-2,0;L^{2}(B(\frac{5}{4})))}\Big)\\
&\leq {C}{\delta_{1}}\\
\end{split}\end{equation} because $ w^{\prime\prime}$ is a constant in $x$
, $supp(\phi)\subset B(1)$ and $u*\phi_r$ in $B(1)$ depends only on $u$  
in $B(1+s_{1})$
 which is a subset of $B(\frac{5}{4})$.
As a result, we have
 \begin{equation}\begin{split}\label{w_u}
\||w|\cdot|u|\|_{L^{2}(-2,0;L^{3/2}(B_{\frac{1}{6}}))}&\leq C
\| u\|_{L^{4}(-2,0;L^{3}(B(1))}
\cdot\| w\|_{L^4(-2,0;L^{3}(B(\frac{1}{6})))}\\
&\leq C{\delta_{1}}\cdot
\| |w^{\prime}|+|w^{\prime\prime}|\|_{L^4(-2,0;L^{3}(B(\frac{1}{6})))}\\
%&\leq C{\delta_{1}}(\| w^{\prime}\|
%_{L^{\infty}(-2,0;L^{2}(B_{\frac{1}{6}}))}+
%\| w^{\prime\prime}\|_{L^2(-2,0;L^{\infty}(B(2)))})\\
&\leq C \cdot{\delta_{1}}^2 \leq C_2 \cdot{\delta_{1}}
\end{split}\end{equation} so that  we obtained \eqref{w_iu_j} for  $r\leq s_{1}$.\\
%Thus we proved \eqref{w_iu_j} for any $r\in(0,\infty)$.\\

\noindent
Hence, taking
 \begin{equation}\label{def_breve_C}
{\Lambda_2}=\max(C_1,C_2,1),
\end{equation} we have 
 \eqref{w_iu_j} and  ${\Lambda_2}$ is independent of  $0<r<\infty$
 as long as $\delta_1<1$.
From now on, we assume $\delta_1<1$ sufficiently
 small to be $10 \cdot {\Lambda_1}\cdot{\Lambda_2}\cdot\delta_1\leq 1/2$
 (Recall that ${\Lambda_1}$ comes from the lemma \ref{lem_pressure_decomposition}). \\

Thanks to the lemma \ref{lem_pressure_decomposition} and \eqref{w_iu_j}, by putting 
$A_{ij}=w_iu_j$ we can decompose $P$ by\\
\begin{equation*}%\label{eq_pressure_decomposition_expressition}
P = P_{1,k}  + P_{2,k} + P_{3}\quad\mbox{     in } [-2,0]\times B_{\frac{1}{3}} 
\end{equation*} for each $k\geq1$
with following properties:\\
\begin{equation}\begin{split}\label{eq_pressure_decomposition_p1k} 
 \| |\nabla P_{1,k}|
   + |P_{1,k}|\|_{L^{2}(-2,0;L^{\infty}(B_{k-\frac{1}{3}}))}  
&\leq {\Lambda_1}  2^{12k}\sum_{ij}\|w_iu_j\|_{L^{2}(-2,0;L^1(B_{\frac{1}{6}}))}\\
%&\leq {\Lambda_1}  2^{12k}\cdot 9\cdot\||w|\cdot|u|\|_{L^{2}(-2,0;L^1(B_{\frac{1}{6}}))}\\
  &\leq 9\cdot{\Lambda_1}\cdot{\Lambda_2} \cdot {\delta_{1}}\cdot2^{12k}
\leq 2^{12k} \quad\mbox{ for any } k\geq1,\\  
\end{split}\end{equation} 

\begin{equation}\label{eq_pressure_decomposition_p2k} 
-\Delta P_{2,k} = \sum_{ij} \partial_i \partial_j (\psi_k w_i u_j)
 \quad\quad \mbox{in } [-2,0]\times\mathbb{R}^3 \quad\mbox{ for any } k\geq1
\quad \mbox{ and} \\  
\end{equation} 
\begin{equation}\begin{split}\label{eq_pressure_decomposition_p3} 
\|\nabla P_{3}\|_{L^{1}(-2,0;L^{\infty}(B_{\frac{2}{3}}))} &\leq {\Lambda_1}
 \Big(\|P\|_{L^{1}(-2,0;L^1(B(1))} + \sum_{ij}\|w_iu_j\|_{L^{2}(-2,0;L^1(B(1))}\Big)\\
 &\leq {\Lambda_1}({\delta_{1}}+ 9\cdot{\Lambda_2}\cdot{\delta_{1}})
 \leq 10\cdot{\Lambda_1}\cdot {\Lambda_2}\cdot{\delta_{1}}\leq \frac{1}{2}. 
\end{split}\end{equation}
\noindent Note that the above \eqref{eq_pressure_decomposition_p3}
enables $E_k$ to be well-defined and it
satisfies $0\leq E_k \leq 1$ (see the definition of $E_k$ in \eqref{def_e_k}).\\

%Let us borrow  conclusions of 
%lemma 2, 4, 6 and the remark of lemma 6 in \cite{vas:partial} 
%without proof.\\

In the following remarks \ref{lem10_39}--\ref{rem_lem10_39}, we gather some easy results, which were obtained in \cite{vas:partial}, without proof. They can be found in 
the lemmas 
%2, 
4, 6 and the remark of the lemma 4 of \cite{vas:partial}. Note that any constants $C$ in the following remarks
 do not depend on $k$.\\% as long as $k\geq 0.\\

%\begin{rem}
%%For all $k\geq0$ and 
%or any $F\in {L^{\infty}(T_k,0;L^2(B_k))}$ with 
%\|\nabla F\|\in {L^{2}(T_k,0;L^2(B_k))}$, we have
%\begin{equation*}\begin{split}
%\|F\|_{L^{\frac{10}{3}}(Q_k)}\leq 
%C\Big(\|F\|_{L^{\infty}(T_k,0;L^2(B_k))} +\|\nabla F\|_{L^{2}%(T_k,0;L^2(B_k))}\Big).
%\end{split}\end{equation*}% where $C$ is independent of $k$.
%\end{rem}
\begin{rem}\label{lem10_39}
For any $k\geq 0$, the function $u$ can be decomposed by
%\begin{equation*}\begin{split}
%u=u\frac{v_k}{|u|} + u(1-\frac{v_k}{|u|})
%\end{split}\end{equation*} where
%\begin{equation*}\begin{split}
%\Big|u(1-\frac{v_k}{|u|})\Big|\leq 1.
%\end{split}\end{equation*}
$u=u\frac{v_k}{|u|} + u(1-\frac{v_k}{|u|})$.
Also we have
%$\Big|u(1-\frac{v_k}{|u|})\Big|\leq 1$ and
\begin{equation}\begin{split}\label{d_k}
&\Big|u(1-\frac{v_k}{|u|})\Big|\leq 1, 
\quad\frac{v_k}{|u|}|\nabla u|\leq d_k, \quad
\mathbf{1}_{|u|\geq E_k}|\nabla|u||\leq d_k,\\
&|\nabla v_k|\leq d_k \quad\mbox{ and }\quad
\big|\nabla\frac{uv_k}{|u|}\big|\leq 3d_k.
\end{split}\end{equation}
\end{rem}
\begin{rem}\label{lem12_39}
For any $k\geq1$ and 
for any $q\geq1$,
\begin{equation*}\begin{split}
\|\mathbf{1}_{v_k>0}\|_{L^q(Q_{k-1})}&\leq C 2^{\frac{10k}{3q}}U^{\frac{5}{3q}}_{k-1} \quad\mbox{ and }\quad
\|\mathbf{1}_{v_k>0}\|_{L^{\infty}(T_{k-1},0;L^q(Q_{k-1})}
\leq C 2^{\frac{2k}{q}}U^{\frac{1}{q}}_{k-1}.\\
\end{split}\end{equation*}% where $C$ is independent of $k$.
\end{rem}
\begin{rem}\label{rem_lem10_39}
%\begin{equation*}\begin{split}
%\|v_{k-1}\|_{L^{\frac{10}{3}}(Q_{k-1})}
%%\\&\leq C\Big(\|v_{k-1}\|_{L^{\infty}(T_{k-1},0;L^2(B_{k-1}))} 
%%+\|\nabla v_{k-1}\|_{L^{2}(T_{k-1},0;L^2(B_{k-1}))}\Big)\\
%%&\leq C\Big(\|v_{k-1}\|_{L^{\infty}(T_{k-1},0;L^2(B_{k-1}))} 
%+\|d_{k-1}\|_{L^{2}(T_{k-1},0;L^2(B_{k-1}))}\Big)
%\leq C U_{k-1}^{\frac{1}{2}}.
%\end{split}\end{equation*}
For any $k\geq1$,
$\|v_{k-1}\|_{L^{\frac{10}{3}}(Q_{k-1})}\leq C U_{k-1}^{\frac{1}{2}}.$
\end{rem}

\ \\

From the above remarks \ref{lem10_39}--\ref{rem_lem10_39}, we have for any $1\leq p\leq\frac{10}{3}$,
\begin{equation}\begin{split}\label{raise_of_power}
\|v_k\|_{L^{p}(Q_{k-1})}&=\|v_k\mathbf{1}_{v_k>0}\|_{L^{p}(Q_{k-1})}\\
&\leq \|v_k\|_{L^{\frac{10}{3}}(Q_{k-1})}\cdot\|\mathbf{1}_{v_k>0}\|_{L^{1/(\frac{1}{p}-\frac{3}{10})}(Q_{k-1})}\\
&\leq \|v_{k-1}\|_{L^{\frac{10}{3}}(Q_{k-1})}\cdot C 2^{\frac{10k}{3}\cdot(\frac{1}{p}-\frac{3}{10})}U^{\frac{5}{3}\cdot(\frac{1}{p}-\frac{3}{10})}_{k-1}\\
%&\leq C(2^{\frac{10}{3p}-1})^kU^{\frac{1}{2}}_{k-1}U^{\frac{5}{3p}-\frac{1}{2}}_{k-1}\\
&\leq C2^{\frac{7k}{3}}U^{\frac{5}{3p}}_{k-1}.\\
\end{split}\end{equation}
Likewise, for any $1\leq p\leq2$,
\begin{equation}\begin{split}\label{raise_of_power2} 
\|v_k\|_{L^{\infty}(T_{k-1},0;L^{p}(B_{k-1}))}%=\|v_k\mathbf{1}_{v_k>0}\|_{L^{\infty}(T_{k-1},0;L^{p}(B_{k-1}))}\\
%&\leq \|v_k\|_{L^{\infty}(T_{k-1},0;L^{2}(B_{k-1}))}\cdot\|\mathbf{1}_{v_k>0}\|_{L^{\infty}(T_{k-1},0;L^{1/(\frac{1}{p}-\frac{1}{2})}(B_{k-1}))}\\
%&\leq \|v_{k-1}\|_{L^{\infty}(T_{k-1},0;L^{2}(B_{k-1}))}\cdot C 2^{{2k}\cdot(\frac{1}{p}-\frac{1}{2})}U^{\frac{1}{p}-\frac{1}{2}}_{k-1}\\
%&\leq C2^{{2k}\cdot(\frac{1}{p}-\frac{1}{2})}U^{\frac{1}{2}}_{k-1}U^{\frac{1}{p}-\frac{1}{2}}_{k-1}\\
&\leq C2^{k}U^{\frac{1}{p}}_{k-1}
\end{split}\end{equation} and% for any $1\leq p\leq 2$,
\begin{equation}\begin{split}\label{raise_of_power3}
\|d_k\|_{L^{p}(Q_{k-1})}&%=\|d_k\mathbf{1}_{v_k>0}\|_{L^{p}(Q_{k-1})}\\
%&\leq \|d_k\|_{L^{2}(Q_{k-1})}\cdot\|\mathbf{1}_{v_k>0}\|_{L^{1/(\frac{1}{p}-\frac{1}{2})}(Q_{k-1})}\\
%&\leq \|d_{k-1}\|_{L^{2}(Q_{k-1})}\cdot C 2^{\frac{10k}{3}\cdot(\frac{1}{p}-\frac{1}{2})}U^{\frac{5}{3}\cdot(\frac{1}{p}-\frac{1}{2})}_{k-1}\\
%&\leq C2^{\frac{10k}{3}\cdot(\frac{1}{p}-\frac{1}{2})}U^{\frac{1}{2}}_{k-1}U^{\frac{5}{3p}-\frac{5}{6}}_{k-1}\\
\leq C2^{\frac{5k}{3}}U^{\frac{5}{3p}-\frac{1}{3}}_{k-1}.\\
\end{split}\end{equation}\\

%We will use \eqref{raise_of_power}, \eqref{raise_of_power2} and \eqref{raise_of_power3} many times.\\

Second, we claim that
 for every $k\geq 1$, functions $v_k$ verifies:
\begin{equation}\label{eq_suitable_inequality_for_v_k_prob_II_r}\begin{split}
 \partial_t \frac{v_k^2}{2} + &\ebdiv (w \frac{v_k^2}{2}) + d_k^2 - \Delta\frac{v_k^2}{2} \\
 &+\ebdiv (u (P_{1,k} + P_{2,k})) + (\frac{v_k}{|u|}-1)u\cdot\nabla (P_{1,k} + P_{2,k})  \leq 0
\end{split}\end{equation}
in  $(-2,0)\times B_{\frac{2}{3}}$.\\
 %Remember that everything is smooth from definition of (Problem II-r).\\
 %Proof for \eqref{eq_suitable_inequality_for_v_k_prob_II_r} 
 %is based on lemma 5 of \cite{vas:partial}.\\

\begin{rem}
Note that the above inequality \eqref{eq_suitable_inequality_for_v_k_prob_II_r} does not contain the $P_3$ term. We will see that this fact comes from
the definition of $E_k(t)$ in \eqref{def_e_k}.
\end{rem}

Indeed, observe that 
      $  \frac{v_k^2}{2} = \frac{|u|^2}{2} + \frac{v_k^2  -|u|^2}{2}$
        and note that $E_k$ does not depend on space variable but on time variable.
Thus we can compute,
for time derivatives,% $\partial_{t}$,
 \begin{equation*}\begin{split}             
 \partial_t&(\frac{v^2_k - |u|^2}{2}) = v_k\partial_{t}v_k - u\partial_{t}u
 = v_k\partial_{t}|u| - v_k\partial_t E_k - u\partial_{t}u \\
&=u (\frac{v_k}{|u|} -1)\partial_t u - v_k\partial_t E_k
= u (\frac{v_k}{|u|} -1)\partial_t u - v_k
\|\nabla P_3(t,\cdot)\|_{L^{\infty}(B_{\frac{2}{3}})}
 \end{split}\end{equation*} while, 
for any space derivatives $\partial_{\alpha}$,
\begin{equation*}\begin{split}             
 \partial_{\alpha}(\frac{v^2_k - |u|^2}{2}) %&= v_k\partial_{\alpha}v_k - u\partial_{\alpha}u\\
% &= v_k\partial_{\alpha}|u| - u\partial_{\alpha}u\\
% &= v_k\frac{u}{|u|}\partial_{\alpha}u - u\partial_{\alpha}u\\
& =u (\frac{v_k}{|u|} -1)\partial_{\alpha} u.
 \end{split}\end{equation*}
 Then we follow the same way as the lemma 5 of \cite{vas:partial} did:
First, we multiply \eqref{navier_Problem II-r} by $u (\frac{v_k}{|u|} -1)$,
and then we sum the result and \eqref{suitable_Problem II-r}. We omit the details which can be found in the proof of the lemma 5 of \cite{vas:partial}.
As a result,
 we have  

\begin{equation*}\begin{split}
 0 \geq & \quad\partial_t \frac{v_k^2}{2} + \ebdiv (w \frac{v_k^2}{2}) + d_k^2 
- \Delta\frac{v_k^2}{2} +v_k\|\nabla P_3(t,\cdot)\|_{L^{\infty}(B_{\frac{2}{3}})}\\
 &+\ebdiv (u P) + (\frac{v_k}{|u|}-1)u\cdot\nabla P \\   
%=&\quad\partial_t \frac{v_k^2}{2} + \ebdiv (w \frac{v_k^2}{2}) + d_k^2 - \Delta\frac{v_k^2}{2}  
%+v_k\|\nabla P_3(t,\cdot)\|_{L^{\infty}(B_{\frac{2}{3}})}\\
% &+\ebdiv (u (P_{1,k}+P_{2,k}+P_3)) + (\frac{v_k}{|u|}-1)u\cdot\nabla (P_{1,k}+P_{2,k}+P_3)\\  
=&\quad\partial_t \frac{v_k^2}{2} + \ebdiv (w \frac{v_k^2}{2}) + d_k^2 - \Delta\frac{v_k^2}{2}  
+\Big(v_k\|\nabla P_3(t,\cdot)\|_{L^{\infty}(B_{\frac{2}{3}})} + \frac{v_k}{|u|}u\cdot \nabla P_3\Big)\\
 &+\ebdiv (u (P_{1,k}+P_{2,k})) + (\frac{v_k}{|u|}-1)u\cdot\nabla (P_{1,k}+P_{2,k}). 
\end{split}\end{equation*}
For the last equality, we used the fact $P = P_{1,k}  + P_{2,k} + P_{3} $     in $ B_{\frac{1}{3}} $ and
\begin{equation}\label{easyp_3} 
\ebdiv (u P_3) + (\frac{v_k}{|u|}-1)u\cdot\nabla P_3= \frac{v_k}{|u|}u\cdot \nabla P_3.
\end{equation} 
%because everything is smooth.\\
%because $\ebdiv (u P_3)=u\nabla P_3 \in L_t^1L_x^{2}$ 
%on $(-2,0)\times B_{2/3}$ from
%\eqref{eq_pressure_decomposition_p3} makes the above computation possible.\\
Thus we proved the claim \eqref{eq_suitable_inequality_for_v_k_prob_II_r}
due to
\begin{equation*}   
 v_k\|\nabla P_3(t,\cdot)\|_{L^{\infty}(B_{\frac{2}{3}})} + \frac{v_k}{|u|}u\cdot \nabla P_3 \geq 0 
\quad \mbox{ on } (-2,0)\times B_{\frac{2}{3}}.
\end{equation*} 
For any integer $k$, we introduce a cut-off function $\eta_k(x)\in C^{\infty}(\mathbb{R}^3)$ satisfying
\begin{equation*}\begin{split}
&\eta_k = 1 \quad\mbox{ in } B_{k} \quad , \quad
\eta_k = 0 \quad\mbox{ in } B_{k-\frac{1}{3}}^C\quad , \quad
0\leq \eta_k \leq 1, \\
&|\nabla\eta_k|\leq C2^{3k}  \quad \mbox{and} \quad
|{\nabla}^2\eta_k|\leq C2^{6k}, 
\quad \mbox{ for }\mbox{any }x\in\mathbb{R}^3.
\end{split}\end{equation*}
Multiplying \eqref{eq_suitable_inequality_for_v_k_prob_II_r} by $\eta_k$
 and integrating $[\sigma,t]\times 
\mathbb{R}^3 $ for $T_{k-1}\leq \sigma\leq T_k\leq t \leq 0$,
 \begin{equation*}\begin{split}
&\int_{\mathbb{R}^3}\eta_k(x)\frac{|v_k(t,x)|^2}{2}dx
+ \int_{\sigma}^t\int_{\mathbb{R}^3}\eta_k(x)d^2_k(s,x)dxds\\
&\leq \int_{\mathbb{R}^3}\eta_k(x)\frac{|v_k(\sigma,x)|^2}{2}dx\\
+&\int_{\sigma}^t\int_{\mathbb{R}^3}(\nabla\eta_k)(x)w(s,x)\frac{|v_k(s,x)|^2}{2}dxds
+\int_{\sigma}^t\int_{\mathbb{R}^3}(\Delta\eta_k)(x)\frac{|v_k(s,x)|^2}{2}dxds\\
-&\int_{\sigma}^t\int_{\mathbb{R}^3}\eta_k(x)
\Big(\ebdiv (u (P_{1,k} + P_{2,k})) + (\frac{v_k}{|u|}-1)u
\cdot\nabla (P_{1,k} + P_{2,k})\Big)(s,x)dxds.
\end{split}\end{equation*}
Integrating in $\sigma\in[T_{k-1},T_k]$ and dividing by 
$-(T_{k-1}-T_k)=2^{-(k+1)}$, 
\begin{equation*}\begin{split}
&\sup_{t\in[T_k,1]}\Big(\int_{\mathbb{R}^3}\eta_k(x)\frac{|v_k(t,x)|^2}{2}dx
+ \int_{T_k}^t\int_{\mathbb{R}^3}\eta_k(x)d^2_k(s,x)dxds\Big)\\
&\leq 2^{k+1}\cdot \int_{T_{k-1}}^{T_k}\int_{\mathbb{R}^3}\eta_k(x)\frac{|v_k(\sigma,x)|^2}{2}dx\\
+&\int_{T_{k-1}}^{0}\Big|\int_{\mathbb{R}^3}\nabla\eta_k(x)w(s,x)\frac{|v_k(s,x)|^2}{2}dx\Big|ds
+\int_{T_{k-1}}^{0}\Big|\int_{\mathbb{R}^3}\Delta\eta_k(x)\frac{|v_k(s,x)|^2}{2}dx\Big|ds\\
+&\int_{T_{k-1}}^{0}\Big|\int_{\mathbb{R}^3}\eta_k(x)
\Big(\ebdiv (u (P_{1,k} + P_{2,k})) + (\frac{v_k}{|u|}-1)u
\cdot\nabla (P_{1,k} + P_{2,k})\Big)(s,x)dx\Big|ds.\\
\end{split}\end{equation*}
From $\eta_k=1$ on $ B_k$,
\begin{equation*}\begin{split}
U_k%=\quad \|v_k\|_{L^{\infty}(T_k,0;L^2(B_k))}^2 + \|d_k\|^2_{L^2(Q_k)}\\
&\leq\sup_{t\in[T_k,1]}\Big(\int_{\mathbb{R}^3}\eta_k(x)\frac{|v_k(t,x)|^2}{2}dx\Big)
+ \int_{T_k}^0\int_{\mathbb{R}^3}\eta_k(x)d^2_k(s,x)dxds\\
&\leq2\cdot\sup_{t\in[T_k,1]}\Big(\int_{\mathbb{R}^3}\eta_k(x)\frac{|v_k(t,x)|^2}{2}dx
+ \int_{T_k}^t\int_{\mathbb{R}^3}\eta_k(x)d^2_k(s,x)dxds\Big).
\end{split}\end{equation*}
Thus we have
\begin{equation}\begin{split}\label{1234_decompo_1}
&U_k%\leq C2^{6k}\int_{Q_{k-1}}|v_k(s,x)|^2dxds\\
%&+\int_{Q_{k-1}}|\nabla\eta_k(x)|\cdot|w(s,x)|\cdot|v_k(s,x)|^2dxds\\
%&+2\int_{T_{k-1}}^{0}\Big|\int_{\mathbb{R}^3}\eta_k(x)
%\Big(\ebdiv (u P_{1,k}) + (\frac{v_k}{|u|}-1)u\cdot\nabla P_{1,k}\Big)(s,x)dx\Big|ds\\
%&+2\int_{T_{k-1}}^{0}\Big|\int_{\mathbb{R}^3}\eta_k(x)
%\Big(\ebdiv (u P_{2,k}) + (\frac{v_k}{|u|}-1)u\cdot\nabla P_{2,k}\Big)(s,x)dx\Big|ds\\
\leq (I)+(II)+(III)+(IV)
\end{split}\end{equation} where
\begin{equation}\begin{split}\label{1234_decompo_2}
&(I)=C2^{6k}\int_{Q_{k-1}}|v_k(s,x)|^2dxds,\\
&(II)=\int_{Q_{k-1}}|\nabla\eta_k(x)|\cdot|w(s,x)|\cdot|v_k(s,x)|^2dxds,\\
&(III)=2\int_{T_{k-1}}^{0}\Big|\int_{\mathbb{R}^3}\eta_k(x)
\Big(\ebdiv (u P_{1,k}) + (\frac{v_k}{|u|}-1)u\cdot\nabla P_{1,k}\Big)(s,x)dx\Big|ds\quad\mbox{ and}\\
&(IV)=2\int_{T_{k-1}}^{0}\Big|\int_{\mathbb{R}^3}\eta_k(x)
\Big(\ebdiv (u P_{2,k}) + (\frac{v_k}{|u|}-1)u\cdot\nabla P_{2,k}\Big)(s,x)dx\Big|ds.
\end{split}\end{equation}
%Now our goal is to make a bound
%\begin{equation}
% (I)+(II)+(III)+(IV)\leq \begin{cases}& \bar{C}_1^k U_{k-1}^{7/6} ,\quad  \quad\mbox{ if } r\geq s_{1}\\
% &\frac{1}{r^3}\cdot \bar{C}_1^k U_{k-1}^{7/6} ,\quad  \quad\mbox{ if }
% r< s_{1}\\\end{cases}
%\end{equation} where $\bar{C}_1$ is independent of $r$ and $k$.\\

For $(I)$, by using \eqref{raise_of_power}, for any $0<r<\infty$,
\begin{equation}\begin{split}\label{(I)}
&(I)= C2^{6k}\|v_k\|^2_{L^{2}(Q_{k-1})}
%&\leq C2^{6k}(C2^{\frac{7k}{3}}U^{\frac{5}{3\cdot2}}_{k-1})^2\\
\leq C2^{10k}U^{\frac{5}{3}}_{k-1}.
\end{split}\end{equation}\\

For $(II)$ with $r\geq s_{1}$,  by using \eqref{w_large_r} and \eqref{raise_of_power2},
\begin{equation}\begin{split}\label{(II-1)}
(II)%=\|||\nabla\eta_k|\cdot||w|\cdot|v_k|^2\|_{L^1(Q_{k-1})}\\
%&\leq C2^{3k}\||w|\cdot|v_k|^2\|_{L^1(Q_{k-1})}\\
&\leq C2^{3k}\|w\|_{L^2(-4,0;L^{\infty}(B(2)))}\cdot
\||v_k|^2\|_{L^2(T_{k-1},0;L^{1}(B_{k-1}))}\\
&\leq C2^{3k}{\delta_{1}}\|v_k\|_{L^{\infty}(T_{k-1},0;L^{\frac{6}{5}}(B_{k-1}))}
\cdot\|v_k\|_{L^2(T_{k-1},0;L^{6}(B_{k-1}))}\\
%&\leq C2^{3k}{\delta_{1}}2^{k}U^{\frac{5}{6}}_{k-1}
%\cdot\|v_{k-1}\|_{L^2(T_{k-1},0;L^{6}(B_{k-1}))}\\
&\leq C2^{4k}{\delta_{1}}U^{\frac{5}{6}}_{k-1}
\cdot\Big(\|v_{k-1}\|_{L^{\infty}(T_{k-1},0;L^{2}(B_{k-1}))}+
\|\nabla v_{k-1}\|_{L^2(T_{k-1},0;L^{2}(B_{k-1}))}\Big)\\
%&\leq C2^{4k}{\delta_{1}}U^{\frac{5}{6}}_{k-1}
%\cdot\Big(\|v_{k-1}\|_{L^{\infty}(T_{k-1},0;L^{2}(B_{k-1}))}+
%\|d_{k-1}\|_{L^2(T_{k-1},0;L^{2}(B_{k-1}))}\Big)\\
&\leq C2^{4k}\cdot{\delta_{1}}\cdot U^{\frac{5}{6}}_{k-1}
\cdot U^{\frac{1}{2}}_{k-1}\leq C2^{4k}\cdot{\delta_{1}}\cdot U^{\frac{4}{3}}_{k-1}
\leq C2^{4k}\cdot U^{\frac{4}{3}}_{k-1}.
\end{split}\end{equation} 
 For  $r<s_{1}$, follow the above steps using
 \eqref{w_small_r_with_r^3} instead of  using \eqref{w_large_r} 
then we get
 \begin{equation}\begin{split}\label{(II-2)}
&(II)\leq C\frac{1}{r^3}2^{4k}\cdot U^{\frac{4}{3}}_{k-1}.
 \end{split}\end{equation} 

For $(III) $ (non-local pressure term), observe that 
\begin{equation*}\begin{split}
\ebdiv (u P_{1,k}) + (\frac{v_k}{|u|}-1)u\cdot\nabla P_{1,k}
= \frac{v_k}{|u|}u\cdot\nabla P_{1,k}
\end{split}\end{equation*} 
because everything is smooth.
%because
% $\ebdiv (u P_1)=u\nabla P_1 \in L_t^2L_x^{2}$ 
%on $(-2,0)\times B_{2/3}$ from
%\eqref{eq_pressure_decomposition_p1k} makes the above computation possible.\\
Thus, by using \eqref{eq_pressure_decomposition_p1k} 
and \eqref{raise_of_power},  for any $0<r<\infty$,
\begin{equation}\begin{split}\label{(III)}
(III)&\leq C\cdot\|\frac{v_k}{|u|}u\cdot\nabla P_{1,k}\|_{L^1(Q_{k-1})}
\leq C\||v_k|\cdot|\nabla P_{1,k}|\|_{L^1(Q_{k-1})}\\
&\leq\|v_k\|_{L^{2}(T_{k-1},0;L^{1}(B_{k-1})))}
\cdot\|\nabla P_{1,k}\|_{L^{2}(T_{k-1},0;L^{\infty}(B_{k-1}))}\\
&\leq\|\mathbf{1}_{v_k>0}\|_{L^{2}(T_{k-1},0;L^{2}(B_{k-1})))}
\|v_k\|_{L^{\infty}(T_{k-1},0;L^{2}(B_{k-1})))}
\cdot 2^{12k}\\ 
%&\leq C2^{\frac{7k}{3}}U^{\frac{5}{6}}_{k-1}
%\|v_{k-1}\|_{L^{\infty}(T_{k-1},0;L^{2}(B_{k-1})))}
%\cdot 2^{12k}\\ 
&\leq C2^{\frac{43k}{3}}U^{\frac{5}{6}}_{k-1}
U^{\frac{1}{2}}_{k-1}
\leq C2^{\frac{43k}{3}}U^{\frac{4}{3}}_{k-1}. 
\end{split}\end{equation}

For $(IV) $ (local pressure term), as we did for $(III)$, observe
\begin{equation*}\begin{split}
\ebdiv (u P_{2,k}) + (\frac{v_k}{|u|}-1)u\cdot\nabla P_{2,k}
= \frac{v_k}{|u|}u\cdot\nabla P_{2,k}.
\end{split}\end{equation*}% because everything is smooth.\\
By definition of $P_{2,k}$, we have
\begin{equation*}\begin{split}
-\Delta P_{2,k}& = \sum_{ij} \partial_i \partial_j (\psi_k w_i u_j)= \sum_{ij} \partial_i ((\partial_j \psi_k) w_i u_j
+\psi_k (\partial_jw_i) u_j)\\
&= \sum_{ij} \partial_i \Big((\partial_j \psi_k) w_i u_j(1-\frac{v_k}{|u|})
+(\partial_j \psi_k) w_i u_j\frac{v_k}{|u|}\\
&\quad\quad+\psi_k (\partial_jw_i) u_j(1-\frac{v_k}{|u|})
+\psi_k (\partial_jw_i) u_j\frac{v_k}{|u|}\Big)\\
\end{split}\end{equation*} and 
\begin{equation*}\begin{split}
-\Delta (\nabla P_{2,k})&= \sum_{ij} \partial_i\nabla
 \Big((\partial_j \psi_k) w_i u_j(1-\frac{v_k}{|u|})
+(\partial_j \psi_k) w_i u_j\frac{v_k}{|u|}\\
&\quad\quad+\psi_k (\partial_jw_i) u_j(1-\frac{v_k}{|u|})
+\psi_k (\partial_jw_i) u_j\frac{v_k}{|u|}\Big).
\end{split}\end{equation*}\\

\noindent Thus we can decompose $\nabla P_{2,k}$ by the Riesz transform into
\begin{equation*}\begin{split}
\nabla P_{2,k}= G_{1,k}+G_{2,k}+G_{3,k}+G_{4,k}
\end{split}\end{equation*} where 
\begin{equation*}\begin{split}
&G_{1,k} =\sum_{ij} (\partial_i\nabla)(-\Delta)^{-1}\Big(
(\partial_j \psi_k) w_i u_j(1-\frac{v_k}{|u|})\Big),\\
&G_{2,k} =\sum_{ij} (\partial_i\nabla)(-\Delta)^{-1}\Big(
(\partial_j \psi_k) w_i u_j\frac{v_k}{|u|}\Big),\\
&G_{3,k} =\sum_{ij} (\partial_i\nabla)(-\Delta)^{-1}\Big(
\psi_k (\partial_jw_i) u_j(1-\frac{v_k}{|u|})\Big)\quad\mbox{ and}\\
&G_{4,k} =\sum_{ij} (\partial_i\nabla)(-\Delta)^{-1}\Big(
\psi_k (\partial_jw_i) u_j\frac{v_k}{|u|}\Big).
\end{split}\end{equation*}

\noindent From $L^p$-boundedness of the Riesz transform with the fact
$supp(\psi_k)\subset B_{k-(5/6)}\subset B_{k-1}$, we have 
\begin{equation*}\begin{split} 
&\|G_{2,k}\|_{L^2(T_{k-1},0;L^2(\mathbb{R}^3))}
\leq C2^{3k}\|w\|_{L^2(T_{k-1},0;L^{\infty}(B_{k-1}))}
\cdot\|v_k\|_{L^{\infty}(T_{k-1},0;L^2(B_{k-1}))},\\%\mbox{ and}\\
&\|G_{4,k}\|_{L^2(T_{k-1},0;L^2(\mathbb{R}^3))}
\leq C\cdot\|\nabla w\|_{L^2(T_{k-1},0;L^{\infty}(B_{k-1}))}
\cdot\|v_k\|_{L^{\infty}(T_{k-1},0;L^2(B_{k-1}))}.\\
\end{split}\end{equation*}% and, 
 For any  $1<p<\infty$,
\begin{equation*}\begin{split} 
&\|G_{1,k}\|_{L^2(T_{k-1},0;L^p(\mathbb{R}^3))}
\leq C_p\cdot2^{3k}\|w\|_{L^2(T_{k-1},0;L^{\infty}(B_{k-1}))}\quad\mbox{ and}\\
&\|G_{3,k}\|_{L^2(T_{k-1},0;L^p(\mathbb{R}^3))}
\leq C_p\cdot\|\nabla w\|_{L^2(T_{k-1},0;L^{\infty}(B_{k-1}))}.
\end{split}\end{equation*} %where $C_p$ is a Riesz constant.\\ 
Therefore, by using \eqref{nabla_w_large_r} and \eqref{nabla_w_small_r_with_r^3}
\begin{equation*}
\||G_{2,k}|+|G_{4,k}|\|_{L^2(T_{k-1},0;L^2(\mathbb{R}^3))}\leq
\begin{cases} &C\cdot2^{3k}\cdot U^{\frac{1}{2}}_{k-1}
,\quad  \quad\mbox{ if } r\geq s_{1}\\
&C\cdot2^{3k}\cdot\frac{1}{r^3}\cdot U^{\frac{1}{2}}_{k-1}
,\quad  \quad\mbox{ if } r< s_{1}
\end{cases}\end{equation*} and, for any $1<p<\infty$,
\begin{equation*}
\||G_{1,k}|+|G_{3,k}|\|_{L^2(T_{k-1},0;L^p(\mathbb{R}^3))}\leq
\begin{cases} &C_p\cdot2^{3k}
,\quad  \quad\mbox{ if } r\geq s_{1}\\
&C_p\cdot2^{3k}\cdot\frac{1}{r^3}
,\quad  \quad\mbox{ if } r< s_{1}.
\end{cases}\end{equation*}
Thus, by using the above estimates
and \eqref{raise_of_power}, for $r\geq s_{1}$ and $p>5$,
\begin{equation*}\begin{split}
(IV)&\leq C\cdot\|\frac{v_k}{|u|}u\cdot\nabla P_{2,k}\|_{L^1(Q_{k-1})}
\leq C\||v_k|\cdot|\nabla P_{2,k}|\|_{L^1(Q_{k-1})}\\
&\leq C\||v_k|\cdot(|G_{1,k}|+|G_{3,k}|)\|_{L^1(Q_{k-1})}
+ C\||v_k|\cdot(|G_{2,k}|+|G_{4,k}|)\|_{L^1(Q_{k-1})}\\
&\leq\|v_k\|_{L^{2}(T_{k-1},0;L^{\frac{p}{p-1}}(B_{k-1})))}
\cdot\||G_{1,k}|+|G_{3,k}|\|_{L^{2}(T_{k-1},0;L^{p}(B_{k-1}))}\\
&\quad\quad+\|v_k\|_{L^{2}(T_{k-1},0;L^{2}(B_{k-1})))}
\cdot\||G_{2,k}|+|G_{4,k}|\|_{L^{2}(T_{k-1},0;L^{2}(B_{k-1}))}\\
%&\leq \|v_k\|_{L^{2}(T_{k-1},0;L^{\frac{2p}{p-2}}(B_{k-1})))}\cdot
%\|\mathbf{1}_{v_k>0}\|_{L^{\infty}(T_{k-1},0;L^{2}(B_{k-1})))}
%\cdot C_p\cdot2^{3k}\\
%&\quad\quad+ C2^{\frac{7k}{3}}U^{\frac{5}{6}}_{k-1}
%C\cdot2^{3k}\cdot U^{\frac{1}{2}}_{k-1}\\ 
%&\leq C_p\cdot2^{3k}\|v_k\|_{L^{\frac{2p}{p-2}}(T_{k-1},0;L^{\frac{2p}{p-2}}(B_{k-1})))}\cdot
% U^{\frac{1}{2}}_{k-1}
%+ C2^{\frac{16k}{3}}U^{\frac{4}{3}}_{k-1}\\ 
%&\leq C_p\cdot2^{3k} C2^{\frac{7k}{3}}U^{\frac{5(p-2)}{6p}}_{k-1}\cdot
% U^{\frac{1}{2}}_{k-1}
%+ C2^{\frac{16k}{3}}U^{\frac{4}{3}}_{k-1}\\ 
%&\leq C\cdot C_p\cdot2^{\frac{16k}{3}}U^{\frac{4p-5}{3p}}_{k-1}
%+ C2^{\frac{16k}{3}}U^{\frac{4}{3}}_{k-1}\\
&\leq C\cdot C_p\cdot 2^{\frac{16k}{3}}U^{\frac{4p-5}{3p}}_{k-1}.
\end{split}\end{equation*}

\noindent By the same way, for $r< s_{1}$ and $p>5$,
\begin{equation*}\begin{split}
(IV)&\leq  C\cdot C_p\cdot\frac{1}{r^3}
 2^{\frac{16k}{3}}U^{\frac{4p-5}{3p}}_{k-1}.
\end{split}\end{equation*}

\noindent Thus, by taking $p=10$,
\begin{equation}\label{(IV)}
(IV)\leq \begin{cases}& C\cdot
 2^{\frac{16k}{3}}U^{\frac{7}{6}}_{k-1}
 ,\quad  \quad\mbox{ if } r\geq s_{1}\\
& C\cdot\frac{1}{r^3}
 2^{\frac{16k}{3}}U^{\frac{7}{6}}_{k-1}
 ,\quad  \quad\mbox{ if } r< s_{1}.
\end{cases}\end{equation}

Finally, combining \eqref{(I)}, \eqref{(II-1)},
 \eqref{(II-2)}, \eqref{(III)} and \eqref{(IV)} gives us
 \begin{equation*} 
(I)+(II)+(III)+(IV)\leq \begin{cases}& C^k\cdot U^{\frac{7}{6}}_{k-1}
 ,\quad  \quad\mbox{ if } r\geq s_{1}\\
& \frac{1}{r^3}\cdot C^k\cdot U^{\frac{7}{6}}_{k-1}
 ,\quad  \quad\mbox{ if } r< s_{1}.
\end{cases}\end{equation*}

\end{proof}

\subsection{De Giorgi argument to get a control for small $r$}\label{proof_lem_partial_2}

The following  big lemma makes us be able to avoid
the weak point of the previous lemma \ref{lem_partial_1}
when we handle small $r$ including the case $r=0$.\\

Recall the definition of $s_k$ in \eqref{def_s_k} first. It is the distance 
between $B^c_{k-1}$ and $B_{k-\frac{5}{6}}$ 
and $s_k$ is strictly decreasing to zero as $k\rightarrow\infty$.
For any $0<r<s_{1}$  we
define $k_r$ as
the  integer such that $s_{k_r+1}< r\leq s_{k_r}$. 
%Note that % $k_r\geq 1$   
Note that 
$k_r$ is integer-valued, $k_r\geq 1$ 
 and is increasing to $\infty$ as $r$ goes to zero.
 For the case $r=0$, we simply
define $k_r=k_0=\infty$.\\

\begin{lem}\label{lem_partial_2}
% Let $k_0=k_0(r)$ be the largest integer such that $s_{k_0}\geq r$. Then  
There exist universal constants ${\delta}_2$ and $\bar{C}_2>1$  such that 
if u is a solution of (Problem II-r) for some $0\leq r<s_{1}$ verifying both
\begin{equation*}\begin{split}
 &\| u\|_{L^{\infty}(-2,0;L^{2}(B(\frac{5}{4}))}+ 
\|P\|_{L^1(-2,0;L^{1}(B(1))}+\| \nabla u\|_{L^{2}(-2,0;L^{2}(B(\frac{5}{4}))}
\leq {\delta}_2\\  
\mbox{ and } 
& \| \mathcal{M}(|\nabla u|)\|_{L^{2}(-4,0;L^{2}(B(2)))}\leq {\delta}_2,
%\mbox{ and } 
%&\| w\|_{L^2(-4,0;L^{\infty}(B(2)))}\leq \bar{\gamma}\cdot\bar{\delta},\\
\end{split}\end{equation*} 
then we have 
\begin{equation*}
U_k \leq  (\bar{C}_2)^k U_{k-1}^{\frac{7}{6}}\quad\text{ for any integer }
k \mbox{ such that } 1\leq k\leq k_r.%  = 1,2,\cdots,k_r.\\%1\leq k\leq l\\
\end{equation*}
\end{lem}
\begin{rem}
Note that ${\delta}_2$ and $\bar{C}_2$ are independent of any $r\in[0, s_1)$
and the exponent $7/6$ is not optimal and we can make it almost $(4/3)$.
\end{rem}
\begin{rem}
This lemma says that even though $r$ is very small, we can make the above
uniform estimate for the first few steps $k\leq k_r$. Moreover, the number $k_r$ of
these steps is increasing to the infinity with a certain rate as $r$ goes to zero. In the subsection \ref{combine_de_giorgi}, we will see that this rate is enough to obtain a uniform estimate
for any small $r$ once we combine two lemmas \ref{lem_partial_1}
and \ref{lem_partial_2}.
\end{rem}
\begin{proof}% We want ${\delta}_2$ and $\bar{C}_2$ to b
In this proof, we can borrow any inequalities 
in the proof of the previous lemma \ref{lem_partial_1}
 except those which depend on $r$ and blow up as $r$ goes to zero.\\% (having $\frac{1}{r^3}$ factor).\\

Let $0\leq r<s_{1}$ and take any integer k such that $1\leq k\leq k_r$. 
%By definition, $k_0\geq13$.\\
Like ${\delta}_1$ of the previous lemma \ref{lem_partial_1}, we assume ${\delta}_2$  so small that
\begin{equation*}\begin{split}
{\delta}_2<1,\quad 10{\Lambda_1} {\Lambda_2}{\delta}_2\leq\frac{1}{2}.
\end{split}\end{equation*}

We begin this proof by decomposing $w^\prime$ by
\begin{equation*}\begin{split}
w^\prime
=u*\phi_r
&=\Big(u(1-\frac{v_k}{|u|})\Big)*\phi_r 
+ \Big(u\frac{v_k}{|u|}\Big)*\phi_r =w^{\prime,1} + w^{\prime,2}.
\end{split}\end{equation*}
 Thus the advection velocity $w$ has a new decomposition: $w=w^{\prime} -w^{\prime\prime}=
(w^{\prime,1} +w^{\prime,2}) -w^{\prime\prime}
=(w^{\prime,1} -w^{\prime\prime}) +w^{\prime,2} $. 
We will verify that $w^{\prime,1} -w^{\prime\prime}$ is bounded and 
$w^{\prime,2}$ can be controlled locally.
First, for $w^{\prime,1}$,% and $\nabla w^\prime_1$, 
\begin{equation}\begin{split}\label{w_1_prime}
|w^{\prime,1}(t,x)|=\Big|\Big(\Big(u(1-\frac{v_k}{|u|})\Big)*\phi_r\Big)(t,x)\Big|
&\leq\|u(1-\frac{v_k}{|u|})(t,\cdot)\|_{L^{\infty}(\mathbb{R}^3)} \leq 1
\end{split}\end{equation} for any $-4\leq t $ and any $x\in\mathbb{R}^3$.
From \eqref{w_prime_prime_small_r},
we still have
\begin{equation}\begin{split}\label{w_prime_prime_small_r_again}
\| w^{\prime\prime}\|_{L^\infty(-2,0;L^{\infty}(B(2)))}\leq {C}\bar{\delta}\leq C.
\end{split}\end{equation}

\noindent Combining above two results,
\begin{equation}\begin{split}\label{w_1_prime_w_prime_prime_small_r}
\||w^{\prime,1}|+|w^{\prime\prime}|\|_{L^\infty(-2,0;L^{\infty}(B(2)))}\leq C.
\end{split}\end{equation}

\noindent For $w^{\prime,2}$, we observe that any $L^p$ 
norm of $w^{\prime,2}=\Big(u\frac{v_k}{|u|}\Big)*\phi_r$ 
in $B_{k-\frac{5}{6}}$ is less than or equal to that of $v_k$ in $B_{k-1}$ because
 $r\leq s_{k_r}\leq s_{k}$ and $s_k$ 
is the distance between $B^c_{k-1}$ and $B_{k-\frac{5}{6}}$ (see 
\eqref{young}). Thus we have, for any
$1\leq p\leq\infty$,
\begin{equation}\begin{split}\label{w_2_prime_small_r}
% &\| u\|_{L^{\infty}(-2,0;L^{2}(B(1))}\leq \bar{\delta}, \\
\|w^{\prime,2}&\|_{L^{p}(T_{k-1},0;L^{p}(B_{k-\frac{5}{6}}))}
 = \| \Big(u\frac{v_k}{|u|}\Big)*\phi_r\|_{L^{p}(T_{k-1},0;L^{p}(B_{k-\frac{5}{6}}))}\\
 &= \| |v_k|*\phi_r\|_{L^{p}(T_{k-1},0;L^{p}(B_{k-\frac{5}{6}}))}
\leq\| v_k\|_{L^{p}(Q_{k-1})}.%  \leq \bar{\delta} \\
 \end{split}\end{equation}% because $\Big(u\frac{v_k}{|u|}\Big)*\phi_r$ in $B_{k-\frac{5}{6}}$ depends only on $u$  in $B_{k-1}$.\\
 %(recall $s_1$=the distance $B^c_{0}$ and $B_{\frac{1}{6}}$)\\   
So, by using  \eqref{raise_of_power}, we have 
 \begin{equation}\begin{split}\label{raise_of_power_ w_prime_2}
%&\| w^{\prime}_1\|_{L^{\infty}(T_{k-1},0;L^{2}(B_{k-\frac{5}{6}}))}\leq U_{k-1}^{\frac{1}{2}}\\
%&\| w^{\prime}_1\|_{L^{\frac{10}{3}}(T_{k-1},0;L^{\frac{10}{3}}(B_{k-\frac{5}{6}}))}\leq U_{k-1}^{\frac{1}{2}}\\
&\| w^{\prime,2}\|_{L^{p}(T_{k-1},0;L^{p}(B_{k-\frac{5}{6}}))}
\leq C2^{\frac{7k}{3}}U^{\frac{5}{3p}}_{k-1}, 
\quad\mbox{ for any } 1\leq p\leq\frac{10}{3}.
%&\| w^{\prime,2}\|_{L^{\infty}(T_{k-1},0;L^{p}(B_{k-\frac{5}{6}}))}
%\leq C2^{k}U^{\frac{1}{p}}_{k-1}, \quad\mbox{ for any } 1\leq p\leq2\\
\end{split}\end{equation}

\begin{rem}
The above computations says that, for any small $r$,
the advection velocity $w$ can be decomposed into
one bounded part $(w^{\prime,1} -w^{\prime\prime})$ and the other part 
$w^{\prime,2}$, which has a good contribution
to the power of $U_{k-1}$. 
\end{rem}
Recall that the transpost term estimate \eqref{w_iu_j} is valid  for any
 $0<r<\infty$. Moreover, 
 the argument around  \eqref{w_u} says that  \eqref{w_iu_j} holds even for the case $r=0$.
%For the transport term, we have the same 
%\eqref{w_iu_j} from \eqref{w_u}.
 Thus, for any $r\in[0,s_1)$, we have the same pressure estimates
\eqref{eq_pressure_decomposition_p1k},
\eqref{eq_pressure_decomposition_p2k} and
\eqref{eq_pressure_decomposition_p3}. Thus
we can follow the proof of the previous lemma \ref{lem_partial_1} up to
\eqref{1234_decompo_1} without any single modification. 
It remains to  control  $(I)$--$(IV)$.  \\ 

 For $(I)$, \eqref{(I)} holds here too because \eqref{(I)} is independent of $r$.\\

 For $(II)$, by using \eqref{w_1_prime_w_prime_prime_small_r}
 and \eqref{raise_of_power_ w_prime_2} with
%the fact that $\eta_k$ is supported in $B_{k-\frac{1}{3}}\subset B_{k-\frac{5}{6}}$,
the fact  $supp(\eta_k)\subset B_{k-\frac{1}{3}}\subset B_{k-\frac{5}{6}}$, we have
\begin{equation}\begin{split}\label{second_(II)}
&(II)=\||\nabla\eta_k|\cdot|w|\cdot|v_k|^2\|_{L^1(Q_{k-1})}\\
&\leq C2^{3k}\Big(\|(|w^{\prime,1}|+|w^{\prime\prime}|)\cdot|v_k|^2\|_{L^1(Q_{k-1})}
+\||w^{\prime,2}|\cdot|v_k|^2\|_{L^{1}(T_{k-1},0;L^1(B_{k-\frac{5}{6}}))}\Big)\\
&\leq C2^{3k}\|v_k\|^2_{L^2(Q_{k-1})}
+C2^{3k}\|w^{\prime,2}\|_{L^{\frac{10}{3}}(T_{k-1},0;L^\frac{10}{3}(B_{k-\frac{5}{6}}))}\cdot
\||v_k|^2\|_{L^\frac{10}{7}(Q_{k-\frac{5}{6}})}\\
%&\leq C2^{3k}C2^{\frac{14k}{3}}U^{\frac{5}{3}}_{k-1}
%+C2^{3k} C2^{\frac{7k}{3}}U^{\frac{1}{2}}_{k-1}\cdot
%\||v_k|^2\|_{L^\frac{10}{7}(Q_{k-1})}\\
%&\leq C2^{\frac{23k}{3}}U^{\frac{5}{3}}_{k-1}
%+C2^{\frac{16k}{3}}U^{\frac{1}{2}}_{k-1}\cdot
%\|v_k\|^2_{L^\frac{20}{7}(Q_{k-1})}\\
%&\leq C2^{\frac{23k}{3}}U^{\frac{5}{3}}_{k-1}
%+C2^{\frac{16k}{3}}U^{\frac{1}{2}}_{k-1}\cdot
%C2^{\frac{14k}{3}}U^{\frac{7}{6}}_{k-1}\\
&\leq C2^{\frac{23k}{3}}U^{\frac{5}{3}}_{k-1}
+C2^{10k}U^{\frac{5}{3}}_{k-1}\leq C2^{10k}U^{\frac{5}{3}}_{k-1}.
\end{split}\end{equation}

 For  $(III)$(non-local pressure term),
 we have  \eqref{(III)} here too since \eqref{(III)} is independent of $r$.\\
 
 For $(IV)$(local pressure term), 
by definition of $P_{2,k}$ and decomposition of $w$,%$w=(w^{\prime,1}  -w^{\prime\prime})+w^{\prime,2}$, we have
\begin{equation*}\begin{split}
-\Delta P_{2,k}%& = \sum_{ij} \partial_i \partial_j (\psi_k w_i u_j)\\
&= \sum_{ij} \partial_i \partial_j\Big( \psi_k w_i u_j(1-\frac{v_k}{|u|})
+ \psi_k w_i u_j\frac{v_k}{|u|}\Big)\\
&= \sum_{ij} \partial_i \partial_j\Big(  \psi_k (w^{\prime,1}_i-w^{\prime\prime}_i) u_j(1-\frac{v_k}{|u|})+\psi_k w^{\prime,2}_i u_j(1-\frac{v_k}{|u|})\\
&\quad\quad+ \psi_k (w^{\prime,1}_i-
w^{\prime\prime}_i) u_j\frac{v_k}{|u|}+ \psi_k
 w^{\prime,2}_i u_j\frac{v_k}{|u|}\quad\Big).
\end{split}\end{equation*}

\noindent Thus we can decompose $ P_{2,k}$ by
\begin{equation*}\begin{split}
 P_{2,k}= P_{2,k,1}+P_{2,k,2}+P_{2,k,3}+P_{2,k,4}
\end{split}\end{equation*} where 
\begin{equation*}\begin{split}
&P_{2,k,1} =\sum_{ij} (\partial_i\partial_j)(-\Delta)^{-1}\Big(
\psi_k (w^{\prime,1}_i-w^{\prime\prime}_i) u_j(1-\frac{v_k}{|u|})\Big),\\
&P_{2,k,2} =\sum_{ij} (\partial_i\partial_j)(-\Delta)^{-1}\Big(
\psi_k w^{\prime,2}_i u_j(1-\frac{v_k}{|u|})\Big),\\
&P_{2,k,3} =\sum_{ij} (\partial_i\partial_j)(-\Delta)^{-1}\Big(
 \psi_k (w^{\prime,1}_i-w^{\prime\prime}_i) u_j\frac{v_k}{|u|}\Big)\quad\mbox{ and}\\
&P_{2,k,4} =\sum_{ij} (\partial_i\partial_j)(-\Delta)^{-1}\Big(
\psi_k w^{\prime,2}_i u_j\frac{v_k}{|u|}\Big).
\end{split}\end{equation*} %by Riesz transforms.
By using $\Big|u\Big(1-\frac{v_k}{|u|}\Big)\Big|\leq 1$ 
and the fact 
$\psi_k$ is supported in $B_{k-\frac{5}{6}}$ with
 \eqref{w_1_prime_w_prime_prime_small_r},
\begin{equation}\begin{split}\label{second_P_{2,k,1}}
\|P_{2,k,1}\|_{L^p(T_{k-1},0;L^p(\mathbb{R}^3))}&\leq C_p,\quad\mbox{ for } 1<p<\infty\\
\end{split}\end{equation} and, 
with
 \eqref{raise_of_power_ w_prime_2},
\begin{equation}\begin{split}\label{second_P_{2,k,2}}
\|P_{2,k,2}\|_{L^p(T_{k-1},0;L^p(\mathbb{R}^3))}&\leq
C_p\||\psi_k|\cdot |w^{\prime,2}|\|_{L^p(T_{k-1},0;L^{p}(\mathbb{R}^3)))}\\
%&\leq C_p\cdot\|w^{\prime,2}\|_{L^p(T_{k-1},0;L^{p}(B_{k-\frac{5}{6}}))}\\
%&\leq C_p\cdot\|v_k\|_{L^p(Q_{k-1})}\\
&\leq CC_p2^{\frac{7k}{3}}U^{\frac{5}{3p}}_{k-1}\quad\mbox{ for } 1\leq p\leq\frac{10}{3}.\\
\end{split}\end{equation}
Observe that for $i=1,2$, 
\begin{equation}\begin{split}\label{second_P_{2,k,1}_P_{2,k,2}}
&\ebdiv \Big(u G_i\Big) + \Big(\frac{v_k}{|u|}-1\Big)u\cdot\nabla G_i
=\ebdiv \Big(v_k\frac{u}{|u|}G_i\Big) - 
G_i\ebdiv\Big(\frac{uv_k}{|u|}\Big).
\end{split}\end{equation}
For $P_{2,k,1}$, by using \eqref{d_k}, \eqref{raise_of_power}, \eqref{raise_of_power3},
\eqref{second_P_{2,k,1}_P_{2,k,2}} and \eqref{second_P_{2,k,1}} with $p=10$
\begin{equation}\begin{split}\label{second_P_{2,k,1}_TOTAL}
&\int_{T_{k-1}}^{0}\Big|\int_{\mathbb{R}^3}\eta_k(x)
\Big(\ebdiv (u P_{2,k,1}) + (\frac{v_k}{|u|}-1)u\cdot\nabla P_{2,k,1}\Big)(s,x)dx\Big|ds\\
%&\leq 2C^{3k}\int_{Q_{k-1}}v_k|P_{2,k,1}| + 3d_k|P_{2,k,1}| dxds\\
&\leq C^{3k}\|v_k\cdot|P_{2,k,1}|\|_{L^1(Q_{k-1})} +
 3\|d_k\cdot|P_{2,k,1}|\|_{L^1(Q_{k-1})}\\
 &\leq C^{3k}\|v_k\|_{L^\frac{10}{9}(Q_{k-1})}
 \cdot\|P_{2,k,1}\|_{L^{10}(Q_{k-1})} +
 3\|d_k\|_{L^\frac{10}{9}(Q_{k-1})}\cdot\|P_{2,k,1}\|_{L^{10}(Q_{k-1})}\\
%  &\leq 2C^{3k} C2^{\frac{7k}{3}}U^{\frac{3}{2}}_{k-1}
%+
% 3C2^{\frac{5k}{3}}U^{\frac{3}{2}-\frac{1}{3}}_{k-1}\\
  &\leq C2^{\frac{16k}{3}}U^{\frac{3}{2}}_{k-1}+
 C2^{\frac{5k}{3}}U^{\frac{7}{6}}_{k-1}
 \leq C2^{\frac{16k}{3}}U^{\frac{7}{6}}_{k-1}.
\end{split}\end{equation}
Likewise, for $P_{2,k,2}$, by using \eqref{second_P_{2,k,2}} instead of \eqref{second_P_{2,k,1}}
\begin{equation}\begin{split}\label{second_P_{2,k,2}_TOTAL}
&\int_{T_{k-1}}^{0}\Big|\int_{\mathbb{R}^3}\eta_k(x)
\Big(\ebdiv (u P_{2,k,2}) + (\frac{v_k}{|u|}-1)u\cdot\nabla P_{2,k,2}\Big)(s,x)dx\Big|ds\\
%&\leq 2C^{3k}\int_{Q_{k-1}}v_k|P_{2,k,1}| + 3d_k|P_{2,k,1}| dxds\\
%&\leq 2C^{3k}\|v_k\cdot|P_{2,k,2}|\|_{L^1(Q_{k-1})} +
% 3\|d_k\cdot|P_{2,k,2}|\|_{L^1(Q_{k-1})}\\
% &\leq 2C^{3k}\|v_k\|_{L^\frac{10}{7}(Q_{k-1})}
% \cdot\|P_{2,k,2}\|_{L^{\frac{10}{3}}(Q_{k-1})} +
% 3\|d_k\|_{L^\frac{10}{7}(Q_{k-1})}\cdot\|P_{2,k,2}\|_{L^{\frac{10}{3}}(Q_{k-1})}\\
%  &\leq 2C^{3k}C2^{\frac{7k}{3}}U^{\frac{7}{6}}_{k-1}
% C2^{\frac{7k}{3}}U^{\frac{1}{2}}_{k-1}+
% 3C2^{\frac{5k}{3}}U^{\frac{7}{6}-\frac{1}{3}}_{k-1}
% C2^{\frac{7k}{3}}U^{\frac{1}{2}}_{k-1}\\
  &\leq C2^{\frac{23k}{3}}U^{\frac{5}{3}}_{k-1}+
 C2^{4k}U^{\frac{4}{3}}_{k-1}
 \leq C2^{\frac{23k}{3}}U^{\frac{4}{3}}_{k-1}.
\end{split}\end{equation}

\noindent From definitions of
$P_{2,k,3}$ and $P_{2,k,4}$
%$-\Delta(P_{2,k,3}+P_{2,k,4})= 
%\sum_{ij} \partial_i \partial_j\Big( \psi_k w_i u_j\frac{v_k}{|u|}\Big)$ 
with $\ebdiv(w)=0$, we have
%we can decompose $\nabla (P_{2,k,3}+P_{2,k,4})$ 
%by $\nabla (P_{2,k,3}+P_{2,k,4})=(H_{1,k}+H_{2,k}+H_{3,k}+H_{4,k})$ 
\begin{equation*}\begin{split}
-\Delta\nabla (P_{2,k,3}+P_{2,k,4})%&= \sum_{ij} \nabla\partial_i \partial_j
%\Big( \psi_k  w_i
% u_j\frac{v_k}{|u|}\Big)\\
&= \sum_{ij} \partial_i \partial_j\nabla\Big( \psi_k w_i u_j\frac{v_k}{|u|}\Big)\\
&= \sum_{ij} \nabla \partial_j
\Big(  (\partial_i\psi_k) w_i
u_j\frac{v_k}{|u|}+
\psi_k  w_i
\partial_i(u_j\frac{v_k}{|u|})\Big).
%&= \sum_{ij} \nabla \partial_j
%\Big(  (\partial_i\psi_k) (w^{\prime,1}_i-w^{\prime\prime}_i)
%_j\frac{v_k}{|u|}+
%(\partial_i\psi_k) w^{\prime,2}_i
%u_j\frac{v_k}{|u|}
%\\
%&+
%\psi_k (w^{\prime,1}_i-w^{\prime\prime}_i)
%\partial_i(u_j\frac{v_k}{|u|})
%
%\psi_k w^{\prime,2}_i
%\partial_i(u_j\frac{v_k}{|u|})\Big)\\
%&=-\Delta(H_{1,k}+H_{2,k}+H_{3,k}+H_{4,k})
\end{split}\end{equation*}

\noindent Then we use the fact 
$w=(w^{\prime,1} -w^{\prime\prime}) +w^{\prime,2} $
   so that
 we can decompose 
%$\nabla (P_{2,k,3}+P_{2,k,4})$ by 
\begin{equation*}
 \nabla(P_{2,k,3}+P_{2,k,4})=H_{1,k}+H_{2,k}+H_{3,k}+H_{4,k}
\end{equation*}
where
\begin{equation*}\begin{split}
H_{1,k}&= \sum_{ij}( \nabla \partial_j)(-\Delta)^{-1}
\Big((\partial_i\psi_k) (w^{\prime,1}_i-w^{\prime\prime}_i)
u_j\frac{v_k}{|u|} \Big),\\
H_{2,k}&= \sum_{ij}( \nabla \partial_j)(-\Delta)^{-1}
\Big((\partial_i\psi_k) w^{\prime,2}_i
u_j\frac{v_k}{|u|} \Big),\\
H_{3,k}&= \sum_{ij}( \nabla \partial_j)(-\Delta)^{-1}
\Big(\psi_k (w^{\prime,1}_i-w^{\prime\prime}_i)
\partial_i(u_j\frac{v_k}{|u|}) \Big)\quad\mbox{ and}\\
H_{4,k}&= \sum_{ij}( \nabla \partial_j)(-\Delta)^{-1}
\Big(\psi_k w^{\prime,2}_i
\partial_i(u_j\frac{v_k}{|u|}) \Big).
\end{split}\end{equation*}
By using $|u|\leq 1+v_k$,
\begin{equation}\begin{split}\label{second_P_{2,k,3}_P_{2,k,4}_TOTAL}
&\int_{T_{k-1}}^{0}\Big|\int_{\mathbb{R}^3}\eta_k(x)
\Big(\ebdiv (u (P_{2,k,3}+P_{2,k,4})) + (\frac{v_k}{|u|}-1)u\cdot
\nabla (P_{2,k,3}+P_{2,k,4})\Big)
%(s,x)
dx\Big|ds\\
%&\leq \int_{T_{k-1}}^{0}\int_{\mathbb{R}^3}|\nabla\eta_k(x)|
%\cdot |u(s,x)|\cdot|(P_{2,k,3}+P_{2,k,4})(s,x)|
%+|\eta_k(x)|\cdot|\nabla (P_{2,k,3}+P_{2,k,4})(s,x)|dxds\\
%&\leq 2C^{3k}\int_{Q_{k-1}}|u(s,x)|\cdot|(P_{2,k,3}+P_{2,k,4})(s,x)|+|\nabla (P_{2,k,3}+P_{2,k,4})(s,x)|dxds\\
&\leq C^{3k}\int_{Q_{k-1}}(1+v_k
%(s,x)
)\cdot|(P_{2,k,3}+P_{2,k,4})(s,x)|+|\nabla (P_{2,k,3}+P_{2,k,4})
%(s,x)
|dxds\\
%&\leq C^{3k}\int_{Q_{k-1}}|(P_{2,k,3}+P_{2,k,4})|+
%v_k\cdot|(P_{2,k,3}+P_{2,k,4})|+|\nabla (P_{2,k,3}+P_{2,k,4})|dxds\\
&\leq C^{3k}\Big(\|P_{2,k,3}\|_{L^1(Q_{k-1})}+\|v_k\cdot|P_{2,k,3}|\|_{L^1(Q_{k-1})}\\
&\quad\quad+\|P_{2,k,4}\|_{L^1(Q_{k-1})}
+\|v_k\cdot|P_{2,k,4}|\|_{L^1(Q_{k-1})}\\
&\quad\quad+\|H_{1,k}\|_{L^1(Q_{k-1})}+\|H_{2,k}\|_{L^1(Q_{k-1})}+\|H_{3,k}\|_{L^1(Q_{k-1})}
+\|H_{4,k}\|_{L^1(Q_{k-1})}\Big).\\
\end{split}\end{equation} 
From \eqref{raise_of_power} and \eqref{w_1_prime_w_prime_prime_small_r} with
the Riesz transform,
\begin{equation}\begin{split}\label{second_P_{2,k,3}_TOTAL}
\|P_{2,k,3}\|_{L^1(Q_{k-1})}&
\leq C\|P_{2,k,3}\|_{L^\frac{10}{9}(T_{k-1},0;L^\frac{10}{9}(\mathbb{R}^3))}
\leq C\|v_k\|_{L^\frac{10}{9}(Q_{k-1})}
 \leq C2^{\frac{7k}{3}}U^{\frac{3}{2}}_{k-1}. 
\end{split}\end{equation}

\noindent  Likewise
\begin{equation}\begin{split}\label{second_H_{1,k}_TOTAL}
\|H_{1,k}\|_{L^1(Q_{k-1})}&
%\leq C\|H_{1,k}\|_{L^\frac{10}{9}(T_{k-1},0;L^\frac{10}{9}(\mathbb{R}^3))}
%\leq C2^{3k}\|v_k\|_{L^\frac{10}{9}(Q_{k-1})}
 %\leq C2^{3k}2^{\frac{7k}{3}}U^{\frac{3}{2}}_{k-1}
 \leq C2^{\frac{16k}{3}}U^{\frac{3}{2}}_{k-1} 
\end{split}\end{equation} and 
\begin{equation}\begin{split}\label{second_v_k_P_{2,k,3}_TOTAL}
\|v_k\cdot|P_{2,k,3}|\|_{L^1(Q_{k-1})}
&\leq \|v_k\|_{L^2(Q_{k-1})}
\|P_{2,k,3}\|_{L^2(Q_{k-1})}\\
%&\leq  C2^{\frac{7k}{3}}U^{\frac{5}{6}}_{k-1}
%\|P_{2,k,3}\|_{L^2(T_{k-1},0;L^2(\mathbb{R}^3))}\\
%&\leq  C2^{\frac{7k}{3}}U^{\frac{5}{6}}_{k-1}
%\|v_k\|_{L^2(Q_{k-1})}\\
&\leq  C2^{\frac{7k}{3}}U^{\frac{5}{6}}_{k-1}\cdot
 C2^{\frac{7k}{3}}U^{\frac{5}{6}}_{k-1}
 \leq  C2^{\frac{14k}{3}}U^{\frac{5}{3}}_{k-1}.\\
\end{split}\end{equation}
Using \eqref{raise_of_power}, \eqref{raise_of_power_ w_prime_2},
\eqref{d_k} and \eqref{raise_of_power3}, we have
\begin{equation}\begin{split}\label{second_P_{2,k,4}_TOTAL}
\|P_{2,k,4}\|_{L^1(Q_{k-1})}
%\leq C\|P_{2,k,4}\|_{L^\frac{10}{9}(T_{k-1},0;L^\frac{10}{9}(\mathbb{R}^3))}\\
%&\leq\| w^{\prime,2}\|_{L^{\frac{10}{3}}(T_{k-1},0;L^{\frac{10}{3}}(B_{k-\frac{5}{6}}))}
%\cdot\| v_k\|_{L^{\frac{5}{3}}(Q_{k-1})}\\
%&\leq C2^{\frac{7k}{3}}U^{\frac{1}{2}}_{k-1}\cdot
%C2^{\frac{7k}{3}}U^{1}_{k-1}\\ 
&\leq C2^{\frac{14k}{3}}U^{\frac{3}{2}}_{k-1},
\end{split}\end{equation}
\begin{equation}\begin{split}\label{second_H_{2,k}_TOTAL}
\|H_{2,k}\|_{L^1(Q_{k-1})}
%\leq C\|H_{2,k}\|_{L^\frac{10}{9}(T_{k-1},0;L^\frac{10}{9}(\mathbb{R}^3))}\\
%&\leq 2^{3k}\| w^{\prime,2}\|_{L^{\frac{10}{3}}(T_{k-1},0;L^{\frac{10}{3}}(B_{k-\frac{5}{6}}))}
%\cdot\| v_k\|_{L^{\frac{5}{3}}(Q_{k-1})}\\
%&\leq C2^{3k}2^{\frac{7k}{3}}U^{\frac{1}{2}}_{k-1}\cdot
%2^{\frac{7k}{3}}U^{1}_{k-1}\\ 
%&\leq C2^{3k}2^{\frac{14k}{3}}U^{\frac{3}{2}}_{k-1}
&\leq C2^{\frac{23k}{3}}U^{\frac{3}{2}}_{k-1},
\end{split}\end{equation}
%By using \eqref{raise_of_power} and \eqref{raise_of_power_ w_prime_2},
\begin{equation}\begin{split}\label{second_v_k_P_{2,k,4}_TOTAL}
\|v_k\cdot|P_{2,k,4}|\|_{L^1(Q_{k-1})}
%&\leq \|v_k\|_{L^{\frac{10}{3}}(Q_{k-1})}
%\|P_{2,k,4}\|_{L^{\frac{10}{7}}(Q_{k-1})}\\
%&\leq  C2^{\frac{7k}{3}}U^{\frac{1}{2}}_{k-1}
%\| w^{\prime,2}\|_{L^{\frac{10}{3}}(T_{k-1},0;L^{\frac{10}{3}}(B_{k-\frac{5}{6}}))}
%\cdot\| v_k\|_{L^{\frac{10}{4}}(Q_{k-1})}\\
%&\leq  C2^{\frac{7k}{3}}U^{\frac{1}{2}}_{k-1}\cdot
% C2^{\frac{7k}{3}}U^{\frac{1}{2}}_{k-1}\cdot 
%  C2^{\frac{7k}{3}}U^{\frac{2}{3}}_{k-1}\\
 &\leq  C2^{\frac{21k}{3}}U^{\frac{5}{3}}_{k-1},
\end{split}\end{equation}
%By using \eqref{d_k} and \eqref{raise_of_power3},
\begin{equation}\begin{split}\label{second_H_{3,k}_TOTAL}
\|H_{3,k}\|_{L^1(Q_{k-1})}
%\leq C\|H_{3,k}\|_{L^\frac{10}{9}(T_{k-1},0;L^\frac{10}{9}(\mathbb{R}^3))}\\
%&\leq C\| 3d_k\|_{L^{\frac{10}{9}}(Q_{k-1})}
%\leq C3C2^{\frac{5k}{3}}U^{\frac{3}{2}-\frac{1}{3}}_{k-1} 
&\leq C2^{\frac{5k}{3}}U^{\frac{7}{6}}_{k-1} 
\end{split}\end{equation} and
%By using \eqref{d_k}
\begin{equation}\begin{split}\label{second_H_{4,k}_TOTAL}
\|H_{4,k}\|_{L^1(Q_{k-1})}
%\leq C\|H_{4,k}\|_{L^\frac{10}{9}(T_{k-1},0;L^\frac{10}{9}(\mathbb{R}^3))}\\
%&\leq \| w^{\prime,2}\|_{L^{\frac{10}{3}}(T_{k-1},0;L^{\frac{10}{3}}(B_{k-\frac{5}{6}}))}
%\|3d_k\|_{L^{\frac{10}{6}}(Q_{k-1})}\\
%&\leq C2^{\frac{7k}{3}}U^{\frac{1}{2}}_{k-1}\cdot
%C2^{\frac{5k}{3}}U^{1-\frac{1}{3}}_{k-1}
&\leq C2^{4k}U^{\frac{7}{6}}_{k-1}.
\end{split}\end{equation}
Combining \eqref{second_P_{2,k,1}_TOTAL},
\eqref{second_P_{2,k,2}_TOTAL} and
\eqref{second_P_{2,k,3}_P_{2,k,4}_TOTAL} together with
\eqref{second_P_{2,k,3}_TOTAL}, $\cdots$, 
%\eqref{second_H_{1,k}_TOTAL},
%\eqref{second_v_k_P_{2,k,3}_TOTAL},
%\eqref{second_P_{2,k,4}_TOTAL},
%\eqref{second_H_{2,k}_TOTAL},
%\eqref{second_v_k_P_{2,k,4}_TOTAL},
%\eqref{second_H_{3,k}_TOTAL} and
\eqref{second_H_{4,k}_TOTAL}, we obtain % we have 
\begin{equation}\begin{split}\label{second_(IV)}
&(IV)\leq C 2^{\frac{23k}{3}} U^{\frac{7}{6}}_{k-1}.
%\leq C^k U^{\frac{7}{6}}_{k-1}
\end{split}\end{equation}

Finally we combine  \eqref{second_(II)} and \eqref{second_(IV)} together with \eqref{(I)} and \eqref{(III)} in the previous lemma
in order to finish the proof of this lemma \ref{lem_partial_2}.
%\begin{equation}\begin{split}
%&U_{k-1}
%%\leq C 2^{\frac{43k}{3}} U^{\frac{7}{6}}_{k-1}
%\leq (\bar{C}_2)^k U^{\frac{7}{6}}_{k-1}
%%\quad\text{ for }k = 1,2,\cdots,k_r
%\quad\text{ for any integer }
%k \mbox{ such that } 1\leq k\leq k_r
% \text{ and for any  } r\in [0,s_{1}).
%\end{split}\end{equation}% for $k = 1,2,\cdots,k_0$ and for $r<s_1$.\\
\end{proof}

\subsection{Combining the two De Giorgi arguments}\label{combine_de_giorgi}
First we present one small lemma. Then the actual proof
of the proposition \ref{partial_problem_II_r} will follow.
The following small lemma says that certain non-linear estimates
give zero limit if the initial term  is sufficiently small.
This fact is one of key arguments of De Giorgi method.
\begin{lem}\label{lem_recursive}
 Let $C>1$ and $\beta>1$. Then there exists a constant $C_0^{*}$ 
such that for every sequence verifying both $ 0 \leq W_0 < C^{*}_0$ 
and \begin{equation*}
 0\leq W_{k} \leq  C^k \cdot W_{k-1}^{\beta} \quad \mbox{ for any } k\geq 1,
\end{equation*} we have $\lim_{k \to\infty} W_k = 0$.
%\begin{equation*}
 %        \lim_{k \to\infty} W_k = 0
  %      \end{equation*}
%Note that $C_0^{*} = C^{\frac{-\beta}{(\beta-1)^2}}$ is enough.\\

%Note that \begin{equation}\begin{split}\label{recursive_condition}
%C_0^{*}=C^{\frac{-\beta}{(\beta-1)^2}}\cdot \alpha^{\frac{-1}{(\beta-1)}}
%\end{split}\end{equation} is enough.

\end{lem}
\begin{proof}
It is quite standard or see the lemma 1 in \cite{vas:partial}.
\end{proof}

Finally we are ready to prove the proposition \ref{partial_problem_II_r}.
\begin{proof}[Proof of proposition \ref{partial_problem_II_r}]
Suppose that u is a solution of (Problem II-r) for some $0\leq r<\infty$ verifying 
\begin{equation*}\begin{split}
 &\| u\|_{L^{\infty}(-2,0;L^{2}(B(\frac{5}{4})))}+ 
\|P\|_{L^1(-2,0;L^{1}(B(1)))}+\| \nabla u\|_{L^{2}(-2,0;L^{2}(B(\frac{5}{4})))}
\leq {\delta}\\ \mbox{ and } 
&\| \mathcal{M}(|\nabla u|)\|_{L^{2}(-4,0;L^{2}(B(2)))}\leq {\delta}\\
\end{split}\end{equation*} where $\delta$ will be chosen within the proof.\\

From two big lemmas \ref{lem_partial_1} and \ref{lem_partial_2} by assuming
$\delta\leq\min({\delta}_1,\delta_2)$, we have
\begin{equation}\label{before_combine}
 U_k \leq 
\begin{cases}& (\bar{C}_1)^k U_{k-1}^{\frac{7}{6}} ,
\quad \mbox{ for any  } k\geq 1 \quad\mbox{ if } r\geq s_{1}.\\
 &\frac{1}{r^3}\cdot (\bar{C}_1)^k U_{k-1}^{\frac{7}{6}} 
,\quad \mbox{ for any  } k\geq 1  \quad\mbox{ if } 0<r< s_{1}.\\%\quad\mbox{ and}\\
&(\bar{C}_2)^k U_{k-1}^{\frac{7}{6}} 
\quad\text{ for } k = 1,2,\cdots,k_r \quad\mbox{ if } 0\leq r< s_{1}.
%\quad\text{ for any integer } k \mbox{ such that } 1\leq k\leq k_r
 %\quad\mbox{ and any } 0\leq r< s_{1}.
\end{cases}
\end{equation}
Note that  $k_r=\infty$ if $r=0$. Thus we can combine
the case $r=0$ with the case $r\geq s_1$ into one estimate:
%This estimate says that 
%for the case $r>s_1$ or $r=0$, we 
%do not have any problem to  use the recursive lemma \ref{lem_recursive}. 
\begin{equation*}
 U_k \leq 
 (\bar{C}_3)^k U_{k-1}^{\frac{7}{6}} 
\quad \mbox{ for any } k\geq 1 \quad\mbox{ if either } r\geq s_1 \mbox{ or } r=0.
\end{equation*} where we define $\bar{C}_3=\max(\bar{C}_1,\bar{C}_2)$.\\
%(Recall $k_r=\infty$ if $r=0$.)\\

We consider now the case $0<r<s_1$. % needs one more work.
%Let's assume $0<r<s_1$ for a moment. 
Recall that
$s_k=D\cdot 2^{-3k}$ where $D=\Big((\sqrt{2}-1)2\sqrt{2}\Big)>1$
and $s_{k_r+1}< r\leq s_{k_r}$ for any $r\in(0,s_1)$. It gives us
$r\geq D\cdot2^{-3(k_r+1)}$. 
  Thus if $k\geq k_r$ and if $0<r< s_1$, then the second line in \eqref{before_combine} becomes
\begin{equation}\begin{split}
 U_k 
&\leq\frac{1}{r^3}\cdot (\bar{C}_1)^k U_{k-1}^{\frac{7}{6}}
\leq\frac{2^{9{(k_r+1)}}}{D^3}\cdot (\bar{C}_1)^k U_{k-1}^{\frac{7}{6}}\\
&\leq{2^{9{(k+1)}}}\cdot (\bar{C}_1)^k U_{k-1}^{\frac{7}{6}}
%\leq\frac{2^{9}}{D^3}\cdot (2^9\cdot\bar{C}_1)^k U_{k-1}^{\frac{7}{6}}\\
\leq({2^{18}}\cdot\bar{C}_1)^k U_{k-1}^{\frac{7}{6}}.
\end{split}\end{equation}
So we have for any $r\in(0,s_1)$,
\begin{equation*}%\label{middle_combine}
 U_k \leq 
\begin{cases}%& (\bar{C}_1)^k U_{k-1}^{\frac{7}{6}} ,
%\quad \mbox{ for any } k\geq 1 \quad\mbox{ if } r\geq s_{1},\\
 & ({2^{18}}\cdot\bar{C}_1)^k U_{k-1}^{\frac{7}{6}} 
,\quad \mbox{ for any } k\geq k_r. \\%\quad\mbox{ if } r< s_{1}\quad\mbox{ and}\\
&(\bar{C}_2)^k U_{k-1}^{\frac{7}{6}}\quad\text{ for }
k = 1,2,\cdots,k_r.% \quad\mbox{ if } r< s_{1}.
\end{cases}
\end{equation*}
Define $\bar{C}=\max({2^{18}}\cdot\bar{C}_1,\bar{C}_2,\bar{C}_3)
=\max({2^{18}}\cdot\bar{C}_1,\bar{C}_2)$. Then we can combine
all three cases $r=0$, $0<r<s_1$ and $s_1\leq r<\infty$ into one uniform estimate:
 %and we can combine \eqref{middle_combine} all together:
\begin{equation*}%\label{after_combine}
 U_k \leq 
 (\bar{C})^k U_{k-1}^{\frac{7}{6}} 
\quad \mbox{ for any } k\geq 1 \quad\mbox{ and for any } 0\leq r<\infty.
\end{equation*}

Finally, by using the recursive lemma 
\ref{lem_recursive}, we obtain $C^{*}_0$ such  that $U_k\rightarrow 0$ as
$k\rightarrow 0$ whenever $ U_0 < C^{*}_0$. This condition $ U_0 < C^{*}_0$ %which can be obtained if
is achievable once
we assume $\delta$ so small that 
%So it's enough to take 
${\delta}\leq\sqrt{\frac{C^{*}_0}{2}}$
%${\delta}=\min(\sqrt{\frac{C^{*}_0}{2}},\delta_1,\delta_2)$ 
because
\begin{equation*}\begin{split}
U_0%&\leq \| u\|_{L^{\infty}(-2,0;L^{2}(B(\frac{5}{4})))}^2+ 
%\| \nabla u\|_{L^{2}(-2,0;L^{2}(B(\frac{5}{4})))}^2\\
%&\leq \Big(\| u\|_{L^{\infty}(-2,0;L^{2}(B(\frac{5}{4})))}+ 
%\| \nabla u\|_{L^{2}(-2,0;L^{2}(B(\frac{5}{4})))}\big)^2\\
&\leq \Big(\| u\|_{L^{\infty}(-2,0;L^{2}(B(\frac{5}{4})))}+ 
\|P\|_{L^1(-2,0;L^{1}(B(1)))}+
\| \nabla u\|_{L^{2}(-2,0;L^{2}(B(\frac{5}{4})))}\big)^2.
\end{split}\end{equation*} %As a result, $U_k$ converges to 0.\\
Thus we fix ${\delta}=\min(\sqrt{\frac{C^{*}_0}{2}},\delta_1,\delta_2)$ 
which does not depend on any $r\in [0,\infty)$. 
Observe that  for any $k\geq 1$,
\begin{equation*}
 \sup_{-\frac{3}{2}\leq t\leq 0}\int_{B(\frac{1}{2})}(|u(t,x)|-1)^2_{+}dx \leq U_k
\end{equation*} from $E_k\leq 1 $ and
 $ (-\frac{3}{2},0)\times B(\frac{1}{2})\subset Q_k$. Due to the fact
$U_k\rightarrow 0$, the conclusion of this proposition \ref{partial_problem_II_r}  follows.
 %$|u(t,x)|\leq 1$
%on $[-\frac{3}{2},0]\times B(\frac{1}{2})$.
\end{proof}
\section{Proof of the second local study proposition \ref{local_study_thm}}\label{proof_local study}

First we present technical lemmas, whose proofs will be given in the appendix.
In the subsection \ref{new_step1_2_together}, it will be explained how to apply the previous local study
proposition \ref{partial_problem_II_r} in order to get a $L^\infty$-bound of 
the velocity $u$ . Then, the subsections \ref{new_step3} and
\ref{new_step4} will give us $L^\infty$-bounds for classical derivatives $\nabla^d u$
and for fractional derivatives $(-\Delta)^{\alpha/2}\nabla^d u$, respectively.
\subsection{Some lemmas}

The following lemma is an estimate about higher derivatives of pressure
which we can find a similar lemma 
%5 
in \cite{vas:higher}. However they are different in the sense that
%Basically difference is that
here we require $(n-1)$th order norm of $v_1$ to control $n$th 
derivatives of pressure (see \eqref{ineq_parabolic_pressure}) while in \cite{vas:higher} we require one more order,
i.e. $n$th order. This fact follows the divergence structure and it 
will be useful for a bootstrap argument in the 
subsection \ref{new_step3}
%in  stage 2 of the proof of the theorem \ref{local_study_thm} to prove
 when large $r$ is large (we will see \eqref{step3_second_claim}). 
\begin{lem}\label{higher_pressure} Suppose that we have  $v_1,v_2\in 
(C^\infty(B(1)))^3$ with $\ebdiv v_1=\ebdiv v_2=0$ and $P \in 
C^\infty(B(1))$  which satisfy
 \begin{equation*}\begin{split} 
-\Delta P &=\ebdiv\ebdiv(v_2\otimes v_1)
%=\sum_{ij}\partial_i \partial_j (v_{2,i}v_{1,j})\\
\end{split}\end{equation*}
on $B(1)\subset \mathbb{R}^3$.\\% for some $0<a<\infty$.\\
%$\mathbb{R}^3$.\\
%as distribution.\\

\noindent Then, for  any $n\geq 2$, $0<b<a<1$ and $1<p<\infty$, we have the
two following estimates:
\begin{equation}\begin{split}\label{ineq_parabolic_pressure} 
\|\nabla^n P\|_{L^{p}(B(b))}
&\leq C_{(a,b,n,p)}\Big(
\|  v_2 \|_{W^{n-1,p_2}(B(a))}\cdot
\|  v_1 \|_{W^{n-1,p_1}(B(a))}\\
&\quad\quad\quad\quad\quad\quad+\| P \|_{L^{1}(B(a))}\Big)
\end{split}\end{equation}  where
 $\frac{1}{p}=\frac{1}{p_1}+\frac{1}{p_2}$, and
\begin{equation}\begin{split}\label{ineq_parabolic_pressure2} 
\|\nabla^n P\|_{L^{\infty}(B(b))}
&\leq C_{(a,b,n)}\Big(
\|  v_2 \|_{W^{n,\infty}(B(a))}\cdot
\|  v_1 \|_{W^{n,\infty}(B(a))}\\
&\quad\quad\quad\quad\quad\quad+\| P \|_{L^{1}(B(a))}\Big)
\end{split}\end{equation} 

\noindent Note that such constants are independent of any $v_1,v_2$ and $P$. Also,
$\infty$ is allowed for $p_1$ and $p_2$. e.g. if $p_1=\infty$, then $p_2=p$.
\end{lem}
\begin{proof}
 See the appendix.
\end{proof}

The following is a local result by using a parabolic regularization. 
It will be used in the subsection \ref{new_step3}
%of the proof of the theorem \ref{local_study_thm}
 to prove
\eqref{step3_first_claim} and \eqref{step3_second_claim}.
\begin{lem}\label{lem_parabolic_v_1_v_2_pressure}
Suppose that  we have % $v_1\in 
%\Big(C^\infty((-1,0)\times B(1))\Big)^3, v_2\in 
%\Big(C^\infty((-1,0)\times B(1))\Big)^{3}$ and $P \in 
%C^\infty((-a^2,0)\times\mathbb{R}^3)$  which satisfy
smooth solution $(v_1,v_2,P)$ on $Q(1)=(-1,0)\times B(1)$ of
 \begin{equation*}\begin{split} 
%&(v_1)_t+(v_2\cdot\nabla)(v_1)+ \nabla P -\Delta v_1=0\\
&\partial_t(v_1)+\ebdiv(v_2\otimes v_1)+ \nabla P -\Delta v_1=0\\
%&(v_1)_t+\ebdiv(v_2)+ \nabla P -\Delta v_1=0\\
&\ebdiv (v_1)=0 \mbox{ and } \ebdiv (v_2)=0.
\end{split}\end{equation*}% on $(-a^2,0)\times B(a)$.\\
%\mathbb{R}^3$.\\
%as distribution.\\

Then,  for any $n\geq 1$,  $0<b<a<1$, $1< p_1<\infty$ and $1<p_2<\infty$, we have
\begin{equation}\begin{split} \label{ineq_parabolic_v_1_v_2}
&\|\nabla^n v_1 \|_{L^{p_1}(-({b})^2,0  ;L^{p_2}(B({b})))}
\leq C_{(a,b,n,p_1,p_2)}
\Big(\|v_2\otimes v_1 \|_{L^{p_1}(-a^2,0;W^{n-1,p_2}(B(a)))}\\&
\quad\quad\quad\quad\quad\quad\quad\quad+
\| v_1 \|_{L^{p_1}(-a^2,0;W^{n-1,p_2}(B(a)))}+
\| P \|_{L^{1}(-a^2,0;L^{1}(B(a)))}\Big)\\
\end{split}\end{equation} where $v_2\otimes v_1$ is the matrix whose
 $(i,j)$ component is the product of $j$-th component $v_{2,j}$ of $v_2$
 and $i$-th one $v_{1,i}$ of $v_1$ and $\Big(\ebdiv(v_2\otimes v_1)\Big)_i
=\sum_j\partial_j(v_{2,j} v_{1,i})$.\\% =\sum_jv_{2,j}\partial_j v_{1,i}$.\\ 

Note that such constants are independent of any $v_1,v_2$ and $P$.
\end{lem}
%The above lemma follows standard parabolic regularization result. We omit a proof because this is just a summary of result from
% lemma 4 , 
% lemma 5 and step 4 in the proof of proposition
%10 in \cite{vas:higher} except that 
Proof of this lemma \ref{lem_parabolic_v_1_v_2_pressure}  is omitted because it
is based on the standard parabolic regularization result (e.g. Solonnikov \cite{Solo})
and precise argument
is essentially contained in \cite{vas:higher}
 except that 
here we consider
 \begin{equation*}\begin{split} 
&(v_1)_t+\ebdiv(v_2\otimes v_1)+ \nabla P -\Delta v_1=0\\
\end{split}\end{equation*} while \cite{vas:higher} covered 
\begin{equation*}\begin{split} 
&(u)_t+\ebdiv(u\otimes u)+ \nabla P -\Delta u=0.\\
\end{split}\end{equation*}

The following lemma will be
 used in the subsection \ref{new_step3}
 %stage 2 of the proof of the proposition \ref{local_study_thm}
, especially when we prove \eqref{step3_second_claim} for large $r$.
%It increases a space integrability of $\nabla^n v_1$ by $\frac{1}{2}$ for each time.
% Proof for this lemma
%is  technical.
\begin{lem}\label{lem_a_half_upgrading_large_r} 
%Let $\{C_n\}_{n=1}^{\infty}$ be any sequence of positive real numbers.
%Then the following is true: \\
Suppose that we have % $v_1\in 
%\Big(C^\infty((-1,0)\times B(1))\Big)^3, v_2\in 
%\Big(C^\infty((-1,0)\times B(1))\Big)^{3}$ and $P \in 
%C^\infty((-a^2,0)\times\mathbb{R}^3)$  which satisfy
smooth solution $(v_1,v_2,P)$ on $Q(1)=(-1,0)\times B(1)$ of
 \begin{equation*}\begin{split} 
%&(v_1)_t+(v_2\cdot\nabla)(v_1)+ \nabla P -\Delta v_1=0\\
&\partial_t(v_1)+(v_2\cdot\nabla)(v_1)+ \nabla P -\Delta v_1=0\\
%&(v_1)_t+\ebdiv(v_2)+ \nabla P -\Delta v_1=0\\
&\ebdiv (v_1)=0 \mbox{ and } \ebdiv (v_2)=0.
\end{split}\end{equation*}

%Suppose  we have  $v_1\in 
%\Big(C^\infty((-1,0)\times B(1))\Big)^3, v_2\in 
%\Big(C^\infty((-1,0)\times B(1))\Big)^{3}$ and $P \in 
%^\infty((-1,0)\times B(1))$ for some $0<a\leq 1$ which satisfy
% \begin{equation*}\begin{split} 
%&\partial_t(v_1)+(v_2\cdot\nabla)(v_1)+ \nabla P -\Delta v_1=0\\
%%&(v_1)_t+\ebdiv(v_2\otimes v_1)+ \nabla P -\Delta v_1=0\\
%%&(v_1)_t+\ebdiv(v_2)+ \nabla P -\Delta v_1=0\\
%&\ebdiv (v_1)=0 \mbox{ and } \ebdiv (v_2)=0\\
%%on $(-1,0)\times B(1)$% for some $0<a\leq 1$
%%\mathbb{R}^3$
%%as distribution 
%and 
%\begin{equation*}\begin{split}
% \| v_2\|_{L^2(-{{a}^2},0;W^{n,\infty}(B({a})))}&
%\leq {C_n}  \quad\mbox{ for every }n\geq 0.\\
%%^{\frac{1}{4}}\\
%\end{split}\end{equation*}% where  $C_n$ are constants depending only on $n$.

Then, for any $n\geq 0$ and  $0<b<a<1$, we have
 \begin{equation*}\begin{split}  
\|\nabla^n {v_1} &\|_{L^{\infty}(-{({b})}^2,0  ;L^{1}(B{({b})}))}
\leq\\ & C_{(a,b,n)}\Big[
\Big(\| v_2\|_{L^2(-{{a}^2},0;W^{n,\infty}(B({a})))} +1\Big)
\cdot
  \| {v_1} \|_{L^{2}(-{a}^2,0;W^{n,{1}}(B{(a)}))}\\
&\quad\quad\quad\quad\quad\quad\quad\quad+ \|\nabla^{n+1}P \|_{L^{1}(-{a}^2,0;L^{1}(B{(a)}))}  \Big]
\end{split}\end{equation*}
 and, for any ${p}\geq 1$,
  \begin{equation*}\begin{split}  
&\|\nabla^{n} {v_1} \|^{p+\frac{1}{2}}_{L^{\infty}(-{({b})}^2,0  ;L^{p+\frac{1}{2}}(B{({b})}))}\leq\\
& \quad\quad C_{(a,b,n,p)}\Big[%(\frac{2}{2p+1})
\Big(\| v_2\|_{L^2(-{{a}^2},0;W^{n,\infty}(B({a})))} +1\Big)
\cdot
\|{v_1} \|_{L^{2}(-{({a})}^2,0;W^{n,2p}(B{({a})}))}\\
&\quad\quad\quad\quad\quad\quad+ \|\nabla^{n+1} P\|_{L^{1}(-{({a})}^2,0;L^{2p}(B{({a})}))}
\Big]\cdot%\\ &\quad\quad\quad\quad\cdot
\| {v_1} \|^{p-\frac{1}{2}}_{L^{\infty}(-{({a})}^2,0;W^{n,p}(B{(a)}))}.
\end{split}\end{equation*}  
% and, for any ${p}\geq1$,
%  \begin{equation*}\begin{split}  
%&\|\nabla^{n} {v_1} \|_{L^{\infty}(-{({b})}^2,0  ;L^{p+\frac{1}{2}}(B{({b})}))}\\
%&\leq C_{a,b,n,p}\Big(1+\|{v_1}\|
%_{L^{2}(-{a}^2,0;W^{n,2p}(B{(a)}))}\\
%&\quad + 
%\|{v_1}\|_{L^{\infty}(-{a}^2,0;W^{n,p}(B{(a)}))}\\
%&\quad+\|\nabla^{n+1} P\|_{L^{1}(-{a}^2,0;L^{2p}(B{(a)}))}\Big)^{2}.\\
%\end{split}\end{equation*} 

%Note that above constants depends only on their subindices.
%are independent of any $v_1,v_2$ and $P$.
Note that such constants are independent of any $v_1,v_2$ and $P$.
\end{lem}
\begin{proof}
 See the appendix.
\end{proof}

The following non-local version of Sobolev-type lemmas will be useful 
when we  handle fractional derivatives by   Maximal
functions. We will see 
in the subsection \ref{new_step4}
 that the power $({1+\frac{3}{p}})$ of $M$ 
on the right hand side of the following estimate is very important
 to obtain a required estimate \eqref{step4_third_claim}.% for fractional derivatives.
\begin{lem}\label{lem_Maximal 2.5 or 4}
 Let $M_0>0$ and $1\leq p <\infty$. Then there exist $C=C(M_0,p)$ with
 the following property: \\

 %For any $p$ with $1\leq p<\infty$, 
 For any $M\geq M_0$ and 
 for any $f\in C^1(\mathbb{R}^3)$
% for any $f\in L^1_{loc}(\mathbb{R}^3)$ with $\nabla f\in L^1_{loc}(\mathbb{R}^3)$ 
such that $ \int_{\mathbb{R}^3}\phi(x)f(x)dx=0$, we have
\begin{equation*}
 \|f\|_{L^p(B(M))}\leq CM^{1+\frac{3}{p}}\cdot\Big(
\|\mathcal{M}(|\nabla f|^p)\|^{1/p}_{L^{1}(B(1))}
+\|\nabla f\|_{L^1(B(2))}\Big).
\end{equation*} 

\end{lem}
\begin{proof}
 See the appendix.
\end{proof}

With the above lemmas, we are ready to prove the proposition \ref{local_study_thm}.
\begin{proof} [Proof of proposition \ref{local_study_thm}]
We divide this proof into three stages.\\

Stage 1 in subsection \ref{new_step1_2_together}: First, we will obtain a $L_t^{\infty}L_x^2$-local bound for $u$
by using the mean-zero property of $u$ and $w$.
% in order to use the conclusion of the first local study
%proposition \ref{partial_problem_II_r}. As a result, 
%we will have 
Then, a $L^\infty$-local bound of $u$
 follows thanks to the first local study
proposition \ref{partial_problem_II_r}.\\

%Stage 2: use proposition \ref{partial_problem_II_r} to
% get $ L^{\infty}L^{\infty}$ local bound for $u$.\\

Stage 2 in subsection \ref{new_step3}: We will get a $L^\infty$-local bound for $\nabla^d u$ for $d\geq1$ by using 
 an induction argument with
a boot-strapping.
This is not obvious especially when $r$ is large because $w$
depends a non-local part of $u$ while our knowledge about the $L^\infty$-bound of
$u$ from the stage 1 is only local.\\

Stage 3 in subsection \ref{new_step4}: We will achieve a $L^\infty$-local bound for $(-\Delta)^{\alpha/2}\nabla^d u$ for $d\geq1$ with $0<\alpha<2$ from the integral representation
of the fractional Laplacian. The non-locality of this fractional operator 
will let us to adopt more complicated conditions (see \eqref{local_study_condition3}).\\

%(*Stage 3 is used only for $0<\alpha<2$ case.)

\subsection{Stage 1: to obtain $L^{\infty}$-local bound for $u$.}\label{new_step1_2_together}

First we suppose that $u$ satisfies all conditions of the proposition \ref{local_study_thm}
without \eqref{local_study_condition3} (The condition \eqref{local_study_condition3} will be assumed only at the stage 3). Our goal is to find a sufficiently small $\bar{\eta}$ which is independent of
$r\in[0,\infty)$.\\
%Recall our $u$ depends on $0\leq r<\infty$ 
%and we are looking for 

%\subsection{Step 1: obtain $L^{\infty}L^2$ local bound for u.}\label{step1_infty}
%$\mathbf{Step 1:}$ obtain $L^{\infty}L^2$ local bound for $u$.\\

%Basically  we follow step 1 and 2 of proof of proposition 10 in \cite{vas:higher}. 
Assume $\bar{\eta}\leq 1$ and
define $\bar{r}_0=\frac{1}{4}$ for this subsection. 
From \eqref{local_study_condition1}, we get
% $\|u(t,\cdot)\|_{L^6(B(2))}\leq C\|\nabla u(t,\cdot)\|_{L^2(B(2))}$
%for $-4<t<0$.  So, 
\begin{equation*}
\|u\|_{L^2(-4,0;L^{6}(B(2)))}\leq C\|\nabla u\|_{L^2(-4,0;L^{2}(B(2)))}
\leq {C}\cdot\bar{\eta}.
\end{equation*}
%Recall \begin{equation}\begin{split}
%w&= (u*\phi_{r})-\int_{\mathbb{R}^3}\phi(y)(u*\phi_{r})(y)dy\\
%&= w^\prime\quad\quad - \quad\quad w^{\prime\prime}\quad\quad
% \quad t\in (-4,\infty ), x\in \mathbb{R}^3 \\
%\end{split}\end{equation}
From the corollary \ref{convolution_cor}, if $r\geq\bar{r}_0$, then
\begin{equation*}\begin{split}
 \| w\|_{L^2(-4,0;L^{\infty}(B(2)))}&
%\leq {C}
%\cdot(1+\frac{4}{r})^3\cdot\|\mathcal{M}(|\nabla u|)\|_{L^2{(Q(2))}}\\
%&\leq {C}
%\cdot(1+\frac{4}{\bar{r}_0})^3\cdot\|\mathcal{M}(|\nabla u|)\|_{L^2{(Q(2))}}\\
%&\leq {C}\|\mathcal{M}(|\nabla u|)\|_{L^2{(Q(2))}}
\leq {C}\cdot\bar{\eta}.
\end{split}\end{equation*}
On the other hand, if $0\leq r<\bar{r}_0$, then
\begin{equation*}\begin{split}
 \| w^{\prime}\|_{L^2(-4,0;L^{6}(B(\frac{7}{4})))}
%= \| u*\phi_{r}\|_{L^2(-4,0;L^{6}(B(\frac{7}{4})))}\\
&\leq {C}\| u\|_{L^2(-4,0;L^{6}(B(2)))}
%\leq {C}\| \nabla u\|_{L^2(-4,0;L^{2}(B(2)))}
\leq {C}\bar{\eta}
\end{split}\end{equation*} because  $\phi_r$ is supported 
in $B(r)\subset B(\bar{r})$, and $w=u*\phi_r$ (see \eqref{young}). % and  $r<\bar{r}_0=1/4$. 
For $w^{\prime\prime}$,
\begin{equation*}\begin{split}
&\| w^{\prime\prime}\|_{L^2(-4,0;L^{\infty}(B(2)))}
 %=\| w^{\prime\prime}\|_{L^2((-4,0))}\\
%=\|\int_{\mathbb{R}^3}\phi(y)(u*\phi_{r})(y)dy\|_{L^2((-4,0))}\\
\leq\|\|u*\phi_{r}\|_{L^1(B(1))}\|_{L^2((-4,0))}
\leq\|\|u\|_{L^1(B(2))}\|_{L^2((-4,0))}\\
&\leq {C}\| u\|_{L^2(-4,0;L^{6}(B(2)))}
%&\leq {C}\| \nabla u\|_{L^2(-4,0;L^{2}(B(2)))}
\leq {C}\bar{\eta}.
\end{split}\end{equation*}% because $ w^{\prime\prime}$ is a constant in $x$.\\
Thus  $\| w\|_{L^2(-4,0;L^{6}(B(\frac{7}{4})))}
\leq {C}\bar{\eta}$ if $r<\bar{r}_0$ from $w= w^{\prime} +w^{\prime\prime}$.\\

In sum, for any $0\leq r<\infty$, 
\begin{equation}\label{step1_w}
 \| w\|_{L^2(-4,0;L^{6}(B(\frac{7}{4})))}\leq {C}\bar{\eta}.
\end{equation}

Since the equation \eqref{navier_Problem II-r} depends only on $\nabla P$,
without loss of generality, we may assume 
 $\int_{\mathbb{R}^3}\phi(x)P(t,x)=0$ for $t\in(-4,0)$. Then with the mean zero property \eqref{local_study_condition1} of $u$, we have
\begin{equation*}\label{step1_nabla_p}
 \|\int_{\mathbb{R}^3}\phi(x)\nabla P(\cdot,x) dx\|_{L^1(-4,0)}\leq C \bar{\eta}^{\frac{1}{2}}
\end{equation*} after integration in $x$.\\

From Sobolev,
\begin{equation*}\begin{split}           
             \|\nabla P\|
_{L^1(-4,0;L^{\frac{3}{2}}(B(\frac{7}{4}))}
%&\leq C\Big(\|\nabla^2 P\|_{L^1(Q(2))} + \|\nabla P \|
%_{L^1(-4,0;L^{1}(B(\frac{7}{4}))}\Big)\\
%&\leq C\Big(\|\nabla^2 P\|_{L^1(Q(2))}+  \|\int_{\mathbb{R}^3}\phi\nabla P dx\|%_{L^1(-4,0)}\Big)\\
&\leq C \bar{\eta}^{\frac{1}{2}}\\
\mbox{ and } \quad  \| P\|_{L^1(-4,0;L^{3}(B(\frac{7}{4}))}
%&\leq C\Big(\|\nabla P\|_{L^1(-4,0;L^{\frac{3}{2}}(B(\frac{7}{4}))} + \| P \|
%_{L^1(-4,0;L^{\frac{3}{2}}(B(\frac{7}{4}))}\Big)\\
&%\leq C\|\nabla P\|
%_{L^1(-4,0;L^{\frac{3}{2}}(B(\frac{7}{4}))}
\leq C \bar{\eta}^{\frac{1}{2}}\\
   \end{split}
\end{equation*}

Then we follows step 1 and step 2 of the proof of the proposition 10 in \cite{vas:higher},
we can obtain
\begin{equation*}
 \| u\|_{L^{\infty}(-3,0;L^{\frac{3}{2}}(B(\frac{6}{4})))}
\leq {C}\bar{\eta}^{\frac{1}{3}}.\\ 
\end{equation*} and then
\begin{equation*}
 \| u\|_{L^{\infty}(-2,0;L^{2}(B(\frac{5}{4})))}\leq {C}\bar{\eta}^{\frac{1}{4}} 
\end{equation*} for $0\leq r<\infty$. Details are omitted.\\

%\subsection{Step 2:use partial thm to get $ L^{\infty}L^{\infty}$ local bound for $u $}\label{step2_partial}
%$\mathbf{Step 2:}$ use proposition \ref{partial_problem_II_r} to get $ %L^{\infty}L^{\infty}$ local bound for $u $.\\

Finally, by taking  $0<\bar\eta<1$ such that $C\bar\eta^\frac{1}{4}\leq\bar\delta$,
we have all assumptions of the proposition \ref{partial_problem_II_r}. As a result,
we have 
 $|u(t,x)|\leq 1 \mbox{ on } [-\frac{3}{2},0]\times B(\frac{1}{2})$.\\

%\subsection{Step 3:use induction to get estimate higher(integral) derivative}\label{step3_integral}

\subsection{Stage 2: to obtain $L^\infty$ local bound for $\nabla^d u$.}\label{new_step3}

Here we cover only classical derivatives, i.e. $\alpha=0$. 
For any integer $d\geq1$, our goal is to find $C_{d,0}$ such that
$|((-\Delta)^{\frac{0}{2}}\nabla^d) u(t,x)|=|\nabla^d u(t,x)|\leq C_{d,0}$ on
$(-(\frac{1}{3})^2,0)\times(B(\frac{1}{3}))$. \\
  
We define a strictly decreasing sequence of balls 
and parabolic cylinders from 
$(-(\frac{1}{2})^2,0)\times B(\frac{1}{2})$ 
to $(-(\frac{1}{3})^2,0)\times(B(\frac{1}{3}))$ by
\begin{equation*}\begin{split}
&\bar{B}_n =B(\frac{1}{3}+\frac{1}{6}\cdot 2^{-n})=B(l_n) \\
&\bar{Q}_n=(-(\frac{1}{3}+\frac{1}{6}\cdot 2^{-n})^2,0)\times \bar{B}_n
=(-(l_n)^2,0)\times \bar{B}_n
\end{split}\end{equation*} where $l_n=\frac{1}{3}+\frac{1}{6}\cdot 2^{-n}$.\\

%They are strictly decreasing to the origin (0,0)
First we claim in order to cover the small $r$ case:\\

  There exist two positive sequences
$\{\bar{r}_n\}_{n=0}^{\infty}$
and $\{C_{n,small}\}_{n=0}^{\infty}$ such that for any integer $n\geq0$
and for any $r\in [0,\bar{r}_n)$,
\begin{equation}\begin{split}\label{step3_first_claim}
           &\| \nabla^n u\|_{L^{\infty}(\bar{Q}_{11n})}\leq  C_{n,small}.   
\end{split}\end{equation}

 %and $\bar{r}_n$ is decreasing.\\
Indeed, from the previous subsection \ref{new_step1_2_together} (the stage 1), \eqref{step3_first_claim} holds for $n=0$ 
by taking $\bar{r}_0=1$ and $C_{0,small}=1$. 
We define $\bar{r}_n= $ distance between
$B_{11n}$ and $(B_{11n-1})^c$ for $n \geq 1$. 
%Then we can copy every step by inserting
%intermidiate balls every time because 
Then  $\{\bar{r}_n\}_{n=0}^{\infty}$ is decreasing 
to zero as $n$ goes to $\infty$. Moreover,
%because we assume that $r$ is assumed bounded above by $\bar{r}_n$ , 
we
can control $w$  by $u$  %from the definition of $w$,
%on the next bigger ball 
as long as $0\leq r<\bar{r}_n$: for any $n\geq 1$,
% such that  $r<\bar{r}_n$,
 %$w$ is controlled by $u$ on the next bigger ball:
\begin{equation}\begin{split}\label{intermidiate balls}
&\|w\|_{L^{p_1}(-(l_{m})^2,0;L^{p_2}({\bar{B}_{m}}))}\leq 
\Big(\|u\|_{L^{p_1}(-(l_{m-1})^2,0;L^{p_2}({\bar{B}_{m-1}}))} + C\Big) \quad\mbox{ and}\\
&\|\nabla^{k}w\|_{L^{p_1}(-(l_{m})^2,0;L^{p_2}({\bar{B}_{m}}))}\leq 
\|\nabla^{k}u\|_{L^{p_1}(-(l_{m-1})^2,0;L^{p_2}({\bar{B}_{m-1}}))}
\end{split}\end{equation} for any integer $m$ such that
 $m\leq 11\cdot n$, for any $ k\geq1$ 
and for any $p_1\in [1,\infty]$ and $p_2\in [1,\infty]$ (see \eqref{young}).\\% $1\leq p_1,p_2\leq\infty$.\\

We will use an induction with a boot-strapping.
First we fix $d\geq 1$ and suppose that  \eqref{step3_first_claim} is true up to $n=(d-1)$. It implies
for any $r\in [0,\bar{r}_{d-1})$
\begin{equation*}\begin{split}
\|u\|_{L^{\infty}(-l_{s}^2,0;W^{d-1,{\infty}} (\bar{B}_{s}))}
\leq C
\end{split}\end{equation*} where $s=11(d-1)$.
We want to
show that \eqref{step3_first_claim} is also true  for the case $n=d$.\\

\noindent From \eqref{intermidiate balls},
$\|w\|_{L^{\infty }(-l_{s+{1 }}^2,0;W^{{ d-1 } ,{ \infty}} (\bar{B}_{s+{ 1 }}))}\leq C$ and,
From the lemma \ref{lem_parabolic_v_1_v_2_pressure} with
$v_2=w$ and $v_1=u$,\quad
$\|u\|_{L^{16}(-l_{s+{ 2}}^2,0;W^{{ d } ,{32 }} (\bar{B}_{s+{ 2 }}))}\leq C$.
Then, we use \eqref{intermidiate balls} 
and the lemma \ref{lem_parabolic_v_1_v_2_pressure} in turn:
\begin{equation*}\begin{split}
&\rightarrow w\in{L^{16}(-l_{s+{ 3}}^2,0;W^{{ d } ,{32 }} (\bar{B}_{s+{3 }}))}
\rightarrow
 u\in{L^{8}(-l_{s+{ 4}}^2,0;W^{{ d+1 } ,{16 }} (\bar{B}_{s+{ 4 }}))}\\
&\rightarrow w\in{L^{8}W^{{ d+1 } ,{16 }} }
%\rightarrow u\in{L^{4}W^{{ d+2 } ,{8 }} }
\rightarrow ... \rightarrow
%&\|w\|_{L^{4}(-l_{s+{ 7}}^2,0;W^{{ d+2 } ,{8 }} (\bar{B}_{s+{ 7 }}))}\leq C\\
u\in{L^{2}W^{{ d+3 } ,{4 }} }
\end{split}\end{equation*} Then, from Sobolev,
\begin{equation*}\begin{split}
&\rightarrow u\in{L^{2}W^{{ d+2 } ,{\infty}} }\rightarrow
w\in{L^{2}(-l_{s+{ 9}}^2,0;W^{{ d+2 } ,{\infty}} (\bar{B}_{s+{ 9 }}))}.
\end{split}\end{equation*}
This estimate gives us 
\begin{equation*}\begin{split}
&\Delta(\nabla^d u) ,
\ebdiv(\nabla^d(w\otimes u))\mbox{ and }
\nabla(\nabla^dP) \in{L^{1}(-l_{s+{ 10}}^2,0;L^{\infty}(\bar{B}_{s+{10 }}))}
\end{split}\end{equation*}
 where 
we used 
%second conclusion of lemma \ref{higher_pressure} 
\eqref{ineq_parabolic_pressure2}
for the pressure term. Thus
\begin{equation*}\begin{split}
&\partial_t(\nabla^d u)\in{L^{1}(-l_{s+{ 10}}^2,0;L^{\infty}(\bar{B}_{s+{10 }}))}.
\end{split}\end{equation*}
Finally, we obtain that for any $r\in [0,\bar{r}_{d})$
\begin{equation*}\begin{split}
&\|\nabla^d u\|_{L^{\infty}(-l_{s+{ 11}}^2,0;L^{\infty}(\bar{B}_{s+{11 }}))}\leq C.
\end{split}\end{equation*} where $C$ depends only on $d$. 
By the induction argument, we showed the above claim \eqref{step3_first_claim}. \\

Now we introduce the second claim:\\

 There exist a sequences
 $\{C_{n,large}\}_{n=0}^{\infty}$ such that for any integer $n\geq0$
and for any $r\geq\bar{r}_n$,
\begin{equation}\begin{split}\label{step3_second_claim}
           &\| \nabla^n u\|_{L^{\infty}(\bar{Q}_{21\cdot n})}\leq  C_{n,large} \\   
\end{split}\end{equation} where $\bar{r}_n$ comes from previous claim
\eqref{step3_first_claim}.\\

Before proving the above second claim \eqref{step3_second_claim}, 
we need the following two observations \textbf{(I),(II)}
from the lemmas  \ref{lem_parabolic_v_1_v_2_pressure} and 
\ref{higher_pressure}:\\

\textbf{(I).} From the corollary \ref{convolution_cor} for any $n\geq 0$, 
if $r\geq\bar{r}_n$, then
\begin{equation*}\begin{split}\label{nabla_n_w_large_r}
 \| w\|_{L^2(-4,0;W^{n,\infty}(B(2)))}&
\leq {C_n}.
%^{\frac{1}{4}}\\
\end{split}\end{equation*}
We use \eqref{ineq_parabolic_v_1_v_2} in the lemma \ref{lem_parabolic_v_1_v_2_pressure} 
with $v_1=u$ and $v_2=w$. 
Then it
becomes 
\begin{equation}\begin{split} \label{ineq_nabla_n_u_large_r}
\|\nabla^n u \|_{L^{p_1}(-(l_m)^2,0  ;L^{p_2}(\bar{B}_m))}
%\leq C_{(m,n,p_2)}
%\Big(\|w\otimes u \|_{L^{p_1}(-(l_{m-1})^2,0;W^{n-1,p_2}({\bar{B}_{m-1}}))}\\&
%\quad\quad\quad\quad+
%\| u \|_{L^{p_1}(-(l_{m-1})^2,0;W^{n-1,p_2}({\bar{B}_{m-1}}))}+
%\| P \|_{L^{1}(-(l_{m-1})^2,0;L^{1}({\bar{B}_{m-1}}))}\Big)\\
%&\leq C_{(m,n,p_2)}
%\Big(\|w \|_{L^{2}(-(l_{m-1})^2,0;W^{n-1,\infty}({\bar{B}_{m-1}}))}\cdot
%\| u \|_{L^{1/(\frac{1}{p_1}-\frac{1}{2})}
%(-(l_{m-1})^2,0;W^{n-1,p_2}({\bar{B}_{m-1}}))}\\&
%\quad\quad\quad\quad\quad+
%\| u \|_{L^{p_1}(-(l_{m-1})^2,0;W^{n-1,p_2}({\bar{B}_{m-1}}))}+
%\| P \|_{L^{1}(-(l_{m-1})^2,0;L^{1}({\bar{B}_{m-1}}))}\Big)\\
%&\leq C_{(m,n,p_2)}
%\Big(\| u \|_{L^{1/(\frac{1}{p_1}-\frac{1}{2})}
%-(l_{m-1})^2,0;W^{n-1,p_2}({\bar{B}_{m-1}}))}\\&
%\quad\quad\quad\quad\quad+
%\| u \|_{L^{p_1}(-(l_{m-1})^2,0;W^{n-1,p_2}({\bar{B}_{m-1}}))}+
%1\Big)\\
&\leq C_{(m,n,p_2)}
\Big(\| u \|_{L^{\frac{2p_1}{2-p_1}}
(-(l_{m-1})^2,0;W^{n-1,p_2}({\bar{B}_{m-1}}))}+
1\Big)\\
\end{split}\end{equation}
 for  $n\geq 1$, $m\geq 1$, $1< p_1\leq 2$ and $1<p_2<\infty$. (For
the case $p_1=2$, we may interpret $\frac{2p_1}{2-p_1}=\infty$.)\\
%by using \eqref{nabla_n_w_large_r}. \\

\textbf{(II).} Moreover, \eqref{ineq_parabolic_pressure} in the lemma
\ref{higher_pressure}  becomes
\begin{equation}\begin{split} \label{ineq_nabla_n_pressure_large_r}
\|\nabla^n P\|_{L^{1}(-(l_m)^2,0;L^{p}(\bar{B}_m))}
%\leq C_{(m,n,p)}\Big(
%\| \nabla w\otimes \nabla u \|_{L^{1}(-(l_{m-1})^2,0;W^{n-2,p}({\bar{B}_{m-1}}))}\\
%&\quad\quad\quad\quad+\| P \|_{L^{1}(-(l_{m-1})^2,0;L^{1}({\bar{B}_{m-1}}))}\Big)\\
&\leq C_{(m,n,p)}\Big(
\|  u \|_{L^{2}(-(l_{m-1})^2,0;W^{n-1,p}({\bar{B}_{m-1}}))}+1\Big)
\end{split}\end{equation} for  $n\geq 2$ and $1<p<\infty$.\\

Now we are ready to prove the second claim \eqref{step3_second_claim}
 by an induction with a boot-strapping. From the previous subsection
 \ref{new_step1_2_together} (the stage 1), \eqref{step3_second_claim} holds for $n=0$ 
with  $C_{0,large}=1$. Fix $d\geq 1$ and suppose that we
 have \eqref{step3_second_claim}  up to $n=(d-1)$. It implies
for any $r\geq\bar{r}_{d-1}$
\begin{equation*}\begin{split}
\|u\|_{L^{\infty}(-l_s^2,0;W^{d-1,{\infty}} (\bar{B}_{s}))}
\leq C_{d-1,large}
\end{split}\end{equation*} where $s=21(d-1)$.
We want to
show \eqref{step3_second_claim}  for $n=d$.\\

\noindent By using \eqref{ineq_nabla_n_u_large_r} with $n=d, p_1=2$ and $p_2=11$,
\begin{equation*}\begin{split}
\|u\|_{L^{ 2}(-l_{s+{1 }}^2,0;W^{{ d } ,{11 }} (\bar{B}_{s+{ 1 }}))}\leq C
\end{split}\end{equation*}
and,
from \eqref{ineq_nabla_n_pressure_large_r} with $n=d+1,m=0$ and $p=11$,
\begin{equation*}\begin{split}
\|\nabla^{d+{1  }}P\|_{L^{1 }(-l_{s+{ 2}}^2,0;L^{ 11 } (\bar{B}_{s+{  2}}))}\leq C.
\end{split}\end{equation*}
 Combining the above two results with the 
lemma \ref{lem_a_half_upgrading_large_r} for $v_1=u$ and $v_2=w$
, we can have increased integrability in space by $0.5$ up to $6$:
\begin{equation*}\begin{split}
&\|u\|_{L^{\infty }(-l_{s+{3 }}^2,0;W^{{ d } ,{ 1}} (\bar{B}_{s+{ 3 }}))}\leq C,\\
&\|u\|_{L^{\infty }(-l_{s+{ 4}}^2,0;W^{{ d } ,{1.5 }} (\bar{B}_{s+{ 4 }}))}\leq C,\\
%&\|u\|_{L^{\infty }(-l_{s+{ 5}}^2,0;W^{{ d } ,{2 }} (\bar{B}_{s+{ 5 }}))}\leq C\\
&\quad\quad\quad\quad\quad\quad\cdots, \quad\mbox{ and} \\
&\|u\|_{L^{\infty }(-l_{s+{ 13}}^2,0;W^{{ d } ,{6 }} (\bar{B}_{s+{ 13 }}))}\leq C.\\
\end{split}\end{equation*}
By using \eqref{ineq_nabla_n_u_large_r} and \eqref{ineq_nabla_n_pressure_large_r}
again, we have 
\begin{equation*}\begin{split}
&\|u\|_{L^{ 2}(-l_{s+{14 }}^2,0;W^{{d+1  } ,{ 6}} (\bar{B}_{s+{ 14 }}))}\leq C 
\quad\mbox{and}\\
&\|\nabla^{d+{ 2 }}P\|_{L^{1 }(-l_{s+{ 15}}^2,0;L^{6  } (\bar{B}_{s+{ 15 }}))}\leq C.
\end{split}\end{equation*}
Combining the above two results with  the lemma \ref{lem_a_half_upgrading_large_r} again
, we have
\begin{equation*}\begin{split}
&\|u\|_{L^{\infty }(-l_{s+{16 }}^2,0;W^{{ d+1 } ,{ 1}} (\bar{B}_{s+{ 16 }}))}\leq C,\\
%&\|u\|_{L^{\infty }(-l_{s+{17 }}^2,0;W^{{ d+1 } ,{ 1.5}} (\bar{B}_{s+{ 17 }}))}\leq C\\
%& u\in L^{\infty}_tW^{d+1,2}_x \quad\mbox{ for } \bar{Q}(l_{s+18}),\\
&\quad\quad\quad\quad\quad\quad\cdots,\quad\mbox{ and} \\
&\|u\|_{L^{\infty }(-l_{s+{21 }}^2,0;W^{{ d+1} ,{ 3.5}} (\bar{B}_{s+{21 }}))}\leq C.\\
\end{split}\end{equation*}
Finally, from Sobolev's inequality, 
\begin{equation*}\begin{split}
&\|\nabla^d u\|_{L^{\infty }(-l_{s+{21 }}^2,0;L^{\infty} (\bar{B}_{s+{ 21 }}))}
\leq C
%&\nabla^{d} u\in L^{\infty}_tL^{\infty}_x \quad\mbox{ for } \bar{Q}(l_{s+21})
%= \bar{Q}(l_{21d}).\\
\end{split}\end{equation*} where $C$ depends
 only $d$ not $u$ nor $r$ as long as $r\geq \bar{r}_d$.
From induction, we proved second claim \eqref{step3_second_claim}.\\

Define for any $n\geq0$, $C_{n,0} = \max({C_{n,small},C_{n,large}})$ where
$C_{n,small}$ and $C_{n,large}$ come from \eqref{step3_first_claim} 
and \eqref{step3_second_claim} respectively. 
Then we have:
\begin{equation}\begin{split}\label{step3_conclusion}
           &\| \nabla^n u\|_{L^{\infty}(Q(\frac{1}{3}))}\leq  C_{n,0} \\   
\end{split}\end{equation} for any $n\geq0$ and for any $0\leq r<\infty$
because $Q(\frac{1}{3})\subset \bar{Q}_{n}$. 
It ends this stage 2.\\

%\subsection{*Step 4: get estimate higher fractional derivative $(0<\alpha<2)$}\label{step4_fractional}
%$\mathbf{*Step 4:}$ get estimate higher fractional derivative 
%$(0<\alpha<2)$.\\

\subsection{Stage 3: to obtain $L^\infty$ local bound for $(-\Delta)^{\alpha/2}\nabla^d u$.}\label{new_step4}

From now on, we 
assume further that $(u,P)$ satisfies \eqref{local_study_condition3} as well as 
all the other conditions of the proposition \ref{local_study_thm}.
%consider  $0<\alpha<2$ case (fractional derivatives).
%Suppose that $(u,P)$ satisfies all conditions of the proposition \ref{local_study_thm}.
%Then step 1,2 and 3 fo
%We will combine the result \eqref{step3_conclusion} of the previous step 3 with 
%\eqref{local_study_condition2} and \eqref{local_study_condition3}.
In the following proof, we will not
divide the proof into a small $r$ part and a large $r$ part.\\%  as you will see. 

Fix an integer $d\geq 1$ and  
a real $\alpha$ with $0<\alpha<2$.
i.e. any constant which will appear may depend $d$ and $\alpha$. 
But they will be  independent of any $ r\in [0,\infty)$ and any solution $(u,P)$.\\

% The rest of this subsection will focus on a local estimate.

First, we claim:\\

 %For any $d\geq 1$ and $0<\alpha<2$, w
 There exists a constant $C=C({d,\alpha})$
such that
\begin{equation}\begin{split}\label{step4_first_claim}
|(-\Delta)^{\frac{\alpha}{2}}\nabla^d u(t,x)|
\leq C({d,\alpha})+\Big|\int_{|y|\geq {(1/6)}}
\frac{\nabla^{d} u(t,x-y)}{|y|^{3+\alpha}}dy\Big|
\end{split}\end{equation} for $ |x|\leq (1/6)$ and 
for  $ -(1/3)^2\leq t\leq 0$.\\% for $-{(a_{n+1})}^2\leq t\leq 0$.\\

%proof of the above claim \eqref{step4_first_claim}: \\

\noindent To prove \eqref{step4_first_claim}, 
we first recall the Taylor expansion of any $C^2$ function $f$ at $x$:
$f(y)-f(x)=(\nabla f)(x)\cdot(y-x)+R(x,y)$, and we have 
 %with the 
 an error estimate $|R|\leq C|x-y|^2\cdot\|\nabla^2 f\|_{L^\infty(B(x;|x-y|))}$.
Note that if we integrate the first order term $(\nabla f)(x)\cdot(y-x)$
in $y$ on   any sphere with the center $x$, we have zero by symmetry.
 As a result, if we
take any $x$ and $t$ for $ |x|\leq (1/6)$ and 
for  $ -(1/3)^2\leq t\leq 0$ respectively, then we have 
\begin{equation*}\begin{split}
|(-\Delta)^{\frac{\alpha}{2}}\nabla^d u(t,x)|&
=\Big|P.V.\int_{\mathbb{R}^3}\frac{\nabla^d u(t,x)-\nabla^d u(t,y)}{|x-y|^{3+\alpha}}dy\Big|\\
%&\leq\Big|\int_{|x-y|< (1/6)}\frac{\nabla^d u(t,x)-\nabla^d u(t,y)}{|x-y|%^{3+\alpha}}dy\Big|\\&\quad
%+\Big|\int_{|x-y|\geq (1/6)}\frac{\nabla^d u(t,x)-\nabla^d u(t,y)}{|x-y|%^{3+\alpha}}dy\Big|\\
&\leq \sup_{z\in B((1/3))}(|\nabla^{d+2}u(t,z)|)\cdot\int_{|x-y|< (1/6)}\frac{1}{|x-y|^{3+\alpha-2}}dy\\
&\quad+ \sup_{z\in B((1/3))}(|\nabla^d u(t,z)|)\cdot
\int_{|x-y|\geq (1/6)}\frac{1}{|x-y|^{3+\alpha}}dy\\
&\quad+\Big|\int_{|x-y|\geq (1/6)}\frac{\nabla^d u(t,y)}{|x-y|^{3+\alpha}}dy\Big|\\
%&\leq C\cdot\int_{|y|<(1/6)}\frac{1}{|y|^{2+\alpha}}dy
%+C\cdot\int_{|y|\geq (1/6)}\frac{1}{|y|^{3+\alpha}}dy+
%\Big|\int_{|y|\geq {(1/6)}}\frac{\nabla^d u(t,x-y)}{|y|^{3+\alpha}}dy\Big|\\
%&\leq C\cdot\int_{0}^{(1/6)}s^{-\alpha}ds
%+C\cdot\int_{(1/6)}^\infty s^{-1-\alpha}ds+\Big|\int_{|z|\geq (1/6)}\frac{\nabla^n u(t,x-z)}{|z|^{3+\alpha}}dz\Big|\\
%&\leq C_{2,0}\cdot({a_{n+1}})^{1-\alpha}\cdot\frac{1}{1-\alpha}
%+C_{1,0}\cdot({a_{n+1}})^{-\alpha}\cdot\frac{1}{\alpha} +\Big|\int_{|y|\geq a_{2}/2}\frac{\nabla u(t,x-y)}{|y|^{3+\alpha}}dy\Big|\\
&\leq C({d,\alpha}) +\Big|\int_{|y|\geq (1/6)}
\frac{\nabla^d u(t,x-y)}{|y|^{3+\alpha}}dy\Big|
\end{split}\end{equation*} 
where we used 
the result \eqref{step3_conclusion} of the previous 
subsection \ref{new_step3} (the stage 2)
 together with the Taylor expansion of
 $\nabla^d u(t,\cdot)$ at $x$
 %(e.g. see around \eqref{fractional laplacian for W_2_infty})
in order to reduce singularity by $2$ at the origin $x=y$. We
 proved the first claim \eqref{step4_first_claim}.\\

Second, we claim:\\

There exists $C=C({d,\alpha})$ such that
\begin{equation}\begin{split}\label{step4_second_claim}
\Big|\int_{|y|\geq {{(1/6)}}}\frac{\nabla^{d} u(t,x-y)}
{|y|^{3+\alpha}}dy\Big|
\leq C({d,\alpha})+\sum_{j=k}^{\infty}(\frac{1}{2^{\alpha}})^j\cdot
|({(h^{\alpha})_{2^j}}*\nabla^d u)(t,x)|\\
\end{split}\end{equation} 
 for $ |x|\leq (1/6)$ and 
for  $ -(1/3)^2\leq t\leq 0$ where
$k$ is the integer such that
$2^k\leq (1/6)< 2^{k+1}$.(i.e. from now on, we fix $ k=-3$).  Recall that $ h^\alpha$ is defined around \eqref{property_h}.\\%(k may be negative)\\

\noindent To prove the above second claim \eqref{step4_second_claim}: (Recall \eqref{property_zeta}
and \eqref{property_h})
\begin{equation*}\begin{split}
&\Big|\int_{|y|\geq {{(1/6)}}}\frac{\nabla^{d} u(t,x-y)}
{|y|^{3+\alpha}}dy\Big|=
\Big|\int_{|y|\geq {{(1/6)}}}\sum_{j=k}^{\infty}\zeta
(\frac{y}{2^j})\frac{\nabla^{d} u(t,x-y)}
{|y|^{3+\alpha}}dy\Big|\\
&=\Big|\int_{|y|\geq {{(1/6)}}}\sum_{j=k}^{\infty}
\frac{1}{(2^j)^{\alpha }} \cdot {(h^{\alpha})_{2^j}}(y)\nabla^{d} u(t,x-y)dy\Big|\\
\end{split}\end{equation*}
\begin{equation*}\begin{split}
%&\leq\sum_{j=k}^{\infty}\frac{1}{(2^j)^{\alpha }} 
%\cdot\Big|\int_{|y|\geq {{(1/6)}}} {(h^{\alpha})_{2^j}}(y)\nabla^{d} u(t,x-y)dy\Big|\\
&\leq\sum_{j=k}^{k+1}\frac{1}{(2^j)^{\alpha }} 
\cdot\Big|\int_{|y|\geq {{(1/6)}}} {(h^{\alpha})_{2^j}}(y)\nabla^{d} u(t,x-y)dy\Big|\\
&\quad+\sum_{j=k+2}^{\infty}\frac{1}{(2^j)^{\alpha }}
 \cdot\Big|\int_{|y|\geq {{(1/6)}}} {(h^{\alpha})_{2^j}}(y)\nabla^{d} u(t,x-y)dy\Big|\\
&=(I)+(II).
\end{split}\end{equation*}
\noindent For $(I)$,
\begin{equation*}\begin{split}
(I)&\leq\sum_{j=k}^{k+1}\frac{1}{(2^j)^{\alpha }} 
\cdot\Big(\Big|\int_{\mathbb{R}^3} {(h^{\alpha})_{2^j}}(y)\nabla^{d} u(t,x-y)dy\Big|\\
&\quad\quad\quad\quad\quad\quad\quad\quad + 
\int_{|y|\leq {{(1/6)}}} |{(h^{\alpha})_{2^j}}(y)|\cdot|\nabla^{d} u(t,x-y)|dy\Big)\\
&\leq\sum_{j=k}^{k+1}\frac{1}{(2^j)^{\alpha }}
 \Big(| ({(h^{\alpha})_{2^j}}*\nabla^{d} u)(t,x)|
%\\&\quad\quad\quad\quad\quad\quad\quad\quad
 + C\cdot \sup_{z\in B(1/3)}|\nabla^{d} u(t,z)|\Big)\\
&=  \sum_{j=k}^{k+1}(\frac{1}{2^{\alpha}})^j\cdot| ({(h^{\alpha})_{2^j}}
*\nabla^{d} u)(t,x)|+C({d,\alpha}).
\end{split}\end{equation*}
\noindent For $(II)$, by using $supp( h^{\alpha}_{2^j})
\subset (B(2^{j-1}))^C\subset (B(1/6))^C$
for any $j\geq k+2$,
\begin{equation*}\begin{split}
(II)&=
\sum_{j=k+2}^{\infty}\frac{1}{(2^j)^{\alpha }} 
\cdot\Big|\int_{\mathbb{R}^3} {(h^{\alpha})_{2^j}}(y)\nabla^{d} u(t,x-y)dy\Big|\\
&= \sum_{j=k+2}^{\infty}(\frac{1}{2^{\alpha}})^j
\cdot| ({(h^{\alpha})_{2^j}}*\nabla^{d} u)(t,x)|.
\end{split}\end{equation*} We showed the second claim
\eqref{step4_second_claim}.\\

Third, we claim:\\%that %for any small $M_0>0$, 

There exists $C=C({d,\alpha})$ such that
\begin{equation}\begin{split}\label{step4_third_claim}
\|{(h^{\alpha})_M}*\nabla^d u\|_{L^{\infty}( -(1/6)^2,0;L^1(B(1/6)))} 
\leq C({d,\alpha})\cdot M^{1-d}\\
\end{split}\end{equation} for any $M\geq 2^k$. (Recall
$k= -3$.)\\% is the integer such that
%$2^k\leq \frac{a_{n+1}}{2}< 2^{k+1}$.(k may be negative)\\.\\

\noindent To prove the above third claim \eqref{step4_third_claim}, 
take a convolution first with
 $ \nabla^d[{(h^{\alpha})_M}]$ into the equation 
\eqref{navier_Problem II-r}.
% $u_t +(w\cdot\nabla)u +\nabla P -\Delta u=0$.\\
Then we have
\begin{equation*}\begin{split}
(\nabla^d [{(h^{{\alpha}})_M}]*u)_t &+(\nabla^d [{(h^{{\alpha}})_M}]*\Big((w\cdot\nabla)u\Big))\\&
 +(\nabla^d [{(h^{{\alpha}})_M}]*\nabla P) -(\nabla^d [{(h^{{\alpha}})_M}]*\Delta u)=0
\end{split}\end{equation*} so that
\begin{equation*}\begin{split}
( \nabla^{d-1}[{(h^{{\alpha}})_M}]&*\nabla u)_t +(\nabla^d {[(h^{{\alpha}})_M}]*\Big((w\cdot\nabla)u\Big))\\& +(\nabla^{d-1}[{(h^{{\alpha}})_M}]*\nabla^2 P) -\Delta( 
\nabla^{d-1}
[{(h^{{\alpha}})_M}]* \nabla u)=0.
\end{split}\end{equation*}

\noindent Define a cut-off $\Phi(t,x)$ by
\begin{equation*}\begin{split} 
&0\leq\Phi(x)\leq 1 \quad , \quad
supp(\Phi)\subset (-4,0)\times B({2})\\
&\Phi(t,x) = 1 \mbox{ for } (t,x)\in  (-(1/6)^2,0)\times B({(1/6)}).
  \end{split}
\end{equation*}

\noindent Multiply $\Phi(t,x)\frac{(\nabla^{d-1}[{(h^{{\alpha}})_M}]*\nabla u)(t,x)}{|(\nabla^{d-1}[{(h^{{\alpha}})_M}]*\nabla u) (t,x)|}$, then
integrate in $x$:%\in\mathbb{R}^3$
\begin{equation*}\begin{split} 
&\frac{d}{dt}\int_{\mathbb{R}^3}\Phi(t,x)|(\nabla^{d-1}[{(h^{{\alpha}})_M}]*
\nabla u) (t,x)|dx\\
&\leq\int_{\mathbb{R}^3}(|\partial_t\Phi(t,x)|+|\Delta\Phi(t,x)|)
|(\nabla^{d-1}[{(h^{{\alpha}})_M}]*\nabla u) (t,x)|dx \\
&\quad\quad+\int_{\mathbb{R}^3}|\Phi(t,x)||(\nabla^{d-1}[{(h^{{\alpha}})_M}]*\nabla^2 P)|dx \\
&\quad\quad+\int_{\mathbb{R}^3}|\Phi(t,x)||\nabla^d [{(h^{{\alpha}})_M}]*
\Big((w\cdot\nabla)u\Big)|dx.
\end{split}\end{equation*} 

\noindent Then integrating on $[-4,t]$ for  any $t\in[-(1/6),0]$ gives 
\begin{equation*}\begin{split} 
&\|{(h^{{\alpha}})_M}*\nabla^d u\|_{L^{\infty}( -(1/6)^2,0;L^1(B(1/6)))}\\ 
&=\|\nabla^{d-1}[{(h^{{\alpha}})_M}]*\nabla u\|_{L^{\infty}( -(1/6)^2,0;L^1(B(1/6)))}\\ 
&\leq C\Big(\|\nabla^{d-1}[{(h^{{\alpha}})_M}]*\nabla u\|_{L^{1}( -4,0;L^1(B(2)))} \\
&+ \|\nabla^{d-1}[{(h^{{\alpha}})_M}]*\nabla^2 P\|_{L^{1}( -4,0;L^1(B(2)))} \\
&+ \|\nabla^d [{(h^{{\alpha}})_M}]*\Big((w\cdot\nabla)u\Big)\|
_{L^{1}( -4,0;L^1(B(2)))}\Big) \\
&=(I) + (II)+ (III).
\end{split}\end{equation*} 

For $(I)$, we use simple observations 
$\nabla^m[(f)_\delta]=\delta^{-m}\cdot(\nabla^mf)_\delta$ 
and \\$|(f)_\delta*\nabla u|(x)\leq C_f\cdot\mathcal{M}(|\nabla u|)(x) $
for any $f\in C^\infty_0(\mathbb{R}^3)$ so that
\begin{equation*}\begin{split}
|(\nabla^{d-1}[{(h^{{\alpha}})_M}]*\nabla u)(t,x)|&
=M^{-(d-1)}\cdot|((\nabla^{d-1}{h^{{\alpha}})_M}*\nabla u)(t,x)|\\
&\leq C\cdot M^{-(d-1)}\cdot\mathcal{M}(|\nabla u|)(t,x)
\end{split}\end{equation*} for any $0<M<\infty$ so that
\begin{equation*}\begin{split}
(I)&=\|(\nabla^{d-1}[{(h^{{\alpha}})_M}]*\nabla u)\|_{L^{1}( -4,0;L^1(B(2)))}\\ 
&\leq C\cdot M^{-(d-1)}\cdot\|\mathcal{M}(|\nabla u|)\|_{L^{1}( -4,0;L^1(B(2)))}\\
%&\quad\quad\quad\leq C\|\mathcal{M}(|\nabla u|)\|_{L^{1}( -a_{n+1}^2,0;L^1(B(2)))} \\
&\leq C\cdot M^{-(d-1)}\cdot\|\mathcal{M}(|\nabla u|)\|_{L^{2}( -4,0;L^2(B(2)))}
\leq C\cdot M^{1-d}
%\leq C\leq C M^{\alpha/2} \\
\end{split}\end{equation*} for any $0<M<\infty$.\\

For $(II)$, we use our global information about pressure in \eqref{local_study_condition3}
 thanks to the property of the Hardy space \eqref{hardy_property}:
\begin{equation}\begin{split}\label{pressure_hardy_used}
(II)&=\|\nabla^{d-1}[{(h^{{\alpha}})_M}]*\nabla^2 P\|_{L^{1}( -4,0;L^1(B(2)))}\\
&=M^{-(d-1)}\cdot\|(\nabla^{d-1}{h^{{\alpha}})_M}*\nabla^2 P\|_{L^{1}( -4,0;L^1(B(2)))}\\
%&\quad\quad\quad\quad\leq 
%\|(h^{({\alpha})}_{M}*\nabla^2 P)\|_{L^{1}( -4,0;L^1(B(2)))}\\
&\leq M^{-(d-1)}\cdot
\| \sup_{\delta>0}(|(\nabla^{d-1}{h^{{\alpha}})_\delta}*\nabla^2 P|)\|_{L^{1}(-4,0;L^{1}(B(2)))}\\&
\leq C \cdot M^{1-d}
%\cdot \|\nabla^2 P\|_{L^{1}(-4,0;L^{1}(B(2)))}\\
%\leq C M^{\alpha/2} \\
%&\|\nabla^{d-1}[{(h^{{\alpha}})_M}]*\nabla^2 P\|_{L^{1}( %-4,0;L^1(B(2)))}
%\\&=M^{-(n-1)}\cdot|((\nabla^{d-1}{h^{{\alpha}})_M}*\nabla u)(t,x)|\\
\end{split}\end{equation} for any $0<M<\infty$.\\

For $(III)$, we use  following useful facts \textbf{(1, }$\cdots$\textbf{, 5)}:\\

\noindent  \textbf{1.} From $supp((h^{{\alpha}})_M)\subset B(2M)$,
\begin{equation*}\begin{split}
\quad&\|\nabla^d [{(h^{{\alpha}})_M}]*\Big((w\cdot\nabla)u\Big)(t,\cdot)\|
_{L^1(B(2))} \\
%&=\int_{B(2)}\Big|
%\int_{\mathbb{R}^3}\Big((w\cdot\nabla)u\Big)(t,y) (\nabla^d [{(h^{{\alpha}})_M}])(x-y)dy   
%\Big|dx\\
&\leq\int_{B(2)}
\int_{\mathbb{R}^3}\Big|\Big((w\cdot\nabla)u\Big)(t,y)\cdot (\nabla^d [{(h^{{\alpha}})_M}])(x-y)\Big|dy   
dx\\
&\leq\int_{B(2M+2)}\Big|\Big((w\cdot\nabla)u\Big)(t,y)\Big|\cdot\Big[
\int_{{B(2)}}| (\nabla^d [{(h^{{\alpha}})_M}])(x-y)|dx  \Big]
dy\\
&\leq C\|\nabla^d [{(h^{{\alpha}})_M}]\|_{L^{\infty}(\mathbb{R}^3)}\cdot
\|\Big((w\cdot\nabla)u\Big)(t,\cdot)\|_{L^1{(B(2M+2))}}
\\
&\leq C \cdot\frac{1}{M^{3+d}}\cdot
\|\Big((w\cdot\nabla)u\Big)(t,\cdot)\|_{L^1{(B(2M+2))}}
\\
&\leq C \cdot\frac{1}{M^{3+d}}\cdot
\|w(t,\cdot)\|_{L^{q^\prime}{(B(2M+2))}}
\cdot\|\nabla u(t,\cdot)\|_{L^q{(B(2M+2))}}
\end{split}\end{equation*} where $q=12/(\alpha +6)$ and $1/q+1/q^\prime = 1$.\\
Note: Because $0<\alpha<2$, we know $12/8<q<2$.

\begin{equation*}\begin{split}
 \textbf{2.}\quad& \|w(t,\cdot)\|_{L^{q}{(B(2M+2))}}\\
&\quad\leq
 CM^{1+\frac{3}{q}}\cdot\Big(
\|\mathcal{M}(|\nabla w|^q)(t,\cdot)\|^{1/q}_{L^{1}(B(1))}
+\|\nabla w(t,\cdot)\|_{L^1(B(2))}\Big)\\
&\quad\leq
 CM^{1+\frac{3}{q}}\cdot\Big(
\|\mathcal{M}(|\mathcal{M}(|\nabla u|)|^q)(t,\cdot)\|^{1/q}_{L^{1}(B(1))}
+\|\mathcal{M}(|\nabla u|)(t,\cdot)\|_{L^1(B(2))}\Big)\\
\end{split}\end{equation*} for any $M\geq 2^k$.\\

\noindent For the first inequality, we used the lemma \ref{lem_Maximal 2.5 or 4}
 and for the second one, we used the fact
  $|\nabla w (t,x)|=|(\nabla u * \phi_r)(t,x)|\leq
   C|\mathcal{M}(|\nabla u|)(t,x)|$ where
  $C$ is independent of $0\leq r<\infty$. (For $r>0$, it follows
 definitions of the convolution and the Maximal function while
for $r=0$, it follows the Lebesgue differentiation theorem with
continuity of $\nabla u$.)
 So, for any $M\geq 2^k$, from  \eqref{local_study_condition3},
\begin{equation*}\begin{split}
&\|w\|_{L^2(-4,0;L^q(B(2M+2)))}\\
&\quad\leq
 CM^{1+\frac{3}{q}}\Big(
\|\|\mathcal{M}(|\mathcal{M}(|\nabla u|)|^q)\|^{1/q}_{L_x^{1}(B(1))}\|
_{L_t^2(-4,0)}\\
&\quad\quad\quad\quad\quad\quad+\|\|\mathcal{M}(|\nabla u|)\|_{L_x^1(B(2))}\|
_{L_t^2(-4,0)}\Big)\\
&\quad\leq
 CM^{1+\frac{3}{q}}\Big(
\|\mathcal{M}(|\mathcal{M}(|\nabla u|)|^q)\|^{1/q}_{L^{2/q}(-4,0;L^1(B(2)))}\\
&\quad\quad\quad\quad\quad\quad+\|\|\mathcal{M}(|\nabla u|)\|_{L^2(-4,0;L^1(B(2)))}\Big)\\
&\quad\leq
 CM^{1+\frac{3}{q}}.
\end{split}\end{equation*}% Here we used \eqref{local_study_condition3}.

Before stating the third fact, we needs the following two observations:\\

From standard Sobolev-Poincare inequality on balls (e.g. see 
%(For a proof, see Theorem 1.5.2. in 
Saloff-Coste \cite{sobolev}), we have $C$ such that
\begin{equation}\begin{split}\label{Sobolev-Poincare inequality}
\|f-\Bar{f}\|_{L^{3q/(3-q)}(B(M))}\leq C\cdot\|\nabla f\|_{L^q(B(M))}
\end{split}\end{equation}  for any $0 <M <\infty$ and for any $f$ 
whose derivatives are in $L^q_{loc}(\mathbb{R}^3)$ where $\Bar{f}
=\int_{B}fdx/|B|$ is 
the mean value on $B$. Note that $C$ is independent of $M$.\\
%(For a proof, see Theorem 1.5.2. in \cite{sobolev}) \\

  On the other hand, once we fix $M_0>0$, then there exist $C=C(M_0)$ with
the following property: \\

 For any $p$ with $1\leq p<\infty$,  for any $M\geq M_0$
 and for any $f\in L^p_{loc}(\mathbb{R}^3)$, we have
 \begin{equation}\label{lem_Maximal q}
 \|f\|_{L^p(B(M))}\leq CM^{\frac{3}{p}}\cdot
\|\mathcal{M}(|f|^p)\|^{1/p}_{L^{1}(B(2))}
\end{equation}
%Let $g=|f|^p$. Then it's enough to show that
To prove \eqref{lem_Maximal q}, it is enough to show that
 \begin{equation*}
 \|g\|_{L^1(B(M))}\leq CM^{3}\cdot
\|\mathcal{M}(g)\|_{L^{1}(B(2))}
\end{equation*}% for any $M\geq M_0$.\\
For any $z\in B(2)$,
 \begin{equation*}\begin{split}
 \int_{B(M)} |g(x)|dx&=\frac{(M+2)^3}{(M+2)^3}\cdot\int_{B(M+2)} |g(z+x)|dx\\
 &\leq (M+2)^3\mathcal{M}(g)(z)
 \leq C_{M_0}M^{3}\mathcal{M}(g)(z)
\end{split}\end{equation*} Then we take integral on $z\in B(2)$.\\

Now we states the third fact.
\begin{equation*}\begin{split}
\textbf{3.} \quad& \|w(t,\cdot)\|_{L^{3q/(3-q)}{(B(2M+2))}}\\
&\quad\quad\quad\leq  C\cdot\|\nabla w(t,\cdot)\|_{L^q(B(2M+2))}+
\|\Bar{w}(t,\cdot)\|_{L^{3q/(3-q)}(B(2M+2))}\\
&\quad\quad\quad\leq  
C\cdot M^{3/q} \cdot \|\mathcal{M}(|\nabla w|^q)(t,\cdot)\|
_{L^1(B(2))}^{1/q}
\\&\quad\quad\quad\quad+
CM^{-3}\|w(t,\cdot)\|_{L^1(B(2M+2))}\cdot CM^{3\cdot\frac{3-q}{3q}}\\
&\quad\quad\quad\leq  
C\cdot M^{3/q} \cdot \|\mathcal{M}(|\mathcal{M}(|\nabla u|)|^q)(t,\cdot)\|
_{L^1(B(2))}^{1/q}
\\&\quad\quad\quad\quad+
CM^{\frac{3}{q}-4}\|w(t,\cdot)\|_{L^1(B(2M+2))}\\
&\quad\quad\quad\leq  
C\cdot M^{3/q} \cdot \|\mathcal{M}(|\mathcal{M}(|\nabla u|)|^q)(t,\cdot)\|
_{L^1(B(2))}^{1/q}
\\&\quad\quad\quad\quad+
CM^{\frac{3}{q}-4} CM^{1+\frac{3}{1}}\cdot\Big(
\|\mathcal{M}(|\nabla w|^1)(t,\cdot)\|^{1/1}_{L^{1}(B(1))}
+\|\nabla w(t,\cdot)\|_{L^1(B(2))}\Big)\\
&\quad\quad\quad\leq  
C\cdot M^{3/q} \cdot \|\mathcal{M}(|\mathcal{M}(|\nabla u|)|^q)(t,\cdot)\|
_{L^1(B(2))}^{1/q}
\\&\quad\quad\quad\quad+
CM^{\frac{3}{q}}\Big(
\|\mathcal{M}(|\mathcal{M}(|\nabla u|)|)(t,\cdot)\|_{L^{1}(B(1))}
+\|\mathcal{M}(|\nabla u|)(t,\cdot)\|_{L^1(B(2))}\Big)\\
%&\leq C M^{\gamma}\\
\end{split}\end{equation*} 
we used \eqref{Sobolev-Poincare inequality} for the first inequality,
\eqref{lem_Maximal q} and definition of mean value 
for the second one  and
 $|\nabla w (t,x)|\leq
   C|\mathcal{M}(|\nabla u|)(t,x)|$ and the lemma \ref{lem_Maximal 2.5 or 4}
  for fourth and fifth ones respectively. So, by taking $L^2$-norm on time $[-4,0]$ with \eqref{local_study_condition3}, %as we did in $2$,
\begin{equation*}\begin{split}
&\|w\|_{L^2(-4,0;L^{\frac{3q}{3-q}}(B(2M+2)))}
\leq
 CM^{\frac{3}{q}}
\end{split}\end{equation*}  for any $M\geq 2^k$.

\begin{equation*}\begin{split}
\textbf{4.}\quad& \|w(t,\cdot)\|_{L^{q\prime}{(B(2M+2))}}\leq
\|w(t,\cdot)\|^{\theta}_{L^{q}{(B(2M+2))}}\cdot
\|w(t,\cdot)\|^{1-\theta}_{L^{3q/(3-q)}{(B(2M+2))}}\quad\quad
\end{split}\end{equation*}  where $q^\prime=q/(q-1)$ and $\theta=(4q-6)/q$.\\
Note: Because $12/8<q<2$, we have  $0<\theta<1$.
So, for any $M\geq 2^k$,
\begin{equation*}\begin{split}
&\|w\|_{L^2(-4,0;L^{q\prime}(B(2M+2)))}\\&\quad\leq
\|w\|^{\theta}_{L^2(-4,0;L^{q}{(B(2M+2)))}}\cdot
\|w\|^{1-\theta}_{L^2(-4,0;L^{3q/(3-q)}{(B(2M+2)))}}\\
&\leq C\cdot (M^{1+(3/q)})^{\theta}(M^{3/q})^{1-\theta}
= C\cdot M^{4-\frac{3}{q}   }.
%&= C\cdot M^{(1+(3/q))((4q-6)/q) +\frac{3}{q}(1-((4q-6)/q))   }\\
%&= C\cdot M^{(1+\frac{3}{q}) + \frac{3q-6}{q}   }\leq CM^{1+\frac{3}{q}}\\
\end{split}\end{equation*} % because $\frac{3q-6}{q}\leq 0$.\\

\noindent \textbf{5.} From  \eqref{lem_Maximal q}, for any $M\geq 2^k$,
\begin{equation*}\begin{split}
\quad& \|\nabla u(t,\cdot)\|_{L^q{(B(2M+2))}}
\leq C\cdot M^{3/q} \cdot \|\mathcal{M}(|\nabla u|^q)(t,\cdot)\|_{L^1(B(2))}^{1/q}.
\end{split}\end{equation*} So, for any $M\geq 2^k$, from  \eqref{local_study_condition3},
\begin{equation*}\begin{split}
\|\nabla u\|_{L^2(-4,0;L^q(B(2M+2)))}
&\leq C\cdot M^{3/q} \cdot \|\|\mathcal{M}(|\nabla u|^q)\|
_{L_x^1(B(2))}^{1/q}\|_{L_t^2(-4,0)}\\
&\leq C\cdot M^{3/q} \cdot \|\mathcal{M}(|\nabla u|^q)\|
_{L^{2/q}(-4,0;L^1(B(2)))}^{1/q}
\leq C\cdot M^{3/q}.
\end{split}\end{equation*}

\noindent Using above five results $\textbf{(1, }\cdots\textbf{, 5)}$ all together, we have for any $M\geq 2^k$,
\begin{equation*}\begin{split}
(III)
%=\|\nabla^d [{(h^{{\alpha}})_M}]*\Big((w\cdot\nabla)u\Big)\|
%_{L^{1}( -4,0;L^1(B(2)))} \\
%&\leq C \cdot\frac{1}{M^{3+d}}\cdot
%\|(w\cdot\nabla)u\|_{L^{1}( -4,0;L^1(B(2M+2)))}\\
&\leq C \cdot\frac{1}{M^{3+d}}\cdot
\|w\|_{L^{2}( -4,0;L^{q^{\prime}}(B(2M+2)))}
\|\nabla u\|_{L^{2}( -4,0;L^q(B(2M+2)))}\\
&\leq C\cdot\frac{1}{M^{3+d}}\cdot M^{4-(3/q)}\cdot M^{3/q}
= C\cdot M^{1-d}
%= C\cdot M^{\alpha/2}
\end{split}\end{equation*}
which proved the above third claim \eqref{step4_third_claim}.\\

Finally we combine three claims
 \eqref{step4_first_claim}, \eqref{step4_second_claim}
 and \eqref{step4_third_claim}:
\begin{equation*}\begin{split}
&\|(-\Delta)^{\frac{\alpha}{2}} 
\nabla^d u\|_{L^\infty(-{(1/6)}^2,0;
L^{1}(B((1/6))))}\\
&\quad\leq \|
C\Big(1+\Big|\int_{|y|\geq {(1/6)}}
\frac{\nabla^{d} u(\cdot_t,\cdot_x-y)}{|y|^{3+\alpha}}dy\Big|\Big)\|
_{L^\infty(-{(1/6)}^2,0;
L^{1}(B((1/6))))}\\
&\quad\leq C+C\sum_{j=k}^{\infty}(\frac{1}{2^{{\alpha}}})^j\cdot \|
|({(h^{{\alpha}})_{2^j}}*\nabla^d u)(\cdot_t,\cdot_x)|\|
_{L^\infty(-{(1/6)}^2,0;
L^{1}(B((1/6))))}\\
&\quad\leq C+C\sum_{j=k}^{\infty}(\frac{1}{2^{{\alpha}}})^j
\cdot (2^j)^{1-d}
\leq C+C\sum_{j=k}^{\infty}(\frac{1}{2^{d+\alpha-1}})^j
\leq C%_{n,\alpha}
\end{split}\end{equation*} because $d+\alpha-1>0$ 
from $d\geq 1$ and $\alpha>0$.\\ % (see \eqref{local_study_condition3}).\\

 \noindent By the exact same way, we can also prove that
\begin{equation*}
\|(-\Delta)^{\frac{\alpha}{2}} 
\nabla^m u\|_{L^\infty(-{(1/6)}^2,0;
L^{1}(B((1/6))))}\leq C 
%|(-\Delta)^{\frac{\alpha}{2}}\nabla^d u(0,0)|\leq C_{n,\alpha}
\end{equation*} for $m=d+1, ... ,d+4$. By repeated uses of Sobolev's inequality,
\begin{equation*}
\|(-\Delta)^{\frac{\alpha}{2}} 
\nabla^d u\|_{L^\infty(-{(1/6)}^2,0;
L^{\infty}(B((1/6))))}\leq C({d,\alpha}) 
%|(-\Delta)^{\frac{\alpha}{2}}\nabla^d u(0,0)|\leq C_{n,\alpha}
\end{equation*} and it finishes this proof of the proposition \ref{local_study_thm}.
\end{proof}

\section{Proof of the main theorem \ref{main_thm}}\label{proof_main_thm_II}

We begin this section by presenting one small lemma about pivot quantities.
After that,
the subsection \ref{prof_main_thm_II_alpha_0} covers  the part (II) for $\alpha=0$
while the subsection \ref{prof_main_thm_II_alpha_not_0} does the part (II) for $0<\alpha<2$.
Finally the part (I) for $0\leq\alpha<2$ follows in the subsection \ref{proof_main_thm_I}.

\subsection{$L^1$ Pivot quantities}
%(\mathbb{R}^3\times\mathbb{R}_{>0})
The following lemma says that $L^1$ space-time norm of our pivot quantities can be controlled by $L^2$ space norm of the initial data. These things have
the best scaling like $|\nabla u|^2$ and $|\nabla^2 P|$  among all other $a$ $priori$ quantities from $L^2$ initial data (also see \eqref{best_scaling}).

\begin{lem}\label{lemma7_problem I-n}
 There exist constant $C>0$ and $C_{d,\alpha}$ 
for integer $d\geq1$ and real $\alpha\in(0,2)$ with the following property:\\

If $(u,P)$ is a solution of (Problem I-n) for some $1\leq n \leq \infty$, then we have
\begin{equation*} 
 \int_0^{\infty}\int_{\mathbb{R}^3}\big(|\nabla u(t,x)|^2 +
 |\nabla^2P(t,x)| + |\mathcal{M}(|\nabla u|)(t,x)|^2\big)dxdt 
\leq C\|u_0\|^2_{L^2(\mathbb{R}^3)}
\end{equation*} and 
\begin{equation*}\begin{split} 
 \int_0^{\infty}\int_{\mathbb{R}^3}&\Big(
|\mathcal{M}(\mathcal{M}(|\nabla u|))|^2+|\mathcal{M}(|\nabla u|^q)|^{2/q}+
|\mathcal{M}(|\mathcal{M}(|\nabla u|)|^q)|^{2/q}\\
&+
\sum_{m=d}^{d+4} \sup_{\delta>0}(|(\nabla^{m-1}{h^{\alpha})_\delta}*\nabla^2 P|)
\Big)dxdt 
\leq C_{d,\alpha}\|u_0\|^2_{L^2(\mathbb{R}^3)}
\end{split}\end{equation*} for any integer $d\geq1$ and any real $\alpha\in(0,2)$  
where $q=q(\alpha)$ is defined by $12/(\alpha+6)$.
%, and the definition of $h^{\alpha}$ can be found in \eqref{property_h}.
%$h^{m+\alpha}$ was defined in section \ref{prelim}
\begin{rem}
The definitions of $h^{\alpha}$ and
 $(\nabla^{m-1}{h^{\alpha})_\delta}$ can be found around \eqref{property_h}.
\end{rem}
\begin{rem}
In the following proof, we will see that every quantity in the left hand sides of the above two estimates can be controlled by
dissipation of energy $\|\nabla u\|_{L^2((0,\infty)\times\mathbb{R}^3)}^{2}$
only.
It explains the latter part of the remark \ref{rmk_dissipation of energy}. 
\end{rem}

\end{lem}

\begin{proof}
From \eqref{energy_eq_Problem I-n},
\begin{equation*}\begin{split}
\|\nabla u\|^2_{L^2(0,\infty;L^2(\mathbb{R}^3))}
&\leq \|u_0*\phi_{\frac{1}{n}}\|_{L^2(\mathbb{R}^3)}^2\leq \|u_0\|_{L^2(\mathbb{R}^3)}^2.
\end{split}\end{equation*}

\noindent For the pressure term, 
we use boundedness of the Riesz transform on Hardy space $\mathcal{H}$ 
and compensated compactness result in 
Coifman,  Lions,  Meyer and  Semmes  \cite{clms}:
%Or if $u$ is a solution of \eqref{navier_Problem I-n} in (Problem I-n) then 
\begin{equation}\begin{split}\label{pressure_hardy_Problem I-n} 
\|\nabla^2 P\|_{L^1(0,\infty;L^1(\mathbb{R}^3))}&\leq
\|\nabla^2 P\|_{L^1(0,\infty;\mathcal{H}(\mathbb{R}^3))}
\leq C\|\Delta  P\|_{L^1(0,\infty;\mathcal{H}(\mathbb{R}^3))}\\
&=\|\ebdiv\ebdiv
\Big( (u * \phi_{1/n})\otimes  u\Big)\|_{L^1(0,\infty;\mathcal{H}(\mathbb{R}^3))}\\
&\leq C\cdot\|\nabla (u * \phi_{1/n})\|_{L^2(0,\infty;L^2(\mathbb{R}^3))}
\|\nabla u\|_{L^2(0,\infty;L^2(\mathbb{R}^3))}\\
%&=C\cdot\|(\nabla u) * \phi_{1/n}\|_{L^2(0,\infty;L^2(\mathbb{R}^3))}
%\|\nabla u\|_{L^2(0,\infty;L^2(\mathbb{R}^3))}\\
&\leq C\cdot\|\nabla u\|_{L^2(0,\infty;L^2(\mathbb{R}^3))}^2
\leq C\|u_0\|_{L^2(\mathbb{R}^3)}^2.
\end{split}\end{equation} 

\noindent  For Maximal functions, %by $L^2$ boundedness of Maximal operator,
\begin{equation*}\begin{split}
\| \mathcal{M}(\mathcal{M}(|\nabla u|))\|^2_{L^2(0,\infty;L^2(\mathbb{R}^3))}&\leq
C\cdot\|\mathcal{M}(|\nabla u|)\|^2_{L^2(0,\infty;L^2(\mathbb{R}^3))}\\
&\leq C\cdot\|\nabla u\|^2_{L^2(0,\infty;L^2(\mathbb{R}^3))}\\
&\leq C\cdot\|u_0\|^2_{L^2(\mathbb{R}^3)}.\\
\end{split}\end{equation*} 

\noindent Let $d\geq1$ and $0<\alpha<2$ and take $q=12/(\alpha+6)$. From 
%$(12/7)<q<2$ and 
$1<(2/q)<(4/3)$,
%by $L^p$ boundedness of Maximal operator for $1<p<\infty$,
\begin{equation*}\begin{split}
\|\mathcal{M}(|\nabla u|^q)\|^{2/q}_{L^{2/q}(0,\infty;L^{2/q}(\mathbb{R}^3))}
&\leq C\cdot\||\nabla u|^q\|^{2/q}_{L^{2/q}(0,\infty;L^{2/q}(\mathbb{R}^3))}\\
&= C\cdot\|\nabla u\|^2_{L^{2}(0,\infty;L^2(\mathbb{R}^3))}\\
&\leq C\cdot\|u_0\|^{2}_{L^2(\mathbb{R}^3)}
\end{split}\end{equation*} and 
\begin{equation*}\begin{split}
\| \mathcal{M}(|\mathcal{M}(|\nabla u|)|^q)
\|^{2/q}_{L^{2/q}(0,\infty;L^{2/q}(\mathbb{R}^3))}
&\leq C\cdot\||\mathcal{M}(|\nabla u|)|^q\|^{2/q}_{L^{2/q}(0,\infty;L^{2/q}(\mathbb{R}^3))}\\
&\leq C\cdot\|\mathcal{M}(|\nabla u|)\|^{2}_{L^{2}(0,\infty;L^{2}(\mathbb{R}^3))}\\
&\leq C\cdot\|\nabla u\|^2_{L^{2}(0,\infty;L^2(\mathbb{R}^3))}\\
&\leq C\cdot\|u_0\|^{2}_{L^2(\mathbb{R}^3)}
\end{split}\end{equation*}
where   $C$ depends 
only 
on $\alpha$.\\

\noindent Thanks to the property of Hardy space \eqref{hardy_property}
with \eqref{pressure_hardy_Problem I-n}, we have
\begin{equation*}\begin{split}
\sum_{m=d}^{d+4}\| \sup_{\delta>0}(|(\nabla^{m-1}{h^{\alpha})_\delta}
*\nabla^2 P|)\|_{L^{1}(0,\infty;L^{1}(\mathbb{R}^3))}
&\leq\sum_{m=d}^{d+4}C\|\nabla^2 P
\|_{L^1(0,\infty;\mathcal{H}(\mathbb{R}^3))}\\
%&\leq C\|\nabla^2 P\|_{L^1(0,\infty;$\mathcal{H}$(\mathbb{R}^3))}\\
&\leq C\|u_0\|_{L^2(\mathbb{R}^3)}^2
\end{split}\end{equation*} where  the above $C$ depends only on $d$ and $\alpha$.
\end{proof}

We are ready to prove the main theorem \ref{main_thm}.% part (II) first.

\begin{rem}\label{n_infty_rem}
 In the following subsections \ref{prof_main_thm_II_alpha_0}
and \ref{prof_main_thm_II_alpha_not_0}, we consider  solutions of 
{(Problem I-n)}  for positive integers $n$. % first. 
However it will be clear that
every computation in these subsections can also be verified for the case $n=\infty$ 
once we assume that the smooth solution $u$ of the Navier-Stokes exists.
This $n=\infty$ case (the original Navier-Stokes)
will be covered in the subsection \ref{proof_main_thm_I}.  \\
\end{rem}

We focus on the $\alpha=0$ case of the part (II) first.\\
\subsection{Proof of theorem \ref{main_thm} part (II)
 for $\alpha=0$ case}\label{prof_main_thm_II_alpha_0}

\begin{proof}[Proof of proposition \ref{main_thm} part (II) for the $\alpha=0$ case]

 \ \\

Let any $u_0$ of \eqref{initial_condition} be given. 
From the Leray's construction, 
there exists the   $C^{\infty}$ solution sequence $\{u_n\}_{n=1}^{\infty}$ of 
{(Problem I-n)}  on $(0,\infty)$ with 
corresponding pressures $\{P_n\}_{n=1}^{\infty}$. 
%It is well-known that
%a subsequence $u_{n_k}$ converges to a suitable weak solution as $n$ goes to $\infty$.
From now on, our goal is to make an estimate for $\nabla^d u_n$ which is uniform in $n$.\\

\noindent For each  $n$, $\epsilon>0$, $t>0$ and $x\in\mathbb{R}^3$, 
define a new flow $X_{n,\epsilon}(\cdot,t,x)$ by solving
\begin{equation*}\begin{split}
&\frac{\partial X_{n,\epsilon} }{\partial s}
 (s,t,x) = u_{n}*\phi_{\frac{1}{n}}*\phi_{\epsilon}(s,X_{n,\epsilon}(s,t,x))
\quad \mbox{ for } s\in[0,t],\\
&X_{n,\epsilon}(t,t,x) =x.
\end{split} 
\end{equation*}
%Note that $X$ depends on $ n$ and $\epsilon$.\\
For convenience, we define $F_n(t,x)$ and $g_n(t)$.
\begin{equation*}\begin{split}   
  F_n(t,x)& = \big(|\nabla u_n|^2 + |\nabla^2P_n|
 + |\mathcal{M}(\nabla u_n)|^2\big)(t,x),\quad              
  g_n(t) = \int_{\mathbb{R}^3}F_n(t,x)dx.
  \end{split}\end{equation*}
%Note that $\nabla^2P$ is well-defined integrabel function. See
We 
%focus on a space set first and
 define for $n$, $t>0$ and  $0<4\epsilon^2\leq t$
\begin{equation*}
 \Omega_{n,\epsilon,t} = \{ x\in \mathbb{R}^3\quad |\quad \frac{1}{\epsilon}
\int_{t-4{\epsilon}^2}^{t}\int_{B({2\epsilon})}F_n(s,X_{n,\epsilon} 
(s,t,x)+y)dyds\leq{\bar{\eta}}\}
\end{equation*} where $\bar{\eta}$ comes from the proposition
 \ref{local_study_thm}.
%Like we did at lemma 8 in \cite{vas:higher}, we have
We measure size of $(\Omega_{n,\epsilon,t})^C$:
\begin{equation}\begin{split}\label{omega_estimate}
 |(\Omega_{n,\epsilon,t})^C|
&= |\{ x\in \mathbb{R}^3\quad |\quad \frac{1}{\epsilon}
\int_{t-4{\epsilon}^2}^{t}\int_{B({2\epsilon})}F_n(s,X_{n,\epsilon}(s,t,x)+y)dyds > {\bar{\eta}}\}|\\
&\leq\frac{1}{{\bar{\eta}}}\int_{\mathbb{R}^3}\Big(\frac{1}{\epsilon}
\int_{t-4{\epsilon}^2}^{t}\int_{B({2\epsilon})}F_n(s,X_{n,\epsilon}(s,t,x)+y)dyds\Big)dx\\
%&=\frac{1}{{\bar{\eta}}\epsilon}\Big(\int_{B_{2\epsilon}}
%\int_{t-4{\epsilon}^2}^{t}\int_{\mathbb{R}^3}F_n(s,X_{n,\epsilon}(s,t,x)+y)dxdsdy\Big)\\
&=\frac{1}{{\bar{\eta}}\epsilon}\Big(\int_{B({2\epsilon})}
\int_{-4{\epsilon}^2}^{0}\int_{\mathbb{R}^3}F_n(t+s,X_{n,\epsilon}(t+s,t,x)+y)dxdsdy\Big)\\
&=\frac{1}{{\bar{\eta}}\epsilon}\Big(\int_{B({2\epsilon})}
\int_{-4{\epsilon}^2}^{0}\int_{\mathbb{R}^3}F_n(t+s,z+y)dzdsdy\Big)\\
&\leq\frac{1}{{\bar{\eta}}\epsilon}\Big(\int_{B({2\epsilon})}1 dy\Big)\Big(
\int_{-4{\epsilon}^2}^{0}\int_{\mathbb{R}^3}F_n(t+s,\bar{z})d\bar{z}ds\Big)\\
&\leq\frac{C\epsilon^2}{{\bar{\eta}}}\Big(
\int_{-4{\epsilon}^2}^{0}\int_{\mathbb{R}^3}F_n(t+s,\bar{z})d\bar{z}ds\Big)\\
  &\leq C\frac{\epsilon^4}{\bar{\eta}}\Big(\frac{1}{4\epsilon^2}
\int_{-4\epsilon^2}^0g_n(t+s)ds\Big) 
%\leq  \epsilon^4\cdot\mathcal{M}^{(t)}
%(\frac{C}{\bar{\eta}}g_n\cdot \mathbf{1}_{(0,T)})(t)\\
\leq  \epsilon^4\mathcal{M}^{(t)}
(\frac{C}{\bar{\eta}}g_n\cdot \mathbf{1}_{(0,\infty)})(t)
=\epsilon^4 \tilde{g}_n(t)\
\end{split}\end{equation} where  $\tilde{g}_n= \mathcal{M}^{(t)}
(\frac{C}{\bar{\eta}}g_n\cdot \mathbf{1}_{(0,\infty)})$ and $\mathcal{M}^{(t)}$ is 
the Maximal function in $\mathbb{R}^1$. For the third inequality, we used the fact
that $X_{n,\epsilon}(\cdot,t,x)$ is incompressible.
From the fact that the Maximal operator is bounded from $L^1$ to $L^{1,\infty}$ together with
the lemma \ref{lemma7_problem I-n},
%Moreover,
$\|\tilde{g}_n(\cdot)\|_{L^{1,\infty}(0,\infty)}\leq 
\frac{C}{\bar{\eta}}\|g_n(\cdot)\|_{L^{1}(0,\infty)}\leq
\frac{C}{\bar{\eta}}\|u_0\|^2_{L^2(\mathbb{R}^3)}$.\\

Now we fix $n, t, \epsilon$ and $ x$ with $n\geq1$, $0<t<\infty$ , $0<4\epsilon^2\leq t$ and $x\in\Omega_{n,\epsilon,t}$. 
We define ${v}, {Q}$ on $(-4,\infty)\times\mathbb{R}^3$ by
using the Galilean invariance:
\begin{equation}\begin{split}\label{special designed scaling}
 {v}(s,y) =& \epsilon u_n(t+\epsilon^2 s,X_{n,\epsilon}(t+\epsilon^2 s,t,x)+\epsilon y)\\
 &- \epsilon (u_n*\phi_{\epsilon})(t+\epsilon^2s,X_{n,\epsilon}(t+\epsilon^2s,t,x))\\
{Q}(s,y) =& \epsilon^2 P_n(t+\epsilon^2 s,X_{n,\epsilon}(t+\epsilon^2 s,t,x)+\epsilon y)\\
 &+ \epsilon y \partial_s[ (u_n*\phi_{\epsilon})(t+\epsilon^2s,
X_{n,\epsilon}(t+\epsilon^2s,t,x))].
\end{split}\end{equation} 
\begin{rem}
This specially designed
$\epsilon$-scaling will give the mean zero property to both the velocity
 and the advection
velocity of the resulting equation \eqref{special designed scaling result}.
\end{rem}

\noindent Let us denote $\Box$ and $\Diamond$ by $\Box= \big(t+\epsilon^2 s,X_{n,\epsilon}(t+\epsilon^2 s,t,x)+\epsilon y\big)$ and 
$\Diamond=\big(t+\epsilon^2 s,X_{n,\epsilon}(t+\epsilon^2 s,t,x)\big)$, 
respectively.
Then the chain rule gives us
\begin{equation*}\begin{split}
&   \partial_s{v}(s,y) =
 \epsilon^3\partial_t(u_n)(\Box)   
+  \epsilon^3   \big((u_{n}*\phi_{\frac{1}{n}}*\phi_{\epsilon})(\Diamond)\cdot\nabla\big)u_n(\Box)
 - \epsilon  \partial_s[(u_{n}*\phi_{\epsilon})(\Diamond)],\\
&\big({v} *_y \phi_{\frac{1}{n\epsilon}}\big)(s,y)
=\epsilon (u_n*\phi_{\frac{1}{n}})(\Box)-\epsilon(u_{n}*\phi_{\epsilon})(\Diamond),\\
&\int_{\mathbb{R}^3}\big({v} *_y \phi_{\frac{1}{n\epsilon}}\big)(s,z)\phi(z)dz
=\epsilon (u_n*\phi_{\frac{1}{n}}*\phi_{\epsilon})(\Diamond)-\epsilon(u_{n}*\phi_{\epsilon})(\Diamond),\\
&\bigg(\Big(\big({v} *_y \phi_{\frac{1}{n\epsilon}}\big)(s,y) - 
\int_{\mathbb{R}^3}\big({v} *_y \phi_{\frac{1}{n\epsilon}}\big)(s,z)\phi(z)dz\Big)\cdot\nabla\bigg)
{v}(s,y)=\\
&\epsilon^3 \Big((u_n*\phi_{\frac{1}{n}})(\Box)\cdot\nabla\Big)u_n(\Box)
-\epsilon^3 \Big((u_n*\phi_{\frac{1}{n}}*\phi_{\epsilon})(\Diamond)\cdot\nabla\Big)u_n(\Box),\\
&-\Delta_y {v}(s,y) = -\epsilon^3\Delta_y u_n(\Box) \mbox{ and} \\
&\nabla_y{Q}(s,y)= \epsilon^3 \nabla P_n(\Box) 
+ \epsilon \partial_s[ (u_n*\phi_{\epsilon})(\Diamond))].
\end{split}\end{equation*} 

\noindent Thus, for $(s,y)\in (-4,\infty)\times\mathbb{R}^3$,
\begin{equation}\begin{split}\label{special designed scaling result}
\Big[\partial_s{v}+
\bigg(\Big(\big({v} * \phi_{\frac{1}{n\epsilon}}\big) - 
\int\big({v} * \phi_{\frac{1}{n\epsilon}}\big)\phi \Big)\cdot\nabla\bigg)
{v} +\nabla{Q}-\Delta {v}\Big](s,y)=0.
\end{split}\end{equation} As a result,
$({v}(\cdot_s,\cdot_y),{Q}(\cdot_s,\cdot_y))$ is a solution of (Problem II-$\frac{1}{n\epsilon}$).\\

From definition of the Maximal function, we can verify that
$|\mathcal{M}(\nabla {v})|^2$
behaves like $|\nabla v|^2$ under the scaling 
in the following sense:
\begin{equation}\begin{split}\label{Maximal_scaling}
\mathcal{M}(\nabla {v})(s,y)
%&=\sup_{M>0}\frac{C}{M^3}\int_{B(M)}(\nabla {v})(s,y+z)dz\\
&=\sup_{M>0}\frac{C}{M^3}\int_{B(M)}\epsilon^2(\nabla {u_n})
\big(t+\epsilon^2 s,X_{n,\epsilon}(t+\epsilon^2 s,t,x)+\epsilon (y+z)\big)dz\\
&=\sup_{\epsilon M>0}\frac{C}{\epsilon^3M^3}\int_{B(\epsilon M)}\epsilon^2(\nabla {u_n})
\big(t+\epsilon^2 s,X_{n,\epsilon}(t+\epsilon^2 s,t,x)+\epsilon y+\bar{z}\big)d\bar{z}\\
%&=\epsilon^2\mathcal{M}(\nabla {u_n})
%\big(t+\epsilon^2 s,X_{n,\epsilon}(t+\epsilon^2 s,t,x)+\epsilon y\big)\\
&=\epsilon^2\mathcal{M}(\nabla {u_n})
(\Box)
\end{split}\end{equation} As a result,
\begin{equation*}\begin{split}
&\int_{-4}^{0}\int_{B(2)}\big(|\nabla {v}(s,y)|^2 +
 |\nabla^2{Q}(s,y)| + |\mathcal{M}(\nabla {v})(s,y)|^2\big)dyds \\
 &=\epsilon^4\int_{-4}^{0}\int_{B(2)}\Big[|\nabla {u_n}|^2 +
 |\nabla^2{P_n}| + |\mathcal{M}(\nabla {u_n})|^2\Big](\Box)dyds \\ 
  &=\epsilon^4\int_{-4}^{0}\int_{B(2)}F_n(\Box)dyds \\ 
  &=\epsilon^{-1}\int_{t-4\epsilon^2}^{t}\int_{B(2\epsilon)}F_n
 \big(s,X_{n,\epsilon}(s,t,x)+y\big)dyds 
 \leq {\bar{\eta}}
\end{split}\end{equation*} where the first equality comes from the definition of $
(v,Q)$ and the second one follows the change of variable  
$\big(t+\epsilon^2 s,\epsilon y\big)\rightarrow (s,y)$. 
%(See proposition 9 in \cite{vas:higher})\\
Moreover, it satisfies
\begin{equation}\label{local_study_condition_satisfied}\begin{split}
&\int_{\mathbb{R}^3}\phi(z){v}(s,z)dz = 0, \quad -4<s<0.
\end{split}\end{equation} 
So $(v,Q)$ satisfies all conditions (\ref{local_study_condition1},
\ref{local_study_condition2})
 in the proposition \ref{local_study_thm} with $r=1/(n\epsilon)\in[0,\infty)$. \\

The conclusion of the proposition \ref{local_study_thm} implies that if $x\in\Omega_{n,\epsilon,t}$ for some $n,t$ and $\epsilon$ such that
 $4\epsilon^2\leq t$  then
$|\nabla^d {v}(0,0)|\leq C_{d}.$
%\begin{equation}
%|\nabla^d {v}(0,0)|\leq C_{d}\\ 
%\end{equation}.
%Observe that  $(-\Delta)^{\frac{\alpha}{2}}\nabla^d {v}(0,0)=
%\epsilon^{d+\alpha+1}(-\Delta)^{\frac{\alpha}{2}}\nabla^d u_n(t,x)$\\
As a result, using  $\nabla^d {v}(0,0)=
\epsilon^{d+1} \nabla^d u_n(t,x)$ for any integer $d\geq 1$,
%Finally, for any $0<t<\infty$ and for any $0<4\epsilon^2\leq t$
 we have %(see page 19 ``5.From local to global'' of \cite{vas:higher})
\begin{equation*}
|\{x\in \mathbb{R}^3 | \quad|\nabla^d u_n(t,x)|> 
\frac{C_{d}}{\epsilon^{d+1}}\}|\leq|\Omega_{n,\epsilon,t}^C|
\leq  \epsilon^4\cdot\tilde{g}_n(t).\
\end{equation*}% This estimates holds as long as $4\epsilon^2<t$.\\

\noindent Let $K$ be any open bounded subset in $\mathbb{R}^3$.
Also  define $p=4/(d+1)$. Then for any $t>0$, 
\begin{equation*}
 \beta^p\cdot\Big|\{x\in K: |
(\nabla^d u_n)(t,x)|
> \beta\}\Big|\leq 
\begin{cases}& \beta^p\cdot|K|
,\quad\mbox{ if } \beta\leq C\cdot t^{-2/p}\\
 & C\cdot  \tilde{g}_n(t) ,
  \quad\mbox{ if } \beta>C\cdot t^{-2/p}.\\
\end{cases}
\end{equation*} Thus,
 \begin{equation*}
 \|(\nabla^d u_n)(t,\cdot)\|^p_{L^{p,\infty}(K)} 
\leq C\cdot\max\big( \tilde{g}_n(t),\frac{|K|}{t^2} \big)
     \end{equation*}

\noindent We pick any $t_0>0$. If we take ${L^{1,\infty}(t_0,T)}$-norm 
to the above inequality, then we obtain
 \begin{equation}\begin{split}\label{integer_estimate}
\quad\|\nabla^d u_n\|^p_{L^{p,\infty}(t_0,\infty;L^{p,\infty}(K))}
 &\leq C\Big(\| \tilde{g}_n\|_{L^{1,\infty}{(0,\infty)}} + |K|\cdot\|\frac{1}
{{|\cdot|}^2}\|_{L^{1,\infty}(t_0,\infty)}\Big)\\
&\leq C\Big(\|u_0\|^2_{L^{2}(\mathbb{R}^3)} + 
\frac{|K|}{t_0}\Big)
\end{split}\end{equation} where $C$ depends only on $d\geq1$.\\

We observe that
the above 
estimate is uniform in $n$. %, and we recall that
It is well known that both $\nabla u$ and $\nabla^2 u$ are locally integrable functions for any suitable weak solution $u$
which can be obtained by a limiting argument of $u_n$ (e.g. see 
 Lions \cite{lions}). Thus, the
 above estimates \eqref{integer_estimate} holds
even for $u$ with $d=1,2$.\\
\begin{rem}
In fact, for the case $d=1$, the above estimate says $\nabla u\in L^{2,\infty}_{loc}$, which is useless because
we know a better estimate $\nabla u\in L^2$.%((0,\infty)\times\mathbb{R}^3)$.
\end{rem}
\begin{rem}
For $d \geq 3$, 
 the above estimate \eqref{integer_estimate} does not give us any 
 direct information  about higher derivatives $\nabla^d u$ of a weak 
solution $u$ 
because full regularity of weak solutions is still open, so
$\nabla^d u$ may not be  locally integrable for $d\geq3$.
%they may not be  locally integrable functions at all.
 %in the sense of distributions
% due to lack of weak-compactness of $L^p$ for $p\leq1$. 
Instead,
the only thing we can say is that, for $d\geq3$, higher derivatives $\nabla^d u_n$
   of a Leray's approximation $u_n$ have    $L^{4/(d+1),\infty}_{loc}$ bounds which are uniform in $n\geq1$.
% while we do not know an existence of $\nabla^3 u$ as a locally integrable
% distributions.
\end{rem}
\end{proof}

From now on, we will prove the $0<\alpha<2$ case of the part (II).
\subsection{Proof of theorem \ref{main_thm} part (II)
 for $0<\alpha<2$ case}\label{prof_main_thm_II_alpha_not_0}
\begin{proof}[Proof of proposition \ref{main_thm} part (II) for the $0<\alpha<2$ case]

\ \\

We fix $d\geq 1$ and $0<\alpha<2$. 
Then, 
for any positive integer $n$, any $t>0$ and $x\in\mathbb{R}^3$, we denote $F_n(t,x)$ in this time by:
\begin{equation*}\begin{split}   
  F_n(t,x) = \Big(&|\nabla u_n(t,x)|^2 + |\nabla^2P_n(t,x)| + 
|\mathcal{M}(\nabla u_n)(t,x)|^2\\&+
|\mathcal{M}(\mathcal{M}(|\nabla u_n|))|^2+
(\mathcal{M}(|\mathcal{M}(|\nabla u_n|)|^q))^{2/q}\\+
&|\mathcal{M}(|\nabla u_n|^q)|^{2/q}+
\sum_{m=d}^{d+4} \sup_{\delta>0}(|(\nabla^{m-1}{h^{\alpha})_\delta}*\nabla^2 P|)
\Big).
\end{split}\end{equation*} 
We use the same definitions for $g_n$,  $\tilde{g}_n$, $X_{n,\epsilon}$ and $\Omega_{n,\epsilon,t}$   
of the previous section \ref{prof_main_thm_II_alpha_0} for the case $\alpha=0$.  
Note that they depend on $d$ and 
$\alpha$, and  we have 
%\begin{equation}
%\|\tilde{g}_n\|_{L^{1,\infty}(0,\infty)}\leq 
%\frac{C({d,\alpha})}{\bar{\eta}}\cdot\|u_0\|^2_{L^2(\mathbb{R}^3)}. \\
%\end{equation} 
$\|\tilde{g}_n\|_{L^{1,\infty}(0,\infty)}\leq 
\frac{C_{d,\alpha}}{\bar{\eta}}\cdot\|u_0\|^2_{L^2(\mathbb{R}^3)}$
from the lemma \ref{lemma7_problem I-n}. \\

Now we pick any $x\in\Omega_{n,\epsilon,t}$ and any $\epsilon$ such that
 $4\epsilon^2\leq t$, and define $v$ and $Q$
 as the previous section \ref{prof_main_thm_II_alpha_0} (see \eqref{special designed scaling}).\\

\noindent In order to  follow the previous subsection  \ref{prof_main_thm_II_alpha_0},
only thing which remains is  to verify if every quantity in %\eqref{local_study_condition3}
$F_n(t,x)$
 has the same scaling with $|\nabla v|^2$ after the transform
\eqref{special designed scaling}. %as $\nabla v$ has.  
For Maximal of Maximal functions,
\begin{equation*}\begin{split}
&\mathcal{M}(\mathcal{M}(|\nabla v|))(s,y)\\
&=\sup_{M>0}\frac{C}{M^3}\int_{B(M)}\mathcal{M}(|\nabla v|)(s,y+z)dz\\
&=\sup_{M>0}\frac{C}{M^3}\int_{B(M)}\epsilon^2\mathcal{M}
(|\nabla u_n|)\big(t+\epsilon^2 s,X_{n,\epsilon}(t+\epsilon^2 s,t,x)+\epsilon (y+z)\big)dz\\
&=\epsilon^2\mathcal{M}(\mathcal{M}(|\nabla u_n|))(\Box).\\
\end{split}\end{equation*} where $\Box=\big(t+\epsilon^2 s,X_{n,\epsilon}(t+\epsilon^2 s,t,x)+\epsilon y\big)$
and  we used the idea of \eqref{Maximal_scaling}
for second and third equalities. Likewise,
%\begin{equation}\begin{split}
%&\mathcal{M}(\mathcal{M}(|\nabla v|^q))(s,y)\\
%&=\epsilon^{2q}\cdot\mathcal{M}(\mathcal{M}(|\nabla u_n|^q))(\Box).\\
%\end{split}\end{equation}  
$\mathcal{M}(|\nabla v|^q)(s,y)
=\epsilon^{2q}\cdot\mathcal{M}(|\nabla u_n|^q)(\Box)$
and 
$\mathcal{M}(|\mathcal{M}(|\nabla v|)|^q)(s,y)
=\epsilon^{2q}\cdot\mathcal{M}(|\mathcal{M}(|\nabla u_n|)|^q)(\Box)$. \\

\noindent Also, we have for any function $\mathcal{G}\in C_0^\infty$,
\begin{equation*}\begin{split}
%&\mathcal{M}(\mathcal{M}(|\nabla v|))(s,y)=\\
\sup_{\delta>0}&(|{\mathcal{G}_{\delta}}*\nabla^2 Q|)(s,y)=
\sup_{\delta>0}\Big|\int_{\mathbb{R}^3}\frac{1}{\delta^{3}}{\mathcal{G}}(\frac{z}{\delta})\cdot
(\nabla^2 Q)(s,y-z)dz\Big|\\
&=
\sup_{\delta>0}\Big|\int_{\mathbb{R}^3}\frac{\epsilon^4}{\delta^{3}}{\mathcal{G}}(\frac{z}{\delta})\cdot
(\nabla^2 P_n)\big(t+\epsilon^2 s,X_{n,\epsilon}(t+\epsilon^2 s,t,x)+\epsilon (y-z)\big)dz\Big|\\
&=
\sup_{\delta>0}\Big|\int_{\mathbb{R}^3}\frac{\epsilon^4}{\epsilon^3
\delta^{3}}{\mathcal{G}}(\frac{z}{\epsilon\delta})\cdot
(\nabla^2 P_n)\big(t+\epsilon^2 s,X_{n,\epsilon}(t+\epsilon^2 s,t,x)+\epsilon y-z\big)dz\Big|
\end{split}\end{equation*}
\begin{equation*}\begin{split}
&=
\sup_{\epsilon\delta>0}\Big|\int_{\mathbb{R}^3}\epsilon^4
{\mathcal{G}}_{\epsilon\delta}(z)\cdot
(\nabla^2 P_n)\big(t+\epsilon^2 s,X_{n,\epsilon}(t+\epsilon^2 s,t,x)+\epsilon y-z\big)dz\Big|\\
&=
\sup_{\epsilon\delta>0}\epsilon^4\Big|
\Big({\mathcal{G}}_{\epsilon\delta}*
(\nabla^2 P_n)\Big)\big(t+\epsilon^2 s,X_{n,\epsilon}(t+\epsilon^2 s,t,x)+\epsilon y\big)\Big|\\
&=\epsilon^4
\sup_{\delta>0}\Big|
{\mathcal{G}_{\delta}}*
(\nabla^2 P_n)\Big|(\Box).
\end{split}\end{equation*} Thus by taking $\mathcal{G}=(\nabla^{m-1}{h^{\alpha}})$, 
we have \begin{equation*}\begin{split}
\sup_{\delta>0}(|{(\nabla^{m-1}{h^{\alpha})_\delta}}*\nabla^2 Q|)(s,y)
&=\epsilon^4
\sup_{\delta>0}\Big|
{(\nabla^{m-1}{h^{\alpha})_\delta}}*
(\nabla^2 P_n)\Big|\big(\Box\big).
\end{split}\end{equation*}

\noindent  As a result, 
we have 
\begin{equation*}\begin{split}
&\int_{-4}^{0}\int_{B(2)}\Big[|\nabla {v}|^2 +
 |\nabla^2{Q}| + |\mathcal{M}(\nabla {v})|^2+\\
&\quad+|\mathcal{M}(\mathcal{M}(|\nabla v|))|^2+
|\mathcal{M}(|\mathcal{M}(|\nabla v|)|^q)|^{q/2}\\
&\quad+|\mathcal{M}(|\nabla v|^q)|^{2/q}+
\sum_{m=d}^{d+4} \sup_{\delta>0}(
|{(\nabla^{m-1}{h^{\alpha})_\delta}}*\nabla^2 Q|)
\Big](s,y) 
 dyds \\
%&=\epsilon^4 \int_{-4}^{0}\int_{B(\Diamond)}\Big[|\nabla {u_n}|^2 +
% |\nabla^2{P_n}| + |\mathcal{M}(\nabla {u_n})|^2+\\
%&\quad+\mathcal{M} (\mathcal{M}(|\nabla u_n|))+
%\mathcal{M}(\mathcal{M}(|\nabla u_n|^q))\\&\quad+
%\mathcal{M}(|\nabla u_n|^q)+
%\sum_{m=d}^{d+4} \sup_{\delta>0}(
%|{(\nabla^{m-1}{h^{\alpha})_\delta}}*\nabla^2 P_n|)
%\Big](\Box)
% dyds \\
  &=\epsilon^4\int_{-4}^{0}\int_{B(2)}F_n(\Box)dyds \\ 
  &=\epsilon^{-1}\int_{t-4\epsilon^2}^{t}\int_{B(2\epsilon)}F_n
 \big(s,X_{n,\epsilon}(s,t,x)+y\big)dyds 
 \leq {\bar{\eta}}.\\
\end{split}\end{equation*} 

%From $\|u_n\|_{L^{\infty}(0,\infty;L^2(\mathbb{R}^3))}\leq
%\|u_0\|_{L^2(\mathbb{R}^3)}<\infty$, 
%we have $v\in L^{\infty}(-4,\infty;L^2(\mathbb{R}^3))$ and 
\noindent Then $(v,Q)$ satisfies condition 
\eqref{local_study_condition3} as well as 
\eqref{local_study_condition1} and
\eqref{local_study_condition2}
 of the proposition \ref{local_study_thm}
with $r=1/(n\epsilon)\in[0,\infty)$. In sum if $x\in\Omega_{n,\epsilon,t}$
 and if $4\epsilon^2\leq t$, then
\begin{equation*}
| (-\Delta)^{\alpha/2}\nabla^d {v}(0,0)|\leq C_{d,\alpha}.
\end{equation*}

 Because $u_n$ is a smooth solution of (Problem I-n),
$(-\Delta)^{\alpha/2}\nabla^d u_n$ is not only a distribution
but also a locally integrable function. 
Indeed, from a boot-strapping argument, it is 
easy to show that $\nabla^d u_n(t)$ has a good behavior at infinity
which is required
in order to use the integral 
representation \eqref{fractional_integral} pointwise. For example, 
 $(C^2\cap W^{2,\infty})$ is enough (For a better approach, see  Silvestre \cite{silve:fractional}). Also it can be easily verified that 
the resulting function $(-\Delta)^{\alpha/2}[\nabla^d u_n(t,\cdot)](x)$
from the integral representation  \eqref{fractional_integral}
satisfies the definition in the remark \ref{frac_rem}. \\

\noindent 
As a result, it makes sense to talk about pointwise values 
of $(-\Delta)^{\alpha/2}\nabla^d u_n$. Thus,
from the simple observation: %relation between $u_n$ and $v$, 
%the following equation is valid
 for any integer $d\geq 1$ and any real $0<\alpha<2$,
\begin{equation*}
 (-\Delta)^{\alpha/2}\nabla^d {v}(0,0)=
\epsilon^{d+\alpha+1} (-\Delta)^{\alpha/2}\nabla^d u_n(t,x),
\end{equation*} 
% $(-\Delta)^{\alpha/2}\nabla^d {v}(0,0)=
%\epsilon^{d+\alpha+1} (-\Delta)^{\alpha/2}\nabla^d u_n(t,x)$
% for $d\geq1$, 
we can
deduce the following set inclusion:
\begin{equation}\label{fractional_inclusion}
\{x\in \mathbb{R}^3 | \quad|(-\Delta)^{\frac{\alpha}{2}}\nabla^d u_n(t,x)|> 
\frac{C_{d,\alpha}}{\epsilon^{d+\alpha +1}}\}\qquad\subset\qquad\Omega_{n,\epsilon,t}^C.
\end{equation}

\noindent Thus we have 
for any $0<t<\infty$ and for any $0<4\epsilon^2\leq t$
\begin{equation*}
|\{x\in \mathbb{R}^3 | \quad|(-\Delta)^{\frac{\alpha}{2}}\nabla^d u_n(t,x)|> 
\frac{C_{d,\alpha}}{\epsilon^{d+\alpha +1}}\}|\leq|\Omega_{n,\epsilon,t}^C|
\leq \epsilon^4\cdot \tilde{g}_n(t).
\end{equation*} 

\noindent Define  $p=4/(d+\alpha+1)$. Like we did for case $\alpha=0$, we obtain

\begin{equation*}\begin{split} 
 \|(-\Delta)^{\frac{\alpha}{2}}\nabla^d u_n\|^p_{L^{p,\infty}
(t_0,\infty;L^{p,\infty}(K))}
 &\leq C\Big(\|u_0\|^2_{L^{2}(\mathbb{R}^3)} + 
\frac{|K|}{t_0}\Big)
\end{split}\end{equation*} 
for any integer $n,d\geq1$, for any real $\alpha\in(0,2)$, for
any bounded open subset $K$ of $\mathbb{R}^3$ and for any $t_0\in(0,\infty)$ where $C$ depends only on $d$ and $\alpha$.\\

If we restrict further $(d+\alpha)<3$, then  
 $ p = \frac{4}{d+\alpha+1}>1 $. This implies 
$(-\Delta)^{\alpha/2}\nabla^d u_n\in L^q_{loc}((t_0,\infty)\times K)$ 
for every $q$ between $1$ and $p$, 
 and the norm is uniformly bounded in $n$. 
 Thus, 
from weak-compactness of $L^q$ for $q>1$, 
 we conclude that 
 if $u$ is a suitable weak solution obtained by a limiting argument
 of $u_n$, then any higher derivatives
 $(-\Delta)^{\alpha/2}\nabla^d u$,
 which is defined in the remark \ref{frac_rem},
lie in  $L^1_{loc}$ %((t_0,\infty)\times K)$
as long as $(d+\alpha)<3$ with the same estimate
\begin{equation}\begin{split}\label{final_comment} 
 \|(-\Delta)^{\frac{\alpha}{2}}\nabla^d u\|^p_{L^{p,\infty}
(t_0,\infty;L^{p,\infty}(K))}
 &\leq C_{d,\alpha}\Big(\|u_0\|^2_{L^{2}(\mathbb{R}^3)} + 
\frac{|K|}{t_0}\Big).
\end{split}\end{equation} 
\end{proof}

\subsection{Proof of theorem \ref{main_thm} part (I)}\label{proof_main_thm_I}
%Part (I) is about smooth solutions of the Navier-Stokes equations. 

\begin{proof}[Proof of proposition \ref{main_thm} part (I)]

Suppose that $(u,P)$ is a smooth solution of the Navier-Stokes equations
\eqref{navier}
on $(0,T)$ with \eqref{initial_condition}. Then it satisfies all conditions of 
{(Problem I-n)} for $n=\infty$ on $(0,T)$. 
As we mentioned at the remark \ref{n_infty_rem}, we  follow
every steps in the subsections 
 \ref{prof_main_thm_II_alpha_0}
and \ref{prof_main_thm_II_alpha_not_0} except each last arguments
which impose $d<3$ or $(d+\alpha)<3$.
%which require the weak-compactness of $L^p$ for $p>1$. 
Indeed,
under the scaling \eqref{special designed scaling}, 
the resulting function $(v,Q)$ is a solution for (Problem II-r)
for $r=0$.\\
%Here our time interval is  not $(0,\infty)$ but $(0,T)$.\\

\noindent 
Recall that $u$ is smooth by assumption.
As a result, we do NOT have
%and for this time we do not have 
any restriction like $d<3$ or $(d+\alpha)<3$ at this time because
we do not need any limiting argument any more 
which requires a weak-compactness.
Thus, we obtain \eqref{final_comment} 
 for any integer $d\geq1$,
for any real $\alpha\in[0,2)$
and for any $t_0\in(0,T)$.
It finishes 
 the proof of the part $(I)$ of the main theorem \ref{main_thm}.

\end{proof}

\section*{A. Appendix: proofs of some technical lemmas}\label{appendix}
\begin{proof} [proof for lemma \ref{higher_pressure}]
Fix $(n,a,b,p)$ such that $n\geq 2$, $0<b<a<1$ and $1<p<\infty$.
Let $\alpha$ be any multi index such that
$|\alpha|=n$ and $D^{\alpha}=\partial_{\alpha_1} \partial_{\alpha_2}
D^{\beta}$ where $\beta$ is a multi index with  $|\beta|=n-2$.\\

Observe that from $\ebdiv(v_2)=0$ and $\ebdiv(v_1)=0$,
\begin{equation*}\begin{split}
-\Delta(D^{\alpha}P) &=\ebdiv\ebdiv D^{\alpha}(v_2\otimes v_1)\\
%&=\sum_{ij}\partial_i \partial_j D^{\alpha}(v_{2,i}v_{1,j})\\
 &=D^{\alpha}\Big(\sum_{ij} (\partial_j v_{2,i})(\partial_iv_{1,j})\Big)\\
%&=\partial_{\alpha_1} \partial_{\alpha_2}
%D^{\beta}\Big(\sum_{ij} (\partial_j v_{2,i})(\partial_iv_{1,j})\Big)\\
&=\partial_{\alpha_1} \partial_{\alpha_2}
H\\
\end{split}\end{equation*} where $H=D^{\beta}\Big(\sum_{ij} 
(\partial_j v_{2,i})(\partial_iv_{1,j})\Big)$
and  $v_k=(v_{k,1},v_{k,2},v_{k,3})$ for $k=1,2$.\\ 
% and $v_2=(v_{2,1},v_{2,2},v_{2,3})$..\\
%
\noindent Then for any  $(p_1,p_2)$ such that $\frac{1}{p}=\frac{1}{p_1}+\frac{1}{p_2}$
\begin{equation*}\begin{split} 
\|H\|_{L^{p}(B(a))}
&\leq C 
\|  v_2 \|_{W^{n-1,p_2}(B(a))}\cdot
\|  v_1 \|_{W^{n-1,p_1}(B(a))}
\end{split}\end{equation*} %where $\frac{1}{p}=\frac{1}{p_1}+\frac{1}{p_2}$
where $C$ is independent of choice of $p_1$ and $p_2$
and 
\begin{equation*}\begin{split} 
\|H\|_{W^{1,{\infty}}(B(a))}&
\leq C
\|  v_2 \|_{W^{n,\infty}(B(a))}\cdot
\|  v_1 \|_{W^{n,\infty}(B(a))}.
\end{split}\end{equation*}

Fix a function $\psi \in C^{\infty}(\mathbb{R}^3)$ satisfying:
\begin{equation*}\begin{split}
&\psi = 1 \quad\mbox{ in } B(b+\frac{a-b}{3}),\quad
\psi = 0 \quad\mbox{ in } (B(b+\frac{2(a-b)}{3}))^C \mbox{ and }
0\leq \psi \leq 1.
\end{split}\end{equation*}

\noindent We decompose $D^{\alpha}P$ by using $\psi$:
\begin{equation*}\begin{split}   
-\Delta({\psi} D^{\alpha}P) 
&= -{\psi} \Delta D^{\alpha}P 
- 2\ebdiv((\nabla {\psi})(D^{\alpha}P)) +
(D^{\alpha} P)\Delta{\psi}\\
 &= {\psi}\partial_{\alpha_1} \partial_{\alpha_2}H 
- 2\ebdiv((\nabla {\psi})(D^{\alpha}P)) +
(D^{\alpha} P)\Delta{\psi}\\
% &= \partial_{\alpha_1} \partial_{\alpha_2} ({\psi}H) \\
%&\quad\quad- \partial_{\alpha_2}[(\partial_{\alpha_1} {\psi})(H)]
%  -\partial_{\alpha_1}[(\partial_{\alpha_2} {\psi})(H)] 
%+ (\partial_{\alpha_1} \partial_{\alpha_2} {\psi})(H)\\
%& \quad\quad- 2\ebdiv((\nabla {\psi})(D^{\alpha}P)) +
%(D^{\alpha} P)\Delta{\psi}\\
&= - \Delta Q_{1} - \Delta Q_{2} - \Delta Q_{3}\\
\end{split}
\end{equation*} where
\begin{equation*}\begin{split} 
 - \Delta Q_{1} &= \partial_{\alpha_1} \partial_{\alpha_2} ({\psi}H),  \\
- \Delta Q_{2} &=  - \partial_{\alpha_2}[(\partial_{\alpha_1} {\psi})(H)]
  -\partial_{\alpha_1}[(\partial_{\alpha_2} {\psi})(H)] 
+ (\partial_{\alpha_1} \partial_{\alpha_2} {\psi})(H) \quad\mbox{ and} \\
 - \Delta Q_{3}\mbox{ } &= - 2\ebdiv((\nabla {\psi})(D^{\alpha}P)) +
(D^{\alpha} P)\Delta{\psi}.
\end{split}\end{equation*}
 Here $Q_{2}$ and $Q_{3}$ are defined by the representation formula 
${(-\Delta)}^{-1}(f) = \frac{1}{4\pi}(\frac{1}{|x|} * f)$\\
while $Q_{1}$ by the Riesz transforms.\\

\noindent Then, by the Riesz transform, %for $1<p<\infty$,
\begin{equation*}\begin{split} 
\|Q_{1}\|_{L^{p}(B(b))}
&\leq 
C\|\psi H\|_{L^{p}(\mathbb{R}^3)}
\leq C
\|H\|_{L^{p}(B(a))}\\
&\leq C
\|  v_2 \|_{W^{n-1,p_2}(B(a))}\cdot
\|  v_1 \|_{W^{n-1,p_1}(B(a))}.\\
\end{split}\end{equation*}

%Likewise, using $C^{\alpha}$ boundedness of Riesz transform for $\alpha>0$,
%\begin{equation*}\begin{split} 
%\|Q_{1}\|_{C^{\alpha}(B(b))}
%&\leq 
%C\|\psi H\|_{C^{\alpha}(\mathbb{R}^3)}
%\leq C
%\|H\|_{W^{1,{\infty}}(B(a))}\\
%&\leq C
%\|  v_2 \|_{W^{n,\infty}(B(a))}\cdot
%\|  v_1 \|_{W^{n,\infty}(B(a))}.\\
%\end{split}\end{equation*}

\noindent Moreover, using Sobolev,
\begin{equation*}\begin{split} 
\|Q_{1}\|_{L^\infty(B(b))}
&\leq 
C\Big(\|Q_{1}\|_{L^4(B(b))}+\|\nabla Q_{1}\|_{L^4(B(b))}\Big)\\
&\leq C
\|H\|_{W^{1,{4}}(B(a))}
\leq C
\|H\|_{W^{1,{\infty}}(B(a))}\\
&\leq C
\|  v_2 \|_{W^{n,\infty}(B(a))}\cdot
\|  v_1 \|_{W^{n,\infty}(B(a))}.\\
\end{split}\end{equation*}

\noindent For $x\in B(b)$,
\begin{equation*}\begin{split}
|Q_2(x)| 
%&= \Bigg|\frac{1}{4\pi}\int_{\mathbb{R}^3}
%\frac{1}{|x-y|} \Big(  \partial_{\alpha_2}[(\partial_{\alpha_1} {\psi})(H)](y)\\
% &\quad\quad\quad\quad -\partial_{\alpha_1}[(\partial_{\alpha_2} {\psi})(H)](y) 
%+ (\partial_{\alpha_1} \partial_{\alpha_2} {\psi})(H)(y)\Big)dy\Bigg|\\
&= \Bigg|\frac{1}{4\pi}\int_{(B(b+\frac{2(a-b)}{3})-B(b+\frac{a-b}{3}))}
\frac{1}{|x-y|} \Big(  \partial_{\alpha_2}[(\partial_{\alpha_1} {\psi})(H)](y)\\
 &\quad\quad\quad\quad -\partial_{\alpha_1}[(\partial_{\alpha_2} {\psi})(H)](y) 
+ (\partial_{\alpha_1} \partial_{\alpha_2} {\psi})(H)(y)\Big)dy\Bigg|\\
&\leq 2\|\nabla \psi\|_{L^\infty}\cdot
\sup_{y\in B(b+\frac{a-b}{3})^C}(|\nabla_y\frac{1}{|x-y|}|) 
\cdot \|H\|_{L^1(B(a))}\\
&\quad\quad+ \|\nabla^2 \psi\|_{L^\infty}\cdot
\sup_{y\in B(b+\frac{a-b}{3})^C}(|\frac{1}{|x-y|}|) 
\cdot \|H\|_{L^1(B(a))}\\
&\leq C\cdot \|H\|_{L^1(B(a))}\\
\end{split}\end{equation*} because $|x-y|\geq (a-b)/3$.
Likewise, for $x\in B(b)$,
\begin{equation*}\begin{split}
|Q_3(x)| 
%= \Bigg|\frac{1}{4\pi}\int_{(B(a)-B(b+\frac{a-b}{2}))}
%\frac{1}{|x-y|} \Big( - 2\ebdiv((\nabla {\psi})(D^{\alpha}P)) +
%(D^{\alpha} P)\Delta{\psi}\Big)dy\Bigg|\\
&\leq C\Big(\sum_{k=0}^{n}\|\nabla^{k+1} \psi\|_{L^\infty}\Big)\cdot
\Big(\sum_{k=0}^{n}\sup_{y\in B(b+\frac{a-b}{3})^C}
|\nabla^{k+1}_y\frac{1}{|x-y|}|\Big) 
\cdot \|P\|_{L^1(B(a))}\\
&+ C\Big(\sum_{k=0}^{n}\|\nabla^{k+2} \psi\|_{L^\infty}\Big)\cdot
\Big(\sum_{k=0}^{n}\sup_{y\in B(b+\frac{a-b}{3})^C}
|\nabla^{k}_y\frac{1}{|x-y|}|\Big) 
\cdot \|P\|_{L^1(B(a))}\\
&\leq C\cdot \|P\|_{L^1(B(a))}.
\end{split}\end{equation*} 

\noindent Finally,% for $1<p<\infty$,
\begin{equation*}\begin{split} 
\|\nabla^n &P\|_{L^{p}(B(b))}
\leq \|Q_1\|_{L^{p}(B(b))}+
C\||Q_2|+|Q_3|\|_{L^{\infty}(B(b))}\\
&\leq  C\cdot
\|H\|_{L^p(B(a))} +
C\cdot \|H\|_{L^1(B(a))}
+C\cdot \|P\|_{L^1(B(a))}\\
%&\leq 
%C_{a,b,p,n}\cdot \|H\|_{L^p(B(a))}
%+C_{a,b,n}\cdot \|P\|_{L^1(B(a))}\\
&\leq  C_{a,b,p,n}\Big(
\|  v_2 \|_{W^{n-1,p_2}(B(a))}\cdot
\|  v_1 \|_{W^{n-1,p_1}(B(a))} 
+\cdot \|P\|_{L^1(B(a))}\Big)
\end{split}\end{equation*} and
\begin{equation*}\begin{split} 
\|\nabla^n &P\|_{L^{\infty}(B(b))}
\leq \||Q_1|+|Q_2|+|Q_3|\|_{L^{\infty}(B(b))}\\
&\leq  C\cdot
\|H\|_{W^{1,\infty}(B(a))} +
C\cdot \|H\|_{L^1(B(a))}
+C\cdot \|P\|_{L^1(B(a))}\\
%&\leq 
%_{a,b,n}\cdot \|H\|_{W^{1,\infty}(B(a))}
%+C_{a,b,n}\cdot \|P\|_{L^1(B(a))}\\
&\leq  C_{a,b,n}\Big(
\|  v_2 \|_{W^{n,\infty}(B(a))}\cdot
\|  v_1 \|_{W^{n,\infty}(B(a))} 
+ \|P\|_{L^1(B(a))}\Big).
\end{split}\end{equation*}
\end{proof}

\begin{proof} [proof for lemma \ref{lem_a_half_upgrading_large_r}]
We fix $(n,a,b)$ such that $n\geq 0$ and $0<b<a<1$
 and let $\alpha$ be a multi index with $|\alpha|=n$. Then,
by taking $D^{\alpha}$ to \eqref{navier_Problem II-r}, we have
\begin{equation}
 \begin{split}\label{eq_d_alpha_large_r}
0
%=&\partial_t(D^{\alpha}{v_1})+D^{\alpha}\Big(({v_2}
%\cdot\nabla){v_1}\Big)+ \nabla(D^{\alpha} P) -\Delta(D^{\alpha} {v_1})\\
%=&\partial_t(D^{\alpha}{v_1})+\sum_{\beta\leq\alpha}\binom{\alpha}{\beta}
%((D^{\beta}{v_2})
%\cdot\nabla)(D^{\alpha-\beta}{v_1})+ \nabla(D^{\alpha} P) -\Delta(D^{\alpha} {v_1})\\
=&\partial_t(D^{\alpha}{v_1})+\sum_{\beta\leq\alpha, |\beta|>0}\binom{\alpha}{\beta}
((D^{\beta}{v_2})
\cdot\nabla)(D^{\alpha-\beta}{v_1} )+({v_2} 
\cdot\nabla)(D^{\alpha}{v_1} )\\&
\quad\quad\quad\quad\quad
\quad\quad\quad\quad\quad
+ \nabla(D^{\alpha} P) -\Delta(D^{\alpha} {v_1} ).
\end{split}
\end{equation} 

\noindent We define $\Phi(t,x)\in C^\infty$ by
$0\leq\Phi\leq 1, 
\Phi = 1 \mbox{ on }   Q_{{b}} \mbox{ and } 
\Phi = 0 \mbox{ on }   Q_{{a}}^C.$
%\begin{equation*}\begin{split} 
%&0\leq\Phi\leq 1, \quad
%\Phi = 1 \mbox{ on }   Q_{{b}}\quad \mbox{ and } \quad
%\Phi = 0 \mbox{ on }   Q_{{a}}^C.
%%&|\nabla \Phi| + | \nabla^2 \Phi | + 
%%|\partial_t \Phi|\leq C=C(a,b).\\
%  \end{split}
%\end{equation*}

%For second part of this lemma,
We observe that, for $p\geq \frac{1}{2}$ and for $f\in C^\infty$,
 \begin{equation*}\begin{split} 
&({p}+\frac{1}{2})|f|^{{p}-\frac{3}{2}}f\cdot\partial_{x} f = \partial_{x}|f|^{{p}+\frac{1}{2}} 
\mbox{ and }  ({p}+\frac{1}{2})|f|^{{p}-\frac{3}{2}}f\cdot\Delta f 
\leq \Delta(|f|^{{p}+\frac{1}{2}}) .
\end{split}\end{equation*} which can be verified by direct computations
with the fact $|\nabla f|\geq |\nabla|f||$.\\

\noindent Now
% fix any $p\geq 1$,
 we multiply $(p+\frac{1}{2}){\Phi}\frac{D^{\alpha}v_1}
{|D^{\alpha}v_1|^{(3/2)-p}}$ to \eqref{eq_d_alpha_large_r},
 use the above observation
%lemma \ref{lem_a_half_upgrading} 
 and integrate in x. Then we have for any $p\geq \frac{1}{2}$,\\
\begin{equation*}\begin{split} 
&\frac{d}{dt}\int_{\mathbb{R}^3}{\Phi}(t,x)|D^{\alpha}{v_1} (t,x)|^{p+\frac{1}{2}}dx\\
&\leq\int_{\mathbb{R}^3}(|\partial_t{\Phi}(t,x)|+|\Delta{\Phi}(t,x)|)
|D^{\alpha}{v_1} (t,x)|^{p+\frac{1}{2}}dx \\
&\quad+(p+\frac{1}{2})\int_{\mathbb{R}^3}|\nabla D^{\alpha}P(t,x)||D^{\alpha}{v_1} (t,x)|^{p-\frac{1}{2}}dx \\
&\quad+(p+\frac{1}{2})\sum_{\beta\leq\alpha, |\beta|>0}\binom{\alpha}{\beta}
\int_{\mathbb{R}^3}\Big|(D^{\beta}{v_2}(t,x)
\cdot\nabla)D^{\alpha-\beta}{v_1} (t,x)\Big||D^{\alpha}{v_1} (t,x)|^{p-\frac{1}{2}}dx \\
&\quad\quad -\int_{\mathbb{R}^3}{\Phi}(t,x)(v_2(t,x)
\cdot\nabla)(|D^{\alpha}{v_1} (t,x)|^{p+\frac{1}{2}})dx \\
 \end{split}
\end{equation*}
\begin{equation*}\begin{split} 
&\leq C\||\nabla^n {v_1} (t,\cdot)|^{p+\frac{1}{2}}\|_{L^{1}(B{(a)})} \\
&\quad+C\|\nabla^{n+1} P(t,\cdot)\|_{L^{2p}(B{(a)})} 
\cdot\||\nabla^n {v_1} (t,\cdot)|^{p-\frac{1}{2}}\|_{L^{\frac{2p}{2p-1}}(B{(a)})} \\
&+C \| {v_2} (t,\cdot)\|_{W^{n,\infty}(B{(a)})}\cdot
\|{v_1} (t,\cdot)\|_{W^{n,p+\frac{1}{2}}(B{(a)})}\cdot
\||\nabla^n {v_1} (t,\cdot)|^{p-\frac{1}{2}}\|_{L^{\frac{p+\frac{1}{2}}{p-\frac{1}{2}}}(B{(a)})} \\
&-\int_{\mathbb{R}^3}{\Phi}(t,x)\ebdiv\Big({v_2} (t,x)
\otimes|D^{\alpha}{v_1} (t,x)|^{p+\frac{1}{2}}\Big)dx \\
 \end{split}
\end{equation*}
\begin{equation*}\begin{split} 
&\leq C\|{v_1} (t,\cdot)\|^{p+\frac{1}{2}}_{W^{n,p+\frac{1}{2}}(B{(a)})}\\
&\quad+C\|\nabla^{n+1} P(t,\cdot)\|_{L^{2p}(B{(a)})} 
\cdot\|\nabla^n {v_1} (t,\cdot)\|^{p-\frac{1}{2}}_{L^{p}(B{(a)})} \\
&\quad+C \| {v_2} (t,\cdot)\|_{W^{n,\infty}(B{(a)})}
\cdot\|{v_1} (t,\cdot)\|^{p+\frac{1}{2}}_{W^{n,p+\frac{1}{2}}(B{(a)})} \\
&\quad+C\| {v_2} (t,\cdot)\|_{L^{\infty}(B{(a)})}\cdot
\|\nabla^n {v_1} (t,\cdot)\|^{p+\frac{1}{2}}_{L^{p+\frac{1}{2}}(B{(a)})}.\quad\quad\quad\quad\quad\quad\quad\quad\quad\quad
 \end{split}
\end{equation*}

\noindent Then integrating on $[-a^2,t]$ for  any $t\in[-b^2,0]$ gives 
 \begin{equation*}\begin{split}  
&\|D^{\alpha} {v_1}  \|^{p+\frac{1}{2}}_{L^{\infty}(-{({b})}^2,0  ;L^{p+\frac{1}{2}}(B{({b})}))}\\
&\leq C\|{v_1} \|^{p+\frac{1}{2}}
_{L^{p+\frac{1}{2}}(-{({a})}^2,0;W^{n,p+\frac{1}{2}}(B{({a})}))}\\
&\quad+C\|\nabla^{n+1} P\|_{L^{1}(-{({a})}^2,0;L^{2p}(B{({a})}))}\cdot
\|\nabla^n {v_1} \|^{p-\frac{1}{2}}_{L^{\infty}(-{({a})}^2,0;L^{p}(B{(a)}))}\\ 
&\quad+C \| {v_2} \|_{L^{2}(-{({a})}^2,0;W^{n,\infty}(B{({a})}))}
\cdot\|{v_1} \|^{p+\frac{1}{2}}_{L^{2p+1}(-{({a})}^2,0;W^{n,p+\frac{1}{2}}(B{({a})}))} \\
&\quad+C\| {v_2} \|_{L^{2}(-{({a})}^2,0;L^{\infty}(B{(a)}))}\cdot
\|\nabla^n {v_1} \|^{p+\frac{1}{2}}_{L^{2p+1}(-{({a})}^2,0;L^{p+\frac{1}{2}}(B{(a)}))}\\
%&\leq C\|{v_1} \|^{p+\frac{1}{2}}
%_{L^{p+\frac{1}{2}}(-{({a})}^2,0;W^{n,p+\frac{1}{2}}(B{({a})}))}\\
%&\quad+C_p\|\nabla^{n+1} P\|_{L^{1}(-{({a})}^2,0;L^{2p}(B{({a})}))}\cdot
%\|\nabla^n {v_1} \|^{p-\frac{1}{2}}_{L^{\infty}(-{({a})}^2,0;L^{p}(B{(a)}))}\\ 
%&\quad+C_{p} 
%\|{v_1} \|^{p+\frac{1}{2}}_{L^{2p+1}(-{({a})}^2,0;W^{n,p+\frac{1}{2}}(B{({a})}))} \\
%%&\quad+C
%%\|\nabla^n {v_1} \|^{p+\frac{1}{2}}
%%_{L^{2p+1}(-{({a})}^2,0;L^{p+\frac{1}{2}}(B{(a)}))}\\
 \end{split}
\end{equation*}
\noindent Thus for the case $p=1/2$, we have
\begin{equation*}\begin{split}  
&\|D^{\alpha} {v_1}  \|_{L^{\infty}(-{({b})}^2,0  ;L^{1}(B{({b})}))}\\
%&\leq C\|{v_1} \|_{L^{1}(-{a}^2,0;W^{n,1}(B{(a)}))}
%+\|\nabla^{n+1} P\|_{L^{1}(-{a}^2,0;L^{1}(B{(a)}))} \\
%&\quad+C \| {v_2} \|_{L^{2}(-{a}^2,0;W^{n,\infty}(B{(a)}))}
%\cdot\|{v_1} \|_{L^{2}(-{a}^2,0;W^{n,1}(B{(a)}))} \\
%&\quad+C\| {v_2} \|_{L^{2}(-{a}^2,0;L^{\infty}(B{(a)}))}\cdot
%\|\nabla^n {v_1} \|_{L^{2}(-{a}^2,0;L^{1}(B{(a)}))}\\
&\leq C\Big[
\Big(\| v_2\|_{L^2(-{{a}^2},0;W^{n,\infty}(B({a})))} +1\Big)
\cdot
  \| {v_1} \|_{L^{2}(-{a}^2,0;W^{n,{1}}(B{(a)}))}\\
&\quad\quad\quad\quad\quad\quad\quad\quad+ \|\nabla^{n+1}P 
\|_{L^{1}(-{a}^2,0;L^{1}(B{(a)}))}  \Big]
\end{split}\end{equation*} 
while, for the case $p\geq 1$, we have
\begin{equation*}\begin{split}
&\|D^{\alpha} {v_1}  \|^{p+\frac{1}{2}}_{L^{\infty}(-{({b})}^2,0  ;L^{p+\frac{1}{2}}(B{({b})}))}\\
%&\leq C
%\Big(\| v_2\|_{L^2(-{{a}^2},0;W^{n,\infty}(B({a})))} +1\Big)
%\cdot
%\|{v_1} \|^{p+\frac{1}{2}}
%_{L^{2p+1}(-{({a})}^2,0;W^{n,p+\frac{1}{2}}(B{({a})}))}\\
%&\quad+C\|\nabla^{n+1} P\|_{L^{1}(-{({a})}^2,0;L^{2p}(B{({a})}))}\cdot
%\| {v_1} \|^{p-\frac{1}{2}}_{L^{\infty}(-{({a})}^2,0;W^{n,p}(B{(a)}))}\\ 
%%&\leq C_{p}\Big(\|{v_1} \|^{\frac{2}{2p+1}}
%%_{L^{2}(-{({a})}^2,0;W^{n,2p}(B{({a})}))}
%%\cdot\|{v_1} \|^{1-\frac{2}{2p+1}}_{L^{\infty}(-{({a})}^2,0;W^{n,p}(B{({a})}))}\Big)^{p+\frac{1}{2}}\\
%%&\quad+C_p\|\nabla^{n+1} P\|^2_{L^{1}(-{({a})}^2,0;L^{2p}(B{({a})}))}+
%%C_p\| {v_1} \|^{2p-1}_{L^{\infty}(-{({a})}^2,0;W^{n,p}(B{(a)}))}\\
&\leq C
\Big(\| v_2\|_{L^2(-{{a}^2},0;W^{n,\infty}(B({a})))} +1\Big)
\\&\quad\quad\quad\cdot
\Big(%(\frac{2}{2p+1})
\|{v_1} \|^{\frac{1}{p+\frac{1}{2}}}
_{L^{2}(-{({a})}^2,0;W^{n,2p}(B{({a})}))}
\cdot
%(1-\frac{2}{2p+1})
\|{v_1} \|^{1-\frac{1}{p+\frac{1}{2}}}_{L^{\infty}(-{({a})}^2,0;W^{n,p}(B{({a})}))}\Big)^{p+\frac{1}{2}}\\
&\quad\quad\quad+C\|\nabla^{n+1} P\|_{L^{1}(-{({a})}^2,0;L^{2p}(B{({a})}))}\cdot
\| {v_1} \|^{p-\frac{1}{2}}_{L^{\infty}(-{({a})}^2,0;W^{n,p}(B{(a)}))}\\ 
&\leq C_{a,b,n,p}\Big[%(\frac{2}{2p+1})
\Big(\| v_2\|_{L^2(-{{a}^2},0;W^{n,\infty}(B({a})))} +1\Big)
\cdot
\|{v_1} \|_{L^{2}(-{({a})}^2,0;W^{n,2p}(B{({a})}))}\\
&\quad\quad\quad\quad\quad\quad+ \|\nabla^{n+1} P\|_{L^{1}(-{({a})}^2,0;L^{2p}(B{({a})}))}
\Big]\cdot%\\ &\quad\quad\quad\quad\cdot
\| {v_1} \|^{p-\frac{1}{2}}_{L^{\infty}(-{({a})}^2,0;W^{n,p}(B{(a)}))}.
%&\leq C_{p}\Big(%(\frac{2}{2p+1})
%\|{v_1} \|_{L^{2}(-{({a})}^2,0;W^{n,2p}(B{({a})}))}
%\cdot
%%(1-\frac{2}{2p+1})
%\|{v_1} \|^{{p-\frac{1}{2}}}_{L^{\infty}(-{({a})}^2,0;W^{n,p}(B{({a})}))}\Big)\\
%&\quad+C_p\|\nabla^{n+1} P\|_{L^{1}(-{({a})}^2,0;L^{2p}(B{({a})}))}\cdot
%\| {v_1} \|^{p-\frac{1}{2}}_{L^{\infty}(-{({a})}^2,0;W^{n,p}(B{(a)}))}\\ 
%&\leq C_{p}\Big(%(\frac{2}{2p+1})
%\|{v_1} \|_{L^{2}(-{({a})}^2,0;W^{n,2p}(B{({a})}))}+
%\|\nabla^{n+1} P\|_{L^{1}(-{({a})}^2,0;L^{2p}(B{({a})}))}
%\Big)\\ &\quad\quad\quad\quad\cdot
%\| {v_1} \|^{p-\frac{1}{2}}_{L^{\infty}(-{({a})}^2,0;W^{n,p}(B{(a)}))}\\ 
%%&\leq C_{p}\Big(%(\frac{2}{2p+1})
%%\|{v_1} \|
%%_{L^{2}(-{({a})}^2,0;W^{n,2p}(B{({a})}))}
%%+
%%(1-\frac{2}{2p+1})
%%\|{v_1} \|_{L^{\infty}(-{({a})}^2,0;W^{n,p}(B{({a})}))}\Big)^{p+\frac{1}{2}}\\
%%&\quad+C_p\|\nabla^{n+1} P\|^2_{L^{1}(-{({a})}^2,0;L^{2p}(B{({a})}))}+
%%C_p\| {v_1} \|^{2p-1}_{L^{\infty}(-{({a})}^2,0;W^{n,p}(B{(a)}))}\\ 
%%&\leq C_{p}\Big(1+\|{v_1} \|
%%_{L^{2}(-{({a})}^2,0;W^{n,2p}(B{({a})}))}
%% + 
%%\|{v_1} \|_{L^{\infty}(-{({a})}^2,0;W^{n,p}(B{({a})}))}\\
%%&\quad+\|\nabla^{n+1} P\|_{L^{1}(-{({a})}^2,0;L^{2p}(B{({a})}))}\Big)^{2p}
 \end{split}
\end{equation*}% It proved the second part of this lemma.

%where we used $x^{p_1}\leq 1+x^{p_2}$ 
%for $p_1\leq p_2$ and $x^{\theta}y^{1-\theta}\leq {\theta}x+{(1-\theta)}y$.\\ 

%It ends the proof of second part of this lemma: for $p\geq1$
% \begin{equation}\begin{split}  
%\|\nabla^{n} {v_1}  \|_{L^{\infty}(-{({b})}^2,0  ;L^{p+\frac{1}{2}}(B{({b})}))}\\
%\leq C\Big(1+\|{v_1} \|
%_{L^{2}(-{({a})}^2,0;W^{n,2p}(B{({a})}))}\\
%&\quad + 
%\|{v_1} \|_{L^{\infty}(-{({a})}^2,0;W^{n,p}(B{({a})}))}\\
%&\quad+\|\nabla^{n+1} P\|_{L^{1}(-{({a})}^2,0;L^{2p}(B{({a})}))}\Big)^{2}\\
%\end{split}\end{equation} where $C$ depends on $a,b,n$ and $p$. \\
%\\

\end{proof}

\begin{proof}[proof for lemma \ref{lem_Maximal 2.5 or 4}]
Fix any $M_0>0$ and $1\leq p<\infty$ first. Then, for any $M\geq M_0$ and 
 for any $f\in C^1(\mathbb{R}^3)$
such that $ \int_{\mathbb{R}^3}\phi(x)f(x)dx=0$, 
 we have
 \begin{equation*}\begin{split}
&\|f\|_{L^p(B(M))}=
%\|f-\int_{\mathbb{R}^3}f(y)\phi(y)dy\|_{L^p(B(M))}\\
%&= 
\Big(\int_{B(M)}\Big|\int_{\mathbb{R}^3}(f(x)-f(y))
\phi(y)dy\Big|^{p}dx\Big)^{1/p}\\
&\leq C\Big(\int_{B(M)}\Big(\int_{B(1)}\Big|f(x)-f(y)\Big|
dy\Big)^{p}dx\Big)^{1/p}\\
&\leq C\Big(\int_{B(M)}\Big(\int_{B(1)} 
\int_0^1\Big|(\nabla f)((1-t)x+ty)\cdot(x-y)\Big|dt
dy\Big)^{p}dx\Big)^{1/p}
\end{split}\end{equation*}
 \begin{equation*}\begin{split}
&\leq C(M+1)\Big(\int_{B(M)}\Big(\int_{B(1)} 
\int_0^1\Big|(\nabla f)((1-t)x+ty)\Big|dt
dy\Big)^{p}dx\Big)^{1/p}\\
&\leq C(M+1)\Big(\int_{B(M)}\Big(\int_{B(1)} 
\int_0^{\frac{M}{M+1}}\Big|(\nabla f)((1-t)x+ty)\Big|dt
dy\Big)^{p}dx\Big)^{1/p}\\
&\quad + C(M+1)\Big(\int_{B(M)}\Big(\int_{B(1)} 
\int_{\frac{M}{M+1}}^1\Big|(\nabla f)((1-t)x+ty)\Big|dt
dy\Big)^{p}dx\Big)^{1/p}\\
&=(I)+(II)
\end{split}\end{equation*} where we used $x\in B(M)$ and $y\in B(1)$.\\

\noindent For $(I)$,
 \begin{equation*}\begin{split}
&(I)
\leq C_{M_0}\Big(\int_{B(1)} 
\int_0^{\frac{M}{M+1}}\Big(\int_{B(M)}
\Big|(\nabla f)((1-t)x+ty)\Big|^{p}dx\Big)^{1/p}dt
dy\Big)\\
%&\leq C_{M_0}\Big(\int_{B(1)} 
%\int_0^{\frac{M}{M+1}}\frac{1}{(1-t)^{3/p}}\Big(\int_{B((1-t)M+1)}
%\Big|(\nabla f)(z)\Big|^{p}dz\Big)^{1/p}dt
%dy\Big)\\
&\leq C_{M_0}\cdot M
\int_0^{\frac{M}{M+1}}\frac{1}{(1-t)^{3/p}}\Big(\int_{B((1-t)M+1)}
\Big|(\nabla f)(z)\Big|^{p}dz\Big)^{1/p}dt\\
%&\leq C_{M_0}\cdot M
%\int_0^{\frac{M}{M+1}}\frac{1}{(1-t)^{3/p}}\\
%&\quad\quad\quad\quad
%\Big(
%\int_{B(1)}\frac{((1-t)M+2)^3}{((1-t)M+2)^3}
%\int_{B((1-t)M+2)}
%\Big|(\nabla f)(z+u)\Big|^{p}dzdu\Big)^{1/p}dt\\
&\leq C_{M_0}\cdot M
\int_0^{\frac{M}{M+1}}\frac{1}{(1-t)^{3/p}}
\Big(
\int_{B(1)}
\int_{B((1-t)M+2)}
\Big|(\nabla f)(z+u)\Big|^{p}dzdu\Big)^{1/p}dt\\
&\leq C_{M_0}\cdot M
\int_0^{\frac{M}{M+1}}\frac{((1-t)M+2)^{3/p}}{(1-t)^{3/p}}\Big(
\int_{B(1)}
\mathcal{M}(|\nabla f|^{p})(u)du\Big)^{1/p}dt\\
&\leq C_{M_0,p}\cdot M \cdot\|\mathcal{M}(|\nabla f|^{p})\|^{1/p}_{L^1(B(1))}
\int_0^{\frac{M}{M+1}}\Big(M^{3/p}+\frac{1}{(1-t)^{3/p}}\Big)
dt
\\
&\leq C_{M_0,p}\cdot M \cdot\|\mathcal{M}(|\nabla f|^{p})\|^{1/p}_{L^1(B(1))}
\Big(M^{3/p}+\int_{\frac{1}{M+1}}^{1}\frac{1}{s^{3/p}}
ds\Big)
\\
&\leq C_{M_0}\cdot M \cdot\|\mathcal{M}(|\nabla f|^{p})\|^{1/p}_{L^1(B(1))}
\Big(M^{3/p}+{(M+1)}^{3/p}
\Big)
\\
&\leq C_{M_0,p}\cdot M^{1+\frac{3}{p}} 
\cdot\|\mathcal{M}(|\nabla f|^{p})\|^{1/p}_{L^1(B(1))}
\end{split}\end{equation*} where we used
an integral version of the Minkoski's inequality 
and   $(1+M)\leq C_{M_0}\cdot M$ from $M\geq M_0$
 for the first inequality.\\

\noindent For $(II)$, observe that if $\frac{M}{M+1}\leq t\leq 1$, then
$0\leq 1-t\leq \frac{1}{M+1}$ and
 \begin{equation*}\begin{split}
%|((1-t)x+ty)-y|=|1-t|\cdot|x-y|\leq \frac{1+M}{1+M}=1.
|(1-t)x+ty|\leq (1-t)\cdot|x|+t|y|\leq \frac{M}{M+1}+1\leq 2
\end{split}\end{equation*} because $x\in B(M)$ and $y\in B(1)$. Thus
 \begin{equation*}\begin{split}
(II)% =
%C(M+1)\Big(\int_{B(M)}\Big( 
%\int_{\frac{M}{M+1}}^1\int_{B(1)}\Big|(\nabla f)((1-t)x+ty)\Big|dy
%dt\Big)^{p}dx\Big)^{1/p}\\
& \leq
C_{M_0}\cdot M\Big(\int_{B(M)}\Big( 
\int_{\frac{M}{M+1}}^1\frac{1}{t^3}\int_{B(2)}\Big|(\nabla f)(z)\Big|dz
dt\Big)^{p}dx\Big)^{1/p}\\
& \leq
C_{M_0}\cdot M\cdot M^{3/p}\cdot
\int_{B(2)}|(\nabla f)(z)|dz
\cdot
\int_{\frac{M}{M+1}}^1\frac{1}{t^3}
dt\\
& \leq
C_{M_0} M^{1+\frac{3}{p}}\cdot
\|\nabla f\|_{L^1(B(2))}.
%& \leq
%C_{M_0} M^{1+\frac{3}{p}}\cdot
%\int_{B(1)}\frac{(2+1)^3}{(2+1)^3}\int_{B(2+1)}|(\nabla f)(z+u)|dzdu
%\cdot
%\int_{\frac{M_0}{M_0+1}}^1\frac{1}{t^3}
%dt\\
%& \leq
%C_{M_0} M^{1+\frac{3}{p}}\cdot
%\|\mathcal{M}(\nabla f)\|_{L^1(B(1))}
%\\
\end{split}\end{equation*}
\end{proof}

\begin{ack}
 The second author was partially supported by both the NSF 
and  the EPSRC Science and Innovation award to the Oxford Centre
for Nonlinear PDE (EP/E035027/1).
\end{ack}

\bibliographystyle{plain}
\bibliography{Fractional_NS_Choi_Vasseur}

\end{document}